\documentclass{amsart}
\usepackage{amsmath,enumitem, amssymb, graphicx, epsfig, verbatim, psfrag, color,amscd,stmaryrd} 
\usepackage{tikz}
\usetikzlibrary{arrows}

\def\endproof{\relax\ifmmode\expandafter\endproofmath\else
  \unskip\nobreak\hfil\penalty50\hskip.75em\hbox{}\nobreak\hfil\bull
  {\parfillskip=0pt \finalhyphendemerits=0 \bigbreak}\fi}
\def\endproofmath$${\eqno\bull$$\bigbreak}
\def\bull{\vbox{\hrule\hbox{\vrule\kern3pt\vbox{\kern6pt}\kern3pt\vrule}\hrule}}

\newcommand\One{\mathbf 1}
\newcommand\chiEmb{\chi_{\mathrm{emb}}}

\newcommand\Ngr{\mathbf{N}}

\newcommand\LiftPartPerm{\mathfrak P}

\newcommand\op{\mathrm{op}}
\newcommand\CDisk{\mathbb D}
\newcommand\goesto{\mapsto}
\newcommand\Ymod{\mathcal Y}
\newcommand\inotin[1]{\overline{#1}}
\newcommand\ModFlow{\mathcal M}
\newcommand\CFBA{\widehat{\mathrm{CFPC}}}
\newcommand\States{\mathbf S}
\newcommand\HD{\mathcal H}
\newcommand\rhos{\boldsymbol{\rho}}

\newcommand\Quot[2]{Q(#1,#2)}
\newcommand\QuotP[2]{Q'(#1,#2)}
\newcommand{\QuotPong}{\mathcal Q}
\newcommand\HC{\mathrm{HC}}
\newcommand\HH{\mathrm{HH}}

\newcommand\gr{\mathrm{gr}}
\newcommand\Pong[2]{{\mathcal P}(#1,#2)}
\newcommand\PongP[2]{{\mathcal P}'(#1,#2)}
\newcommand\OneHalf{\frac{1}{2}}
\newcommand\weight{\mathfrak w}
\newcommand\Totweight{\overline{\weight}}
\newcommand\Cr{\mathrm{cr}}

\newcommand\Xij{\Xgen{i}{j}}
\newcommand\Xim{\Xgen{i}{m}}
\newcommand\Xoj{\Xgen{0}{j}}
\newcommand\Rij{\Rgen{i}{j}}

\newcommand\Lji{\Lgen{j}{i}}

\newcommand\Cobar{\mathrm{Cobar}}
\newcommand\Alg\AlgA
\newcommand\Blg\AlgB

\newcommand\Ainf{\mathcal A}
\newcommand\Ainfty\Ainf

\newcommand\Clg{\mathcal C}
\newcommand\IdempRing[2]{{\mathcal I}(#1,#2)}
\newcommand\alphas{\mathbf \alpha}
\newcommand\betas{\mathbf \beta}

\newcommand\Zmod[1]{{\mathbb Z}/{#1}{\mathbb Z}}
\newtheorem{thm}{Theorem}[section]

\newtheorem{cor}[thm]{Corollary}

\newtheorem{lemma}[thm]{Lemma}
\newtheorem{prop}[thm]{Proposition}
\newtheorem{defn}[thm]{Definition}

\newtheorem{example}[thm]{Example}

\newtheorem{remark}[thm]{Remark}

\numberwithin{equation}{section}

\newcommand{\smargin}[1]{\marginpar{\tiny{#1}}}

\newcounter{bean}

\newcommand\Source{\mathcal S}
\newcommand\Idemp[1]{{\mathbf{I}}_{#1}}
\newcommand\DT{\boxtimes}
\newcommand\x{\mathbf x}
\newcommand\y{\mathbf y}

\newcommand\lsup[2]{^{#1}{#2}}
\newcommand\lsub[2]{{}_{#1}{#2}}

\newcommand{\AlgA}{{\mathcal A}}
\newcommand{\AlgB}{{\mathcal B}}

\newcommand\z{\mathbf z}

 \newcommand{\Z}{\mathbb Z}  \newcommand{\Q}{\mathbb Q} \newcommand{\R}{\mathbb R}

\newcommand\HFKa{\widehat{HFK}}
\newcommand\HFKm{{HFK}^-}

\newcommand\Xgen[2]{X_{{#1},{#2}}}
\newcommand\Lgen[2]{L_{{#1},{#2}}}
\newcommand\Rgen[2]{R_{{#1},{#2}}}

\newcommand\Lgenx[2]{{\mathcal L}_{{#1},{#2}}}
\newcommand\Rgenx[2]{{\mathcal R}_{{#1},{#2}}}

\newcommand\Field{\mathbb F}
\newcommand\Ground{\mathfrak k}
\newcommand\AAmod{\mathcal Y}
\newcommand\DDmod{\mathcal X}

\newcommand\inv{\mathrm{inv}}
\newcommand\cross{\mathrm{cross}}
\newcommand\Cross{\mathrm{Cross}}
\newcommand\LiftS{\widetilde S}
\newcommand\Liftf{\widetilde f}
\newcommand\Liftg{\widetilde g}
\newcommand\sgn{\mathrm{sgn}}

\DeclareMathOperator{\Hom}{Hom}

\DeclareMathOperator{\Id}{Id}

\begin{document}
\title{The pong algebra}

\begin{abstract}
  In an earlier paper, we described bordered algebras for knot Floer
  homology.  In this paper, we introduce closely related differential
  graded algebra, the {\em pong algebra}, and compute the $A_{\infty}$
  structure on its homology.
\end{abstract}

\author[Peter S. Ozsv\'ath]{Peter Ozsv\'ath}
\thanks {PSO was partially supported by NSF grant number DMS-1708284, DMS-2104536, and the Simons Grant {\em New structures in low-dimensional topology}.}
\address {Department of Mathematics, Princeton University\\ Princeton, New Jersey 08544} 
\email {petero@math.princeton.edu}

\author[Zolt{\'a}n Szab{\'o}]{Zolt{\'a}n Szab{\'o}}
\thanks{ZSz was supported by NSF grant number DMS-1904628
  and the Simons Grant {\em New structures in low-dimensional topology}.}
\address{Department of Mathematics, Princeton University\\ Princeton, New Jersey 08544}
\email {szabo@math.princeton.edu}

\maketitle
\section{Introduction}

Knot Floer homology~\cite{Knots,RasmussenThesis} is a knot invariant
defined using methods from symplectic geometry. This invariant has
many variants, the most general being a module over $\Z[u,v]$. Its
$u=v=0$ specialization, $\HFKa$, is a categorification of the
 Alexander polynomial.

In~\cite{BorderedKnots,Bordered2} we developed an algebraic invariant for knots,
inspired by Floer homology. This invariant is the homology of a chain
complex associated to a knot projection. The homology is a module over
the ring $\Z[u,v]/uv=0$. In~\cite{HolKnot}, we identified  these
constructions with analogues in knot Floer homology with coefficients mod $2$;
specifically, the $u=v=0$ specialization computes $\HFKa$ with
$\Zmod{2}$ coefficients while the $v=0$ specialization computes
$\HFKm$ over $\Zmod{2}[u]$.

The construction of the bordered knot invariants associates an
algebra to a
generic horizontal slice of the knot diagram and modules
over this algebra to the upper and lower diagrams. Specifically,
versions of the bordered algebras are indexed by pairs of integers
integers $m$ and $k$ with $k<m$. These {\em bordered algebras}
$\Clg(m,k)$ have idempotents $\Idemp{\x}$ corresponding to {\em idempotent states}
$\x$, which are increasing sequences of integers $1\leq
x_1<\dots<x_k\leq m-1$. 
These idempotents generate a subalgebra $\IdempRing{m}{k}\subset \Clg(m,k)$.
The algebra also contains distinguished
elements $U_1,\dots,U_m$, $L_2,\dots,L_{m-1}$, and $R_2,\dots,R_{m-1}$
satisfying various relations, including  $R_i R_{i+1}=0$ and $L_i L_{i-1}=0$.
(See Section~\ref{sec:Clg}.)

Our goal here is to enhance the $\Clg(m,k)$ with an eye towards
constructing a knot invariant over $\Z[u,v]$. We introduce
here a differential graded algebra, the {\em pong algebra}. Like the
bordered algebras, the pong algebras are indexed by pairs of integers
$m$ and $k$; we denote them $\Pong{m}{k}$.  The algebra $\Pong{m}{k}$ is
an algebra over the polynomial algebra over $\Field = \Zmod{2}$ 
in $m$ variables
$v_1,\dots,v_m$. The definitions of these algebras is combinatorial,
and closely connected to the strands algebras which have appeared in
bordered Floer homology~\cite{InvPair}.  These algebras are defined in
Section~\ref{sec:DefPong}. By construction, the pong algebra
$\Pong{m}{k}$ is a differential graded algebra, with a preferred
$\Z$-grading (which drops by one under the boundary operator), and an
additional $({1\over 2}\Z)^m$-grading, which is preserved by the differential.

The bordered algebra 
$\Clg(m,k)$ also has 
a {\em weight grading} with values in $({1\over 2}\Z)^m$, whose
$i^{th}$ component $\weight_i$ is specified  by
\begin{align*}
         \weight_i(L_j)=\weight_i(R_j)&=\left\{\begin{array}{ll}
\OneHalf & {\text{if $i =j$}} \\        
0 &{\text{otherwise}} 
\end{array}        
\right. \\                    
  \weight_i(U_j)&=\left\{\begin{array}{ll}
1 & {\text{if $i =j$}} \\       
0 &{\text{otherwise}}   
\end{array}        
\right.
\end{align*}
We extend this to a grading on $\Clg(m,k)[t]$
with the convention that
$\weight_i(t)=1$ for $i=1,\dots,m$.
We also introduce a further integer-valued grading $\gr$
on $\Clg(m,k)[t]$ so that 
\begin{align*}
        \gr(\Clg(m,k))&= 0 \\
        \gr(t)&=2m-2k-2.
\end{align*}

The main result here identifies the homology of
the pong algebra with the bordered algebras, and computes the
 induced $A_{\infty}$ structure on the homology, as follows:

\begin{thm}
        \label{thm:HomologyPongAinf}
        When $m>1$ and $0<k< m-1$, 
        the homology of $\Pong{m}{m-k-1}$ is isomorphic, as an algebra, to 
        $\Clg(m,k)[t]$, equipped with the above gradings.
        The homology of $\Pong{m}{m-k-1}$, with its induced $A_\infty$ structure,
        is
        uniquely characterized up to isomorphism by the following properties:
        \begin{itemize}
        \item The underlying algebra is  $\Clg(m,k)[t]$.
              \item The actions respect the above gradings, in the sense that
                    \begin{align*}
                \weight_i(\mu_d(a_1,\dots,a_d)) &=
                 \sum_{j=1}^d \weight_i(a_j) \\
                \gr(\mu_d(a_1,\dots,a_d))&=d-2 +\sum_{j=1}^d \gr(a_j)
                \end{align*}                        
             \item the $\mu_d$ are $\Field[t]$-multilinear; i.e.
              \[          
               \mu_d(a_1,\dots,a_{j-1},t \cdot a_j, a_{j+1},\cdots,a_d)
                                 =       
                                 t \cdot        \mu_d(a_1,\dots,a_d).\]
        \item There is a non-zero operation $\mu_{2m-2k}$.
        \end{itemize}  
\end{thm}

See Section~\ref{subsec:Extremes} for a few remarks on the extreme
        cases when $k=0$ or $m-1$.

Our proof of Theorem~\ref{thm:HomologyPongAinf} involves a blend
of algebraic and holomorphic techniques. In particular, the proof 
relies on the construction of modules associated to Heegaard diagrams
whose boundary is marked in a particular manner, the {\em
pong-bordered boundary}.

In this article, we first describe the pong algebra. Next, we compute
its homology as a $\IdempRing{m}{k}$-bimodule, in
Section~\ref{sec:HomologyPong}.
The unique characterization of the
${\mathcal A}_{\infty}$ structure involves a computation in Hochschild cohomology,
which we do in Section~\ref{sec:ComputeAinf}.

The Hochschild cohomology computation is facilitated by a certain
Koszul duality result. (Compare also~\cite{TorusAlg}.) 

Let $\Quot{m}{k}$ denote the quotient
\begin{equation}
        \label{eq:DefQuot}
        \Quot{m}{k}=\frac{\Pong{m}{k}}{v_1=\dots=v_m=0}.
        \end{equation} 
\begin{thm}
        \label{thm:KoszulDuality} 
        Given $0<k<m$, there is a 
        quasi-isomorphism \[ \Quot{m}{k}\simeq \Cobar(\Clg(m,k)).\]
\end{thm}

The above Koszul duality is proved as follows. First, we construct a
type $DD$ bimodule $\lsup{\Clg{}(m,k)}\DDmod^{\Quot{m}{k}}$.
Next, we construct a bordered diagram whose boundary decomposes as a
union of a pong-bordered piece, and another piece which is bordered as
in~\cite{HolKnot}. Adapting methods of bordered Floer
homology~\cite{InvPair}, we associate to this diagram a type $AA$
bimodule, which we prove is the inverse to
$\lsup{\Clg{}(m,k)}\DDmod^{\Quot{m}{k}}$, to prove
Theorem~\ref{thm:KoszulDuality}.  (The content of 
Theorem~\ref{thm:KoszulDuality} is restated in  Theorem~\ref{thm:KoszulDual},
with more attention to gradings.)

Theorem~\ref{thm:HomologyPongAinf} is proved using Hochschild
cohomology.  The computations in Hochschild cohomology are facilitated
by the relationship between the cobar algebra and the much smaller
$\Quot{m}{k}$ from Theorem~\ref{thm:KoszulDuality}, much in the spirit
of~\cite{TorusAlg}.

This paper is organized as follows. In Section~\ref{sec:Basics}, we
introduce much of the algebraic background and notation used
throughout.  In Section~\ref{sec:Clg}, we recall the bordered algebra
$\Clg(m,k)$ from~\cite{HolKnot}. In Section~\ref{sec:DefPong}, we
define the pong algebra itself. In Section~\ref{sec:PongGen}, we
introduce some particular generators of the pong algebra which will be
used throughout and verify some of their properties. In
Section~\ref{sec:HomologyPong}, we compute the homology of the pong
algebra, obtaining a weak version of
Theorem~\ref{thm:HomologyPongAinf}, on the level of vector spaces
(rather than $\Ainfty_{\infty}$-algebras).
Sections~\ref{sec:DDmod}-\ref{sec:Duality} are aimed at establishing
Theorem~\ref{thm:KoszulDuality}. In Section~\ref{sec:DDmod}, we define
(algebraically) a type $DD$ bimodule which connects $\Clg(m,k)$ and
$\Quot{m}{k}$. In Section~\ref{sec:AAmod}, we define its inverse
bimodule, by adapting holomorphic techniques in the spirit
of~\cite{InvPair}; see also~\cite{HolKnot}.  In Section~\ref{sec:Inverses}
we verify these bimodules are quasi-inverses to one another. In
Section~\ref{sec:Duality}, we review Koszul duality (in the spirit
of~\cite{HomPairing}) to deduce Theorem~\ref{thm:KoszulDuality} from
the fact that the $DD$ bimodule constructed before is
quasi-invertible.  In Section~\ref{sec:ComputeAinf}, we use
deformation theory to prove that $\Clg(m,k)[t]$ can be given a unique
non-trivial ${\mathcal A}_{\infty}$ structure. In
Section~\ref{sec:ComputeAinfty}, we show that the induced ${\mathcal
A}_{\infty}$ structure on $H(\Pong{m}{m-k-1})\cong\Clg(m,k)[t]$ is
non-trivial, completing the proof of Theorem~\ref{thm:HomologyPongAinf}.

{\bf Motivation for the constructions.}

We pause to explain some of the geometric content of these algebras.
As explained in~\cite{HolKnot}, the bordered algebras $\Clg(m,k)$ can
be thought of as generated by $k$-tuples of Reeb chords, with products
in the algebra recording the collisions that occur in certain ends of
moduli spaces.  In Section~\ref{sec:AAmod}, we also interpret the pong
algebra (and especially its $v_1=\dots=v_m=0$ specialization) as an
algebra of Reeb chords. (Compare the ``strands algebras''
of~\cite{InvPair}; see
also~\cite{DouglasManolescu,BernsteinGelfandGelfand,PetkovaVertesi}.)

Both algebras can be thought of as part of a more general construction
which simultaneously generalizes both the pointed matched circles
of~\cite{InvPair} and the configuration of boundary circles
from~\cite{HolKnot}.
The Koszul duality result from Theorem~\ref{thm:KoszulDuality} can be
thought of as an analogue of the duality result
from~\cite{HomPairing}. That paper sets up a Koszul duality between
the algebra associated to a pointed matched circle (i.e. a handle
decomposition of a surface with a single $0$-handle and $2$-handle)
and its dual pointed matched circle.  One could envision a more
general, where one associates an algebra to a (more general) handle
decomposition of a surface, which is Koszul dual to the algebra
associated to its dual handle decomposition.

\begin{figure}[ht]
\input{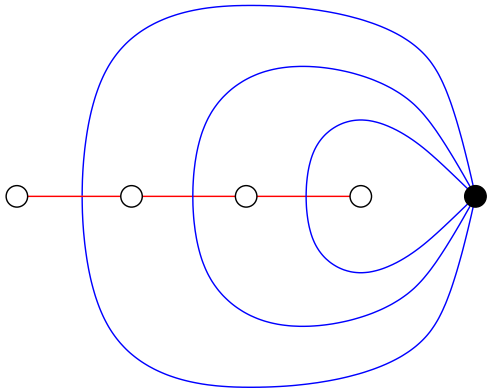}
\caption{\label{fig:DualCells}
{\bf{Dual cell decompositions of $S^2$.}}
The linear arrangement corresponds to $\Clg$ and its dual wedge of circles
correpsponds to the pong algebra.}
\end{figure}

In this interpretation, $\Clg(m,k)$ corresponds to the handle
decomposition of the sphere with $m$ zero-cells, $m-1$ one-cells, and
one $2$-cell. The zero-cells are arranged in a linear change,
connected by the $m-1$ one-cells.  The parameter $k$ specifies the
multiplicity of the symmetric product we are considering.  The algebra
$\Quot{m}{\ell}$ corresponds to the dual handle decomposition, with
one $0$-cell and $m-1$ one-cells arranged in a bouquet, and $m$
$2$-cells. See Figure~\ref{fig:DualCells}.  One expects these algebras
to be Koszul dual, when $k+\ell=m-1$. This Koszul duality is further
explored in~\cite{Tilings}.

Following Auroux~\cite{Auroux}, one can think of the generators of
pong algebra as morphisms in a wrapped Fukaya category, the
(specialized) $v_1=\dots=v_m=0$ product as product on the partially
wrapped Fukaya category of~\cite{AbouzaidSeidel}, and the
unspecialized algebra actions as products in a relative Fukaya
category. This latter identification is explored further
in~\cite{WrapPong}.

On an algebraic level, the algebras we consider here are closely
related to the algebras associated to the extended affine symmetric
groups of Manion and Rouquier~\cite{ManionRouquier}; see also
Remark~\ref{Diagrams} below.

Constructing a signed lift of the pong algebra (to $\Z$ coefficients)
is a straightforward task in the spirit of~\cite[Chapter~15]{GridBook}. Lifting
the holomorphic aspects would require more work.

{\bf Acknowledgements:} We owe a great debt of gratitude to
conversations with and work of Robert Lipshitz and Dylan Thurston.
The overall scheme of this paper follows very
closely~\cite{TorusAlg}. The holomorphic aspects of this paper come
from~\cite{InvPair} and~\cite{TorusMod}; and further algebra is drawn
from~\cite{HomPairing}.  The authors also wish to thank Andy Manion
for many interesting conversations, and to thank him and Rapha{\"e}l
Rouquier for freely sharing their work with us.

\section{Algebraic background}
\label{sec:Basics}

We will use the algebraic underpinnings from~\cite{Bimodules}.

\subsection{Type $D$ structures}

If $A$ is a DG algebra over a ring $I$, a {\em type $D$ structure} is a 
left $I$-module $X$, equipped with a map
\[ \delta^1\colon X \to A\otimes X,\]
satisfying the structural relation 
\begin{equation}
  \label{eq:TypeDStruct}
  (\mu_1\otimes \Id_X)\circ \delta^1 + (\mu_2\otimes\Id_X)\circ
(\Id_A \otimes \delta^1)\circ \delta^1=0.
\end{equation}

If $A$ and $B$ are two $DG$ algebras over $I$, a {\em type $DD$ bimodule}
$\lsup{A}X^B$ is an $I$-bimodule $X$, together with a structure map
\[ \delta^1\colon X \to A \otimes_I X \otimes_I B,\]
which in turn can be seen as a map 
\[ \delta^1\colon X \to (A\otimes B^{\op})_{I\otimes I^{\op}}\otimes X,\]
where here $\op$ denotes the ``opposite'' module.

For $A$ and $B$, we will also consider type $DA$ bimodules $\lsub{A}M^B$
and type $AA$ bimodules $\lsub{A}M_B$, as in~\cite[Section~.2.24]{Bimodules}.

\subsection{Grading conventions}
\label{sec:GradingConventions}

Gradings for our present purposes are simpler than the gradings
considered in ~\cite{InvPair,Bimodules}.

The $DA$ algebras we consider here will be equipped with
augmentations.  Specifically, let $\Ground$ be a ring (in our case,
this is the ring of idempotents).  A $\Ground$-algebra is
a
graded, associated algebra, which is also a $\Ground$-bimodule, so
that the multiplication map $\mu_2\colon A \otimes A \to A$ is a map
of $\Ground$-bimodules.

\begin{defn}
  \label{def:AugmentedAlgebra}
An {\em augmented} $\Ground$-algebra is an algebra, equipped with an
algebra map $\epsilon \colon A \to \Ground$, called the {\em
  augmentation}, whose kernel, the {\em augmentation ideal}, is
denoted $A_+$. Together, the unit (in $\Ground$) and the augmentation
provide a $\Ground$-$\Ground$-bimodule splitting $A\cong A_+\oplus
\Ground$.
\end{defn}

The algebras in this paper will all have the following special structure.

\begin{defn}
  \label{def:MaslovAlexAlgebra}
  A {\em Maslov/Alexander bigraded} algebra is a an augmented $DA$
  algebra $A$, equipped with two gradings, one (homological)
  $\Z$-grading, denoted $\gr^A$, and an Alexander grading, denoted
  $\vec{\weight}$, with values in $\OneHalf(\Z^{\geq 0})^m\subset
  \OneHalf \Z^m$.  Both gradings are additive under multiplication,
  and differential drops the homological grading by one while
  preserving the algebra grading.  The summand of $A$ with vanishing
  Alexander grading is the ground ring, and the augmentation ideal
  $A_+$ is the portion in non-zero Alexander grading. Finally, $A$
  will have the following boundedness hypothesis: for each
  $\vec{w}\in\OneHalf\Z^m$, the portion of $A$ with Alexander grading
  $\vec{w}$ is a finitely generated $\Ground$-bimodule.
\end{defn}

A Maslov/Alexander bigraded type $D$ structure $X$ is one which is
equipped with two gradings, one, denoted $\gr^X$, with values in $\Z$
and another, denoted $\vec{v}$, with values in $\OneHalf \Z^m$.  These
gradings are further required to be compatible with the structure map
$\delta^1$ in the sense that if $a\otimes y$ appears with non-zero
coefficient in $\delta^1(x)$, then
\begin{align*}
  \gr^X(x)-1=\gr^A(a)+\gr^X(y) \\
  \vec{v}(x)=\vec{\weight}(a)+\vec{v}(y).
\end{align*}

More generally, let $\lsub{A}M^B$ be a type $DA$ bimodule.
Recall that $DA$ bimodules are equipped with actions for all $p\geq 0$
\[ \delta^1_{p+1}\colon A^{\otimes(n-1)}\otimes M \to B\otimes M \]
satisfying a structure relation generalizing Equation~\eqref{eq:TypeDStruct}.
(See~\cite[Definition~2.2.43]{Bimodules}.)
We say that 
it is
{\em Maslov/Alexander bigraded} if it is equipped with $\gr^M$ and $\vec{v}$
so that if $y\otimes b$ appears in $\delta^1_{1+p}(a_1,\dots,a_p,x)$, then
\begin{align*}
  \gr^X(x)+\sum_{i=1}^p\gr^A(a_i) &= \gr^X(y)+\gr^B(b)+1-p \\
  \vec{v}(x)+\sum_{i=1}^p \vec{\weight}^A(a_i) &= \vec{v}(y)+\vec{\weight}^B(b).
\end{align*}
Furthermore, we say $\lsub{A}M^B$ is {\em (strictly) unital} if
$\delta^1_2(\One,x)=x\otimes \One$ and for all $p>1$, if $(a_1,\dots,a_p)$ is a sequence of algebra elements such that some $a_i=\One$, then 
$\delta^1_{p+1}(a_1,\dots,a_p,x)=0$.

Similarly, if $\lsub{A}M_B$ is a type $AA$ bimodule, we say that it is
{\em Maslov/Alexander bigraded} if 
for $y=m_{p|1|q}(a_1\otimes\dots\otimes a_p \otimes x \otimes b_1\otimes\dots\otimes b_q)$, then
\begin{align}
  \gr(y)^M &=
  \gr^M(x)+p+q-1+\sum_{i=1}^{p}\gr^A(a_i)+\sum_{i=1}^{q}\gr^B(b_i);   \label{eq:GradedBimodule}
  \\
\vec{v}(y)&=
  \left(\sum_{j=1}^p {\vec{\weight}}^A(a_j)\right) +  \vec{v}(x) + 
  \left(\sum_{j=1}^q {\vec{\weight}}^B(b_j)\right)
  \label{eq:AlexanderGrading}
\end{align}
The strict unitality condition in this case is that for all 
$x\in M$ and all algebra sequences $a_1,\dots,a_p$ and $b_1,\dots,b_q$
of algebra elements of length $p,q\geq 0$ with $p+q>1$,
with $\One\in \{a_1,\dots,a_p,b_1,\dots,b_q\}$, we have that
\begin{align*}
  m_{1|1|0}(\One,x)&=x \\
  m_{0|1|1}(x,\One)&=x \\
  m_{p,1,q}(a_1,\dots,a_p,x,b_1,\dots,b_q)&=0.
\end{align*}

\subsection{Quasi-inverses}

\begin{defn}
  Fix dg algebras $A$ and $B$, and let
  $\lsup{A}X^B$ and $\lsub{B}Y_A$ be a bimodules
  of type $DD$ and $AA$ respectively, as discussed above.
  We say that $X$ and $Y$ are {\em quasi-inverses}
  if
  \begin{align*}
  \lsup{A}X^B\otimes \lsub{B}Y_A &\simeq \lsup{A}\Id_A \\
  \lsub{B}Y_A\otimes\lsup{A}X^B &\simeq \lsub{B}\Id^B
  \end{align*}
\end{defn}

Let $A$ and $B$ be DG algebras over $\Ground$.  Let $\phi\colon A
 \to B$ be an $\Ainfty_\infty$ homomorphism.  There is an associated type $DA$
bimodule $\lsup{B}[\phi]_A$, which is rank one as as a $\Ground$ bimodule,
with actions specified by $\phi$.

The bimodules that arise in this way have a simple characterization:

\begin{lemma}
  \label{lem:BimoduleOfPhi}~\cite[Lemma~2.2.50]{Bimodules}
  Let $A$ and $B$ be a DG algebras over $\Ground$.
  Suppose that $\lsub{A}[M]^B$ is a type $DA$ bimodule with the following properties:
  \begin{itemize}
  \item $\delta^1_1=0$
  \item $M$ has rank $1$ as a $\Ground$ bimodule.
  \end{itemize}
  Then, there is an $\Ainf_\infty$ homomorphism $\phi\colon A\to B$ so that
  $\lsub{A}[M]^B=\lsub{A}[\phi]^B$.
\end{lemma}

When $A$ and $B$ are Maslov/Alexander bigraded, an $\Ainfty_\infty$ homomorphism
respects this structure if
\begin{align*}
  \weight^B(\phi_p(a_1\otimes\dots \otimes a_p))&=\sum_{i=1}^p \weight^A(a_i) \\
  \gr^B(\phi_p(a_1\otimes\dots\otimes a_p)&= \sum_{i=1}^p \gr^A(a_i)+p-1.
\end{align*}
Note that $\phi$ is a Maslov/Alexander graded $\Ainf_\infty$ homomorphism if
its associated bimodule $\lsub{A}[\phi]^B$ is Maslov/Alexander
bigraded, supported in vanishing bigrading.

\newcommand\ClgHat{\widehat\Clg}
\newcommand\Wedge{\Lambda}
\newcommand\BlgZ{\Blg_0}
\section{The bordered algebras}
\label{sec:Clg}

We recall the bordered algebras of~\cite{BorderedKnots}; in fact we
use mostly notation from~\cite[Section 3.2]{HolKnot}.  The algebras we
are interested in are denoted $\Clg(m,k)$, where $(m,k)$ is a pair of
integers with $0\leq k\leq m-1$. These algebras are constructed by
first constructing an $\Field[U_1,\dots,U_m]$-algebra $\BlgZ(m,k)$,
forming the quotient $\Blg(m,k)$ by an ideal, and then considering a
subalgebra $\Clg(m,k)$. We recall the details presently, referring the reader
to~\cite{BorderedKnots} for more details.
(Note also that in~\cite{HolKnot}, we considered the case
where $m=2k$, since these were the algebras that arose naturally in
the case of ``upper Heegaard diagrams''; but this hypothesis is not
needed for the algebraic constructions.)

\subsection{The algebras $\BlgZ$ and $\Blg$}

First, we recall the construction of $\BlgZ(m,k)$, associated to a
pair of integers $(m,k)$ with $0\leq k \leq m-1$. (The
construction of $\BlgZ(m,k)$ also works in the cases where
$k=m$ or $k = m+1$; but in those cases, the subalgebra $\Clg(m,k)$ we are
interested in is zero.)

The base ring of $\BlgZ(m,k)$ is the polynomial
algebra $\Field[U_1,\dots,U_m]$. Idempotents correspond to $k$-element
subsets $\x$ of $\{0,\dots,m\}$ called {\em idempotent states}. 

Given idempotents states $\x,\y$, the $\Field[U_1,\dots,U_m]$-module
$\Idemp{\x}\cdot\BlgZ(m,k)\cdot\Idemp{\y}$ is identified with
$\Field[U_1,\dots,U_m]$, given with a preferred generator $\gamma_{\x,\y}$.
To specify the product, we proceed as follows. Each idempotent state
$\x$ has a {\em weight vector} $v^\x\in \Z^m$, with components $i=1,\dots,m$
given by
$v^\x_i=\#\{x\in \x\big|x\geq i\}$.
The multiplication is specified by 
\[ \gamma_{\x,\y}\cdot \gamma_{\y,\z}=U_1^{n_1}\cdots U_m^{n_m} \cdot \gamma_{\x,\z},\] where
\[ n_i = \frac{1}{2}(|v_i^\x-v_i^\y|+|v_i^\y-v_i^\z|-|v_i^\x-v_i^\z|).\]

\begin{defn}
  \label{def:DistAlgElts}
The algebra $\BlgZ(m,k)$ is equipped with distinguished elements $\{L_i\}_{i=1}^m$,
$\{R_i\}_{i=1}^m$, and $\{U_i\}_{i=1}^m$, defined as follows.
For $i=1,\dots,m$, let $L_i$ be the sum of $\gamma_{\x,\y}$, taken
over all pairs of idempotent states $\x, \y$ so that there is some
integer $s$ with $x_s= i$ and $y_s=i-1$ and $x_t=y_t$ for all $t\neq
s$. Similarly, let $R_i$ denote the sum of all the $\gamma_{\y,\x}$
taken over the same pairs of idempotent states as above.  We will
often think of $U_i$ as an element of $\BlgZ(m,k)$; explicitly, it is
the sum over all idempotents $\x$ of the element
$\gamma_{\x,\x}\cdot U_i$.  
With this definition, multiplication by
$U_i\in \BlgZ(m,k)$ corresponds to the action of
$U_i\in\Field[U_1,\dots,U_m]$ on $\BlgZ(m,k)$.
\end{defn}

Under the identification $\Idemp{\x}\cdot \BlgZ(m,k)\cdot
\Idemp{\y}\cong \Field[U_1,\dots,U_m]$, the elements that correspond to monomials
in the $U_1,\dots,U_m$ are called {\em pure algebra elements in $\BlgZ(m,k)$}.
These elements are  specified by their idempotents $\x$ and $\y$, and their
{\em relative weight vector} $w(b)\in
\Q^m$, which in turn is uniquely characterized by
\[
   \weight_i(\gamma_{\x,\y})=\OneHalf|v_i^\x-v_i^\y| \qquad
   \weight_i(U_j\cdot b) = \weight_i(b)+\left\{\begin{array}{ll}
      0 &{\text{if $i\neq j$}} \\
      1 &{\text{if $i=j$.}} 
      \end{array}\right.
\]

Let ${\mathcal J}\subset \BlgZ(m,k)$ be the two-sided ideal generated by
$L_{i+1}\cdot L_i$, $R_{i}\cdot R_{i+1}$ and, for all choices of
$\x=\{x_1,...,x_k\}$ with $\x\cap \{j-1,j\}=\emptyset$,
the element 
$\Idemp{\x}\cdot U_j$.
Then, 
\[ \Blg(m,k)=\BlgZ(m,k)/{\mathcal J}.\]
The pure algebra elements in $\Blg(m,k)$ are those elements that
are images of pure algebra elements in $\BlgZ(m,k)$ under the above quotient map.
We use the symbols $L_i$, $R_i$, and $U_i$ to denote the elements of $\Blg(m,k)$
which are the 
the images of the corresponding
elements of $\BlgZ(m,k)$ as defined in Definition~\ref{def:DistAlgElts}.

\begin{defn}
  \label{def:TooFar}
  Idempotents $\x=x_1<\dots<x_k$ and $\y=y_1<\dots<y_k$ are said to be {\em too far} 
if for some $t\in \{1,\dots,k\}$, $|x_t-y_t|\geq 2$. If $\x$ and $\y$ are too far,
then $\Idemp{\x}\cdot\BlgZ(m,k)\cdot\Idemp{\y}\subset {\mathcal J}$; i.e.
$\Idemp{\x}\cdot\Blg(m,k)\cdot\Idemp{\y}=0$. 
\end{defn}

We restate here the concrete description of the ideal ${\mathcal J}$ given in~\cite{BorderedKnots}:

\begin{prop}\cite[Proposition~3.7]{BorderedKnots}
  \label{prop:Ideal}
  Suppose that $b=\Idemp{\x}\cdot b\cdot \Idemp{\y}$ is a pure algebra
  element. The element $b$ is in ${\mathcal J}$
  if and only if
  $\x$ and $\y$ are too far,
  or there is a pair of integers $i<j$ so that
  \begin{itemize}
  \item $i,j\in\{0,\dots,m\}\setminus \x\cap\y$
  \item for all $i<t<j$, $t\in\x\cap\y$
  \item $\weight_t(b)\geq 1$ for all $t=i+1,\dots,j$
  \item $\#(x\in \x\big| x\leq i)=\#(y\in \y\big| y\leq i)$.
  \end{itemize}
\end{prop}

\subsection{The algebra $\Clg$}
Having defined $\Blg(m,k)$, we turn now to the construction of $\Clg(m,k)$,
which is of primary interest here. To this end,
let 
\[ {\mathbf J}=\sum_{\{\x\mid \x\cap \{0,
    m\}=\emptyset\}} \Idemp{\x} \]
The subalgebra
$\Clg(m,k)\subset \Blg(m,k)$
is defined  by
\[ \Clg(m,k)= {\mathbf J}\cdot \Blg(m,k)\cdot {\mathbf J}.\]
There is a projection map
$\pi\colon \Blg(m,k)\to \Clg(m,k)$ defined by 
\[ \pi(b)={\mathbf J}\cdot b \cdot {\mathbf J}.\]
Equivalently, $\Clg(m,k)$ corresponds to the subalgebra
whose idempotent subalgebra correspond to those idempotent states
$\x\subset \{1,\dots,m-1\}$.

In Definition~\ref{def:DistAlgElts}, we defined distinguished algebra
elements $\{L_i\}_{i=1}^m$, $\{R_i\}_{i=1}^m$, and $\{U_i\}_{i=1}^m$
in $\BlgZ(m,k)$, which project to elements of $\Blg(m,k)$ (with the
same notation).  Note that $\pi(L_1)=\pi(R_1)=\pi(R_m)=\pi(L_m)=0$.
Let
$\{L_i\}_{i=2}^{m-1}$, $\{R_i\}_{i=2}^{m-1}$, and $\{U_i\}_{i=1}^m$
be the elements of $\Clg(m,k)$ obtained by projecting (under $\pi$) the
corresponding elements of $\Blg(m,k)$.

Given $i\in \{1,\dots,m-1\}$, let
\[ [i]=\sum_{\{\x\mid i\in\x\}}\Idemp{\x} \qquad\text{and}\qquad
[\inotin{i}]=\sum_{\{\x\mid i\not\in\x\}}\Idemp{\x}.\]
The following relations in $\Clg(m,k)$
hold for all $2\leq i\leq m-1$:
\begin{align*}
         ([i]\cdot [\inotin{i-1}])\cdot L_i \cdot [i-1]\cdot[\inotin{i}] = L_i 
         \qquad &\quad 
         ([i-1]\cdot [\inotin{i}])\cdot R_i \cdot [\inotin{i-1}]\cdot[i] = R_i \\
         L_i \cdot R_i = [\inotin{i-1}]\cdot [i]\cdot U_i \qquad 
         & \qquad
         R_i\cdot L_i = [i-1]\cdot [\inotin{i}]\cdot U_i.
\end{align*}
Moreover, for all $2\leq i \leq m-2$  
\[  R_i \cdot R_{i+1}=0 \qquad   \qquad L_{i+1} \cdot L_{i}=0.\]
(The first four relations hold even in $\BlgZ(m,k)$; whereas the last two hold in 
$\Blg(m,k)$.)

\subsection{Gradings}

The weight vector $\weight$ induces a grading on $\Clg$ with values in
$\OneHalf\Z^m$. It will sometimes be helpful to view $\Clg$ as a
Maslov/Alexander bigraded  $DG$ algebra, with trivial
differential.  When doing so, we will view the additional $\Z$-grading
as also trivial; i.e. the algebra is supported entirely in $\gr=0$.

\section{Definition of the pong algebra}
\label{sec:DefPong}

The aim of this section is to introduce the {\em pong algebra} with
$m$ {\em walls} and $k$ strands, denoted $\Pong{m}{k}$. Here
$(m,k)$ is a pair of integers with $0\leq k<m$ and $m >1$. The base ring
of the pong algebra is a polynomial ring $\Field[v_1,\dots,v_{m}]$,
and the algebra comes equipped with a preferred generating set, given by
the so-called {\em lifted partial permutations}, which we define
presently.

\subsection{Lifted partial permutations}

Let $r_t\colon \R\to\R$ be the reflection $r_t(x)=2t-x$; and consider
the subgroup $G_m$ of the reflection group of the real line generated
by $r_{\OneHalf}$ and $r_{m-\OneHalf}$. The quotient of the integral lattice by
this group of rigid motions is naturally an $m-1$ point set; generated
by $\{1,\dots,m-1\}$. Let
\[Q \colon \Z \to \{1,\dots,m-1\}\] denote this quotient map.
Explicitly, $Q(i)$ is the integer $1\leq t \leq m-1$ so that
$t\equiv i\mod{2m-2}$ or $t\equiv 1-i\pmod{2m-2}$.

Note that $G_m$ also acts on the set $\OneHalf+ \Z$. The quotient of $\OneHalf
+ \Z$ by $G_m$ is naturally the $m$-point set, $\{\OneHalf,\dots,
m-\OneHalf\}$.  We think of these points as being in one-to-one
correspondence with the underlying variables in the pong algebra,
where the point $j\in \{\OneHalf,\dots,m-\OneHalf\}$ corresponds to the 
variable $u_{\OneHalf+j}$.

A $G_m$ invariant subset $\LiftS$ of $\Z$ has a natural quotient $\LiftS/G_m$,
which is a subset of $\{1,\dots,m-1\}$. 

\begin{defn}
  A {\em lifted partial permutation on $k$ letters} 
  is a pair $({\widetilde S}, {\widetilde f})$ where:
  \begin{itemize}
  \item ${\widetilde S}\subset \Z$
    is a $G_m$-invariant subset 
  \item ${\widetilde f}\colon {\widetilde S} \to \Z$
    is a $G_m$-equivariant map;
  \end{itemize}
  subject to the following two conditions:
  \begin{itemize}
    \item  ${\widetilde S}/G_m$ consists of $k$ elements
    \item   the induced map ${\widetilde f}\colon {\widetilde S}/G_m\to \Z/G_m$.
      is injective.
  \end{itemize}
  Let $\LiftPartPerm$ denote the set of lifted partial permutations.
\end{defn}

The following data specify a lifted partial permutation:

\begin{defn}
  Given integers $(k,m)$, we say that {\em pong data of type $(k,m)$} consists of a pair $(S,f)$,
  where 
  \[S\subseteq \{1,\dots,m-1\}=\Z/ G_m\] is a $k$ element subset, and
  $f\colon S \to \Z$ is a map, 
  so that the induced map $Q\circ f\colon S \to \Z/G_m$ is injective.
\end{defn}

\begin{lemma}
  \label{lem:PongData}
  Given a lifted partial permutation $({\widetilde S},{\widetilde f})$,
  we can associate pong data $(S,f)$, where $S={\widetilde S}/G_m$
  and $f$ is the restriction of ${\widetilde f}$ to $\{1,\dots,m-1\}\subset \Z$.
  This restriction map induces a one-to-one correspondence betweeen
  lifted partial permutations (on $k$ letters) and pong data (of type $(k,m)$).
\end{lemma}
\begin{proof}
  Restriction gives a map from lifted partial permutations to pong data.
  Conversely,
  given pong data, let ${\widetilde S}=Q^{-1}(S)$; so that any 
  ${\widetilde
    s}\in{\widetilde S}$ can be written as ${\widetilde s}=\gamma\cdot s$,
  for a unique choice of $s\in\{1,\dots,m-1\}$ and a unique $\gamma\in G_m$.
  Define the corresponding lifted partial permutation
  ${\widetilde f}$ by 
  \[{\widetilde f}({\widetilde s})=\gamma\cdot f(s).\]
\end{proof}

If a lifted partial permutation $({\widetilde S},{\widetilde f})$
restricts to pong data $(S,f)$, we call $({\widetilde S},{\widetilde
  f})$ the {\em equivariant lift} of $(S,f)$.

Write $S=s_1<\dots<s_k$. The partial permutation will be labeled by
the $k$-tuple of ordered pairs $(s_1,f(s_1)),\dots,(s_k,f(s_k))$.

Pong data  $(S,f)$ can be represented graphically as
follows. Consider the infinite strip $\R\times [-1,0]$; where $[-1,0]$
is an interval which we think of as measuring the height. Draw a
collection of straight segments connecting $({\widetilde s},0)$ for
${\widetilde s}\in \LiftS$ to $(f({\widetilde s}),-1)$. If we further draw all the translates of these segments by the action of $G_m$, we obtain the {\em lifted pong diagram}.

All of the
information in the partial permutation is contained in the portion of
this picture in the finite rectangle $[\OneHalf,m-\OneHalf]\times
[-1,0]$. For each $s\in S$, the segment from $s$ to
$f(s)/G_m\in\{1,\dots,m-1\}$, called a {\em strand}, is transformed into
a continuous, piecewise linear curve connecting $(s,0)$ to
$(Q(f(s)),-1)$ whose slope has fixed absolute value. The curves fail
to be linear at the points where the horizontal coordinate is
$\OneHalf$ or $m-\OneHalf$.  Points of the first kind are called {\em
  left extrema} and those of the second are called {\em right
  extrema}.
This picture is called the {\em pong diagram} for the lifted partial
permutation.

\begin{remark}
  \label{Diagrams}
  The pong diagrams defined here can be viewed as a kind of $\Zmod{2}$
  quotient of the cylindrical diagrams considered
  in~\cite{ManionRouquier}.
\end{remark}

\begin{example}
  Let $m=4$, and consider the lifted partial permutation $f$ on $2$
  letters sending $1$ to $-2$ and $2$ to $1$, which we write
  $((1,-2),(2,1))$. The pong diagram is pictured in
  Figure~\ref{fig:PongPictures}, and its
  lift is shown in Figure~\ref{fig:LiftPongPictures}.
\end{example}

\begin{figure}[ht]
\input{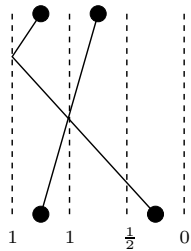}
\caption{\label{fig:PongPictures} {\bf{Pong diagrams.}}
This is a pong diagram for the lifted partial permutation
$((1,-2),(2,1))$ (where the vertical coordinate has been scaled up).
The dotted lines occur at horizontal positions $1/2$, $3/2$, $5/2$, and $7/2$ respectively.
There is one left extremum in the picture, occurring on the strand out of $1$;
and there are no right extrema.
The local multiplicities at these positions are recorded in the figure;
i.e. in the notation of Definition~\ref{def:LocalMult}, we have
$\weight(f)=(1,1,\OneHalf,0)$.}
\end{figure}

\begin{defn}
  \label{def:LocalMult}
  Let $(S,f)$ be pong data, and let ${\widetilde
    f}\colon \LiftS\to \Z$ be its equivariant lift.  For each
  $j\in\OneHalf+\Z$, the {\em local multiplicity of the lift of $f$ at
    $j$}, denoted $\weight_j(f)$, is given by
  \[ \weight_j(f)=\OneHalf \#\{i\in \LiftS\big|i < j < \Liftf(i)~\text{or}~
  \Liftf(i) < j < i\}.\]
  Note that for all $\gamma\in G_m$,
  $\weight_{j}(f)=\weight_{\gamma\cdot j}(f)$, so
  the local multiplicity descends to a function of $j\in\left(\OneHalf+\Z\right)/G_m=\{\OneHalf,\dots,
  m-\OneHalf\}$; or equivalently, it gives vector
  $\weight(f)\in(\OneHalf \Z)^{m}$, which we call the {\em vector of local multiplicities}.
\end{defn}

\begin{figure}[ht]
\input{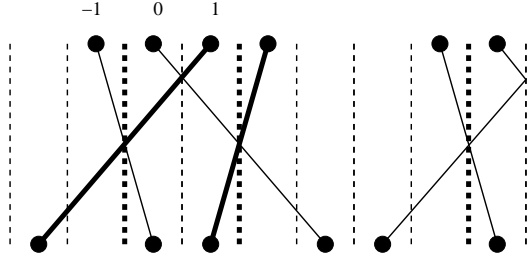}
\caption{\label{fig:LiftPongPictures} {\bf{Lifting the pong diagrams.}}
We have drawn here a lift of the pong diagram from Figure~\ref{fig:PongPictures};
drawing three copies of the fundamental domain. The darkened dashed lines represent the three translates of $3/2$, and the two darkened lines represent the two
strands. The two intersection points of these lines with the darkened dashed lines demonstrate that $\weight_{3/2}(S,f)=1$.}
\end{figure}

The local multiplicities of a lifted partial partial permutation
can be read off from the
pong diagram, as follows. The local multiplicity at $\OneHalf$ is the number of left extrema
in $s$; the local multiplicity at $m-\OneHalf$ is the number of right extrema; 
and the local multiplicity at $j$ with $\OneHalf<j<m-\OneHalf$ is 
half the intersection number of
the strand with the vertical segment at $j$.

Consider equivalence classes of pairs $(i,j)\in\Z\times \Z$ under the
equivalence relation generated by $(i,j)\sim (j,i)$ and $(i,j)\sim
(\gamma\cdot i,\gamma\cdot j)$ for $\gamma\in G_m$.  Let $\langle i,j\rangle$ denote
the equivalence class of $(i,j)\in\Z\times \Z$ under that equivalence
relation.

\begin{defn}
  A {\em lifted crossing} of $(S,f)$ is an element $(i,j)$ of
  ${\widetilde S} \times {\widetilde S}$ with the property that the
  sign of $i-j$ is opposite to the sign of ${\widetilde
    f}(i)-{\widetilde f}(j)$; i.e.  $\sgn(i-j)\neq
  \sgn(\Liftf(i)-\Liftf(j))$. Note that any $G_m$ orbit of a lifted
  crossing is a lifted crossing; also if $(i,j)$ is a lifted crossing,
  then so is $(j,i)$. A {\em crossing} is the equivalence class
  $\langle i,j\rangle$ of the crossing $(i,j)$ under the above
  equivalence relation.  The set of crossings of $(S,f)$ is denoted
  $\Cross(S,f)$, or simply $\Cross(f)$ when $S$ is clear from the
  context.
\end{defn}

See Figure~\ref{fig:LiftCrossing} for an illustration.
\begin{figure}[ht]
\input{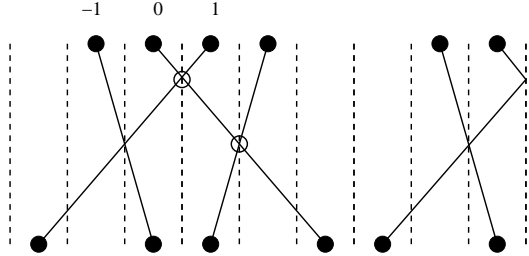}
\caption{\label{fig:LiftCrossing} {\bf{Crossings.}}  We have
  illustrated here part of the lifted pong diagram corresponding to
  Figure~\ref{fig:PongPictures}, consisting of three fundamental
  domains.  The two circled intersection points correspond to the two
  crossings in $\Cross(f)$.}
\end{figure}

To find all the crossings on a given strand connecting $(s,0)$ and
$(f(s),-1)$ in a lifted pong diagram, find all the intersection points
of the corresponding segment with all of the segments connecting
$({\widetilde t},0)$ and $({\widetilde f}({\widetilde t}),-1)$, with
${\widetilde t}\neq s$.  From this description, it is obvious that a
lifted partial permutation has only finitely many crossings.

Crossings in the lifted partial permutation are of two kinds: either
the two strands $i$ and $j$ are in the same $G_m$-orbit, or they are
not. In the pong diagram, crossings of the first kind correspond to
left or right extrema (i.e. whose $x$-coordinate is $\OneHalf$ or
$m-\OneHalf$), while crossings of the second kind can occur either at
these extrema, or they can occur in the interior.
Indeed, we can read off the crossing number from the pong diagram as follows.
The pong diagram contains a one-manifold away from finitely many points. 
In the interior, $t$ strands can meet, and 
such a point contributes $\binom{t}{2}$ crossings; at 
the $x=\OneHalf$ or $m-\OneHalf$ walls, $2t$ line segments meet,
and such points contribute $t^2$ to the crossing count. The validity of the
formula can be readily seen by lifting the pong diagram.

\begin{defn}
  \label{def:cross}
  Define the {\em crossing number} $\cross(f)$ to be the number of crossings in the lifted
  partial permutation $f$; i.e. $\cross(f)=|\Cross(f)|$.
\end{defn}

Alternatively, we can consider a more general kind of lifted pong
diagram, where the curves connecting $({\widetilde s},0)$ and
$({\widetilde f}({\widetilde s}),-1)$ are no longer required to be
straight, but we still require that the picture is $G_m$-equivariant.
We can arrange for this picture to be generic, in the sense that there
are no triple points.  Equivariance gives rise to a a picture in the
quotient, which we think of as a perturbed pong diagram. The quotient
of a generic pong diagram has no triple points; and no two strands
intersect at $\{\OneHalf\}\times [-1,0]$ or $\{m-\OneHalf\times
[-1,0]\}$. The number of crossings plus the number of (left and right)
extrema in any sufficiently small generic perturbation of a pong
diagram for $(f,S)$ agrees with the $\cross(f,S)$.  (Note that larger
perturbations can give rise to additional crossings; see
Figure~\ref{fig:HardToFind}.)

\begin{figure}[ht]
\input{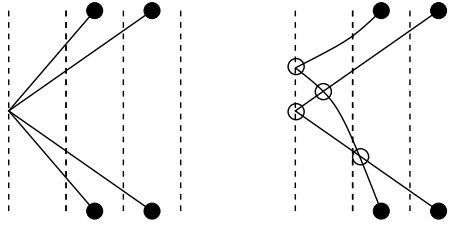}
\caption{\label{fig:PerturbPong} {\bf{Crossings at extrema.}}
The intersection at the wall in the left picture
(corresponding to a quadruple points in the lifted pong diagram) 
is perturbed to give four crossings in the
generic pong diagram in the right picture.}
\end{figure}

\begin{figure}[ht]
\input{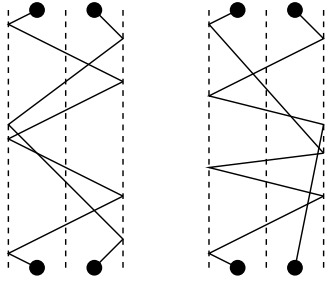}
\caption{\label{fig:HardToFind} {\bf{Generic perturbations of pong diagrams.}}
The pong diagram on the left, which has 12 crossings, 
has a perturbation, shown on the right, which has 14 crossings.}
\end{figure}

Pong data $(S,f)$ and $(T,g)$ with $T=Q\circ f(S)$ specify lifted
partial permutations, whose composition is represented by
$(S,\Liftg\circ f)$.

\begin{lemma}
  \label{lem:ComposePermutations}
  If $(S,f)$ and $(T,g)$ are two pong data,
  with $T=Q\circ f(S)$, then 
  \begin{align}
    \weight(\Liftg\circ f)\leq \weight(g)+\weight(f) \label{eq:ComposeMult} \\
    \cross(\Liftg\circ f)\leq \cross(g)+\cross(f). \label{eq:ComposeCross} 
  \end{align}
  Moreover, $\weight(\Liftg\circ f)-\weight(g)-\weight(f)\in \Z^{m}\subset (\OneHalf \Z)^{m}$.
\end{lemma}

\begin{proof}
  Fix $(i,j) \in (\LiftS\times \LiftS)/G_m$. 
  The following is straightforward to verify:
  \begin{itemize}
    \item if $\langle i,j\rangle \not\in \Cross(f)$, and $\langle {\widetilde f}(i),{\widetilde f}(j)\rangle\not\in\Cross(g)$, then
      $\langle i,j\rangle\not\in\Cross(\Liftg\circ f)$.
    \item if $\langle i,j\rangle\in \Cross(f)$, and $\langle {\widetilde f}(i),{\widetilde f}(j)\rangle\not\in\Cross(g)$, then
      $\langle i,j\rangle\in\Cross(\Liftg\circ f)$.
    \item if $\langle i,j\rangle\not\in \Cross(f)$, and $\langle {\widetilde f}(i),{\widetilde f}(j)\rangle\in\Cross(g)$, then
      $\langle i,j\rangle\in\Cross(\Liftg\circ f)$.
    \item if $\langle i,j\rangle\in \Cross(f)$, and $\langle {\widetilde f}(i),{\widetilde f}(j)\rangle\in \Cross(g)$, then
      $\langle i,j\rangle\not\in\Cross(\Liftg\circ f)$.
  \end{itemize}
  Thus, if $\langle i,j\rangle\in \Cross(\Liftg\circ f)$, then exactly one of the following two conditions holds
  \begin{itemize}
  \item $\langle i,j\rangle\in \Cross(f)$
  \item $\langle {\widetilde f}(i),{\widetilde f}(j)\rangle\in \Cross(g)$.
  \end{itemize}
  This dichotomy gives an injection
  \[ \Cross(\Liftg\circ f)\to\Cross(g)\cup\Cross(f),\]
  establishing Inequality~\eqref{eq:ComposeCross}.

  To see Inequality~\eqref{eq:ComposeMult}, we argue similarly. For any 
  $t\in \OneHalf + \Z$, let 
  \begin{equation}
    \label{eq:DefXr}
    X_t(f)=\{i\in\LiftS\big| i < t < \Liftf(i)~\text{or}~
  \Liftf(i) < t < i \},
  \end{equation}
  so that 
  \[ \weight_t(f)=\OneHalf \# X_t(f).\]
  It is straightforward to verify that the $i\in\Z$ can be partitioned into the following four types: 
  \begin{enumerate}[label=(T-\arabic*),ref=(T-\arabic*)]
  \item \label{type:Special} if $i\in X_t(f)$ and $f(i)\in X_t(g)$, then $i\not\in X_t(\Liftg\circ f)$.
  \item if $i\in X_t(f)$ and $f(i)\not\in X_t(g)$, then $i\in X_t(\Liftg\circ f)$.
  \item if $i\not\in X_t(f)$ and $f(i)\in X_t(g)$, then $i\in X_t(\Liftg\circ f)$.
  \item if $i\not\in X_t(f)$ and $f(i)\not\in X_t(g)$, then $i\not\in X_t(\Liftg\circ f)$.
  \end{enumerate}
  Inequality~\eqref{eq:ComposeMult} follows. Moreover, each $i\in\Z$ of Type~\ref{type:Special} contributes
  $1$ to $\weight_t(f)+\weight_t(g)-\weight_t(\Liftg\circ f)$, and all other $i\in\Z$ contribute $0$.
\end{proof}

\begin{defn}
  Let $(S,f)$ be a pong data, and fix some crossing
  $\langle i,j\rangle $ for $f$. 
  The {\em lifted resolution of $f$ at $\langle i,j\rangle $}
  is the $G_m$-equivariant 
  map $\Liftf_{\langle i,j\rangle}\colon \LiftS \to \Z$ defined by
  \[
  \Liftf_{\langle i,j\rangle}(k) = 
  \left\{\begin{array}{ll}
      \Liftf(k) & {\text{if $[k]\neq [i],[j]$}} \\
      g\cdot \Liftf(j) & {\text{if $k=g\cdot i$}} \\
      g\cdot \Liftf(i) & {\text{if $k=g\cdot j$}}.
    \end{array}\right.
  \]
  The {\em resolution of $f$ at $\langle i,j\rangle$}, denoted
  $f_{\langle i,j\rangle}$, is the induced map
  $f_{\langle i,j\rangle}=\Liftf_{\langle i,j\rangle}/G_m$.
\end{defn}

To see that ${\widetilde f}_{\langle i,j\rangle}$ is indeed a lifted partial
permutation, note that if ${\overline f}\colon S \to \Z/G_m$ denotes
the map induced from $f$, then the map from $S\subset \Z/G_m\to
\Z/G_m$ induced by $f_{\langle i,j\rangle}$ is obtained by
pre-composing ${\overline f}$ with the automorphism of $S$ that
switches $[i]$ and $[j]$; i.e. it is also injective.

\begin{lemma}
  \label{lem:ResolvePermutation}
  Let $(S,f)$ be a lifted partial permutation, 
  and $\langle i,j\rangle \in \Cross(f)$ be a crossing.
  Then,
  \begin{align}
    \weight(f_{\langle i,j\rangle})&\leq \weight(f) \label{eq:ResolveMult} \\
    \cross(f_{\langle i,j\rangle})&\leq \cross(f)-1. \label{eq:ResolveCross}
  \end{align}
  Moreover, $\weight(f)-\weight(f_{\langle i,j\rangle})\in
  \Z^{m}\subset (\OneHalf\Z)^{m}$.
\end{lemma}

\begin{proof}
  Clearly, $\langle k,\ell \rangle$ with
  $\{[k],[\ell]\}\cap\{[i],[j]\}=\emptyset$, 
  $\langle k,\ell\rangle\in\Cross(f)$ if and only if
  $\langle k,\ell\rangle\in\Cross(f_{\langle i,j\rangle})$. 
  Moreover, it is
  straightforward to verify the following:
  \begin{itemize}
    \item If 
      $|\{\langle i,k\rangle, \langle j,k\rangle\}\cap \Cross(f)|=0$
      then 
      $|\{\langle i,k\rangle, \langle j,k\rangle\}\cap \Cross(f_{\langle i,j\rangle})|=0$.
      \item If
        $|\{\langle i,k\rangle, \langle j,k\rangle\}\cap \Cross(f)|=1$,
        then
        $|\{\langle i,k\rangle, \langle j,k\rangle\}\cap \Cross(f_{\langle i,j\rangle})|=1$.
      \item If
        $|\{\langle i,k\rangle, \langle j,k\rangle\}\cap \Cross(f)|=2$,
        then
        $|\{\langle i,k\rangle, \langle j,k\rangle\}\cap \Cross(f_{\langle i,j\rangle})|=0$.
      \end{itemize}
      Obviously, $\langle i,j\rangle\in\Cross(f)$ and $\langle i,j\rangle \not\in\Cross(f_{\langle i,j\rangle})$.
      Thus,
      \[ \cross(f)-\cross(f_{\langle i,j\rangle})
      = 1 +  2 \#\{ k\big| \langle i,k\rangle~\text{and}~\langle j,k\rangle\in \Cross(f)\rangle\}, \]
      and Inequality~\eqref{eq:ResolveCross} follows.

      Fix $t\in \OneHalf+\Z$, and 
      let $X_t(f)$ be as in Equation~\eqref{eq:DefXr}.  
      Given $\langle i,j\rangle\in\Cross(f)$, define a map 
      $T\colon \Z \to \Z$ by
      \[T(k)=\left\{
      \begin{array}{ll}
        k &{\text{if $[k]\not\in\{[i],[j]\}$}} \\
        g\cdot j &{\text{if $k=g\cdot i$ for some $g\in G_m$}} \\
        g \cdot i &{\text{if $k=g\cdot j$ for some $g\in G_m$}}.
      \end{array}\right.\]
    Let 
    \[ A=\{k\in X_t(f)\big|T(k)\in X_t(f), k<t<T(k)\text{~or~}
    T(k)<t<k\} \]
    and $B=X_t(f)\setminus A$.
    It is elementary to see that $B=X_t(f_{\langle i,j\rangle})$.
    To see this, note that $k\in B$ if  $t$ lies between $k$ and $f(k)$, and one of the following three holds:
    \begin{itemize}
      \item  $[k]\not\in\{[i],[j]\}$,  in which case $f(k)=f_{\langle i,j\rangle}(k)$.
      \item $k=g\cdot i$ for some $g\in G_m$, and $t$ does not lie between $g\cdot i$ and $g \cdot j$.
        In this case, $t$ lies between $g \cdot j$ and $g \cdot f(i)$, i.e. $t\in X_t(f_{\langle i,j \rangle})$.
      \item $k=g\cdot j$ for some $g\in G_m$, and $t$ does not lie between $g \cdot i$ and $g \cdot i$.
        In this case, $t$ lies between $g\cdot i$ and $g \cdot f(j)$, so once again, 
        $t\in X_t(f_{\langle i,j \rangle})$.
    \end{itemize}
    Moreover,
    \[ \weight_t(f)-\weight_t(f_{\langle i,j\rangle})=\OneHalf|A|.\]
    The map $T$ is a fixed point free involution on $A$, showing that
    $\weight_{[t]}(f)-\weight_t(f_{\langle i,j\rangle})\in \Z^m$, as claimed.
\end{proof}

\begin{lemma}
  \label{lem:PreDSquared}
  Given any $x\in\Cross(f)$, there is an inclusion
  \[\Phi_{x}\colon \Cross(f_x)\to   \Cross(f)\]
  with the property that if $y=\Phi_x(y')$, then for some
  $x'\in\Cross(f_y)$, $x=\Phi_y(x')$; and
  \[ (f_{x})_{y'}=(f_{y})_{x'}.\]
\end{lemma}

\begin{proof}
  Write $x=\langle i,j\rangle$ and $\langle k,\ell\rangle\in\Cross(f)\setminus\{x\}$.

  If $\{[i],[j]\}\cap \{[k],[\ell]\}=\emptyset$, let
  $\Phi_x(\langle k,\ell\rangle)=\langle k,\ell\rangle$.

  Otherwise, we can assume without loss of generality that $[j]=[\ell]$, i.e.
  $\ell=g\cdot j$. Replacing $k$ by $g\cdot i$, we have a triple of integers
  $(i,j,k)$ so that $\langle i,j\rangle=x$ and $\langle j,k\rangle =y$.
  In this case, let 
  \[ \Phi_x(\langle i,j\rangle)=\left\{\begin{array}{ll}
      \langle i,k\rangle &{\text{if $k$ is between $i$ and $j$.}} \\
      \langle j,k\rangle &{\text{otherwise.}}
      \end{array}\right.\]
    
    It is straightforward to check that $\Phi_x$ is a one-to-one
    correspondence, and that $(f_x)_{y'}=(f_y)_{x'}$, as claimed.
\end{proof}

\subsection{The pong algebra}
The pong algebra $\Pong{m}{k}$ is generated  over
$\Z[v_1,\dots,v_{m}]$ by lifted partial permutations on $k$ letters
in $\{1,\dots,m-1\}$. We extend $\weight(f)$ to a grading by
$({\OneHalf\Z})^{m}$, so that the grading of $v_i$ is the $i^{th}$ standard basis vector in $\Z^{m}$.

\begin{defn}
  \label{def:Mult}
  In view of Lemma~\ref{lem:ComposePermutations}
  (Equation~\eqref{eq:ComposeMult}) , if $(S,f)$ and $(T,g)$ are
  pong data with $T=\Liftf(\LiftS)/G_m$, then there is some monomial
  $v(f,g)\in\Z[v_1,\dots,v_{m}]$ with the property that 
  \begin{equation}
    \label{eq:PreMult}
    \weight(f)+\weight(g)=\weight(v(f,g)\cdot (\Liftg\circ f)).
  \end{equation}

  Similarly, if $\langle i,j\rangle\in\Cross(f)$, then
  there is a monomial $v(f,\langle i,j\rangle)$ with the property that
  \begin{equation}
    \label{eq:PreDeriv}
    \weight(f)=\weight (v(f,\langle i,j\rangle )\cdot f_{\langle i,j\rangle}).
  \end{equation}
\end{defn}

If $(S,f)$ and $(T,g)$ are two lifted partial permutations, we define
$\mu_2(g, f)=0$ unless $T=f(S)/G_m$ and
$\cross(\Liftg\circ f)=\cross(g)+\cross(f)$, in which case we let
\[ \mu_2(g, f) = v(f,g)\cdot \Liftg\circ f,\]
where $v(f,g)$ is as in Equation~\eqref{eq:PreMult}.

Given a partial permutation $(S,f)$, define the derivative $d(f)$ to
be given by
\[ \mu_1(S,f)=\sum_{\{c\in\Cross(f)\big|
\cross(f_{c})=\cross(f)-1\}}
v(f,c)\cdot f_{c},\]
where $v(f,c)$ as in Equation~\eqref{eq:PreDeriv}.

\begin{prop}
  \label{prop:PongIsAlg}
  The set $\Pong{m}{k}$ equipped with $\mu_1$ and $\mu_2$ is a differential graded algebra.
\end{prop}

\begin{proof}
  Associativity of $\mu_2$ follows from the associativity of 
  the composition of lifted partial permutations.

  The property that $\mu_1\circ \mu_1=0$ is an easy consequence of
  Lemma~\ref{lem:PreDSquared}.
  
  A crossing $\langle i,j\rangle$ in ${\widetilde f}\circ g$ is a
  $G_m$-orbit of $i<j$ with ${\widetilde f}\circ {\widetilde
    g}(i)>{\widetilde f}\circ {\widetilde g}(j)$. For each crossing, either
  ${\widetilde g}(i)<{\widetilde g}(j)$, in which case 
  $\langle {\widetilde g}(i),{\widetilde g}(j)\rangle\in \Cross(f)$
  or ${\widetilde g}(i)>{\widetilde g}(j)$, in which case
  $\langle i,j\rangle\in \Cross(g)$.
  In the first case,
  \[ ({\widetilde f}\circ {\widetilde g})_{\langle i,j\rangle}
  = {\widetilde f}_{\langle {\widetilde g}(i),{\widetilde g}(j)\rangle}\circ {\widetilde g};\]
  while in the second case,
  \[ ({\widetilde f}\circ {\widetilde g})_{\langle i,j\rangle}
  = {\widetilde f}\circ {\widetilde g}_{\langle i,j\rangle}.\]
  The Leibniz rule, $\mu_1\circ
  \mu_2=\mu_2\circ (\Id\otimes \mu_1+\mu_1\otimes\Id)$ follows readily.
\end{proof}

We abbreviate $\partial a=\mu_1(a)$ and $a \cdot b=\mu_2(a,b)$.

For example, we have
\begin{align*}\partial ((1,-2),(2,1))&=v_1\cdot ((1,3),(2,1)) + 
v_2 \cdot ((1,0),(2,3)) \\
\partial((1,3)(2,-3))&=((1,4),(2,-2))+v_2 ((1,-3),(2,3)).
\end{align*}
(The pong diagram for $((1,-2),(2,1))$ is shown in Figure~\ref{fig:PongPictures};
pong diagrams for the four other terms are illustrated in Figure~\ref{fig:DiffPongPictures}.)

Informally, multiplication corresponds to stacking strands diagrams on
top of each other, pulling the strands taut, and multiplying by a
monomial in the $v_i$, to preserve the weight vector.

\begin{figure}[ht]
\input{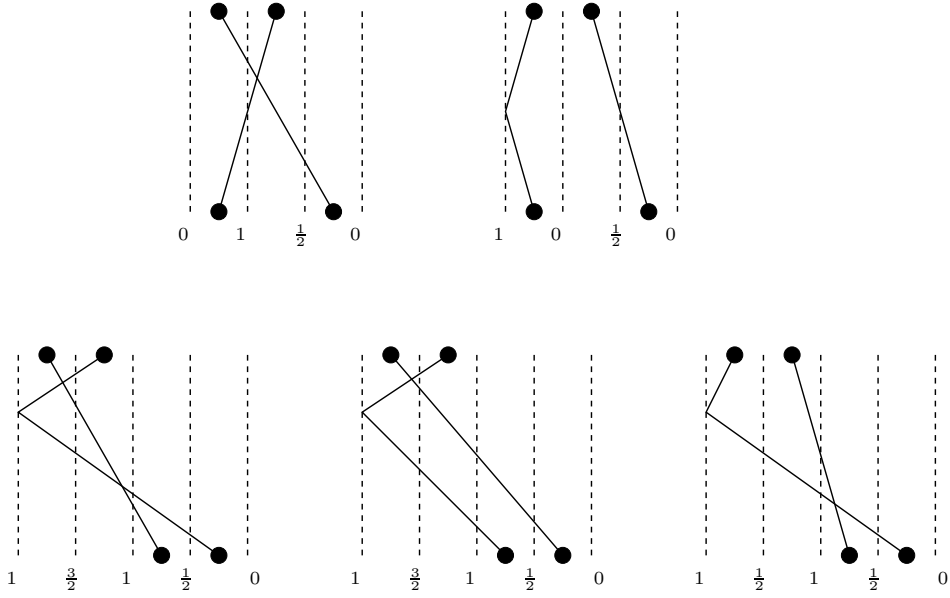}
\caption{\label{fig:DiffPongPictures} {\bf{Differentials of pong
      diagrams.}}.  The pong diagram from
  Figure~\ref{fig:PongPictures} has two crossings in it; the
  corresponding two terms in the differential are represented by the
  two pong diagrams in the top line: the first multiplied by $v_1$ and
  the second multiplied by  $v_2$.  The differential of pong diagram on the
  left of the second row has two non-zero terms in it (although it has
  $3$ crossings): the middle picture, and the
  right picture, multiplied by $v_2$.}
\end{figure}

When thinking of $(S,f)$ as an element of $a\in \Pong{m}{k}$, if $c\in
\Cross(f)$, define
\begin{equation}
  \label{eq:PartialCross}
 \partial_c(a)=\begin{cases}
v(f,c)\cdot (f_c,S)v(f,c)\cdot (f_c,S) & {\text{if $\cross(f_c)=\cross(f)-1$}} \\
0 & {\text{if $\cross(f_c)<\cross(f)-1$}}; 
\end{cases}
\end{equation}
so that 
\[ \partial (a)=\sum_{c\in \Cross(f)} \partial_c a. \]

\subsection{Alexander and Maslov gradings}
Idempotents in the pong algebra correspond to lifted partial
permutations which send $S\subset \Z$ to itself via the identity map.
In terms of their pong diagrams, these algebra elements are
collections of $k$ vertical segments. These $\binom{m-1}{k}$ elements
generate the idempotent subalgebra $\Ground$ of the pong algebra.

In Definition~\ref{def:LocalMult},
the weight $\weight$ of a lifted partial permutation $f$
was defined as 
a map from $(\OneHalf+ Z)/G_m$ to $\OneHalf \Z$. We prefer to think of it
henceforth as a vector $\vec{\weight}$, with components indexed by 
integers $i\in (1,\dots,m)$. With this normalization, $\weight_i$ corresponds
to the value of the function $\weight$ at $i-1/2$.

This vector $\vec\weight$ induces
a grading on $\Pong{m}{k}$ with values in $\OneHalf \Z^m$, with the
understanding that the element $v_i$ has grading $e_i$ (the $i^{th}$
basis vector in $\Z^m$).  This grading is additive under
multiplication; and differentiation preserves this grading.

Similarly, the function $\cross(f)$ on lifted partial permutations
induces a grading on $\Pong{m}{k}$ with values in $\Z$, with the
understanding that $v_i$ has grading $0$. The grading is additive
under multiplication, and differential drops it by $1$.  In the
language of Definition~\ref{def:MaslovAlexAlgebra}, $\Pong{m}{k}$ is a
Maslov/Alexander bigraded algebra.

Sometimes, we find it convenient to work with the following
renormalized $\Z$-grading $\Ngr$ on $\Pong{m}{k}$, specified by 
\begin{equation}
  \label{eq:Ngr}
  \Ngr(f)=\cross(f)-\left(2\sum_{i=1}^m \weight_i(f)\right).
\end{equation}

\begin{defn}
  \label{def:PurePong}
By analogy with $\Clg$, each element in $\Pong{m}{k}$ corresponding to
a lifted partial permutation is called a {\em pure algebra
  element}. Thus, the pong algebra is the free
$\Field[v_1,\dots,v_m]$-module generated by the pure algebra elements.
\end{defn}

\section{Generators of the pong algebra}
\label{sec:PongGen}

Idempotents in the pong algebra correspond to lifted partial
permutations which send $S\subset \Z$ to itself via the identity map.
In terms of their pong diagrams, these algebra elements are
collections of $k$ vertical segments. These $\binom{m-1}{k}$ elements
generate the idempotent subalgebra of the pong algebra.

Given a pair of integers $1\leq i<j\leq m-1$, let $\Xij$ denote the
algebra element consisting of two moving strands, which go from $i$ to
$j$ and $j$ to $i$, so that all positions $\ell$ with $i<\ell<j$ are occupied.
This is the element corresponding to the sum of all pong data $(S,f)$
with $|S|=k$, with the following further properties:
\begin{enumerate}[label=(X-\arabic*),ref=(X-\arabic*)]
\item $f(i)=j$, $f(j)=f(i)$
\item $\forall\ell\in S\setminus \{i,j\}$, $f(\ell)=\ell$.
\item \label{idempcondX}
  $\forall\ell$ with $i<\ell<j$, $\ell \in S$.
\end{enumerate}
We let ${\mathcal X}_{i,j}$ denote the algebra element
obtained by dropping Condition~\ref{idempcondX}.

We extend to $i=0$ or $j=m$ (but not both), so that $\Xoj$
has one strand from $j$ to itself with one left cusp and $X_{i,m}$
has one strand from $i$ to itself with one right cusp.
See Figure~\ref{fig:PictureOfX} for some examples.
The element $\Xoj$ corresponds to the sum of all pong data $(S,f)$
with $|S|=k$, with the following further properties:
\begin{enumerate}[label=(X-\arabic*),ref=(X-\arabic*)]
\item $f(j)=r_{1/2}(j)$
\item $\forall\ell\in S\setminus \{j\}$, $f(\ell)=\ell$.
\item \label{idempcondX0}$\forall\ell$ with $\ell<j$, $\ell\in S$.
\end{enumerate}
The element $\Xim$ is defined analogously, now using $r_{m+1/2}$ instead of $r_{1/2}$.
We define ${\mathcal X}_{0,j}$  and ${\mathcal X}_{i,m}$ again dropping
Condition~\ref{idempcondX0}

\begin{figure}[ht]
\input{PictureOfX.pstex_t}
\caption{\label{fig:PictureOfX} {\bf{The generators $\Xij$.}}
These pictures are in $\Pong{4}{2}$ and $\Pong{4}{3}$ respectively.}
\end{figure}

Given a pair of integers $1\leq i<j\leq m-1$, let $\Rij$ denote the algebra
element consisting of one moving strand which goes from $i$ to $j$,
so that all positions $\ell$ with $i<\ell<j$ are occupied. This is the element
corresponding to the sum of all pong data $(S,f)$ with $|S|=k$, with the following further properties:
\begin{enumerate}[label=(R-\arabic*),ref=(R-\arabic*)]
\item $f(i)=j$ 
\item $\forall\ell\in S\setminus\{i\}$, $f(\ell)=\ell$
\item\label{idempcondR} $\forall\ell$ with $i<\ell<j$, $\ell\in S$.
\end{enumerate}
Again, we have elements ${\mathcal R}_{i,j}$ obtained
by dropping Condition~\ref{idempcondR}.
We define elements $\Lji$, and ${\mathcal L}_{j,i}$,
analogously. (In that case, the second
condition above holds $\forall\ell\in S\setminus\{j\}$.)

\begin{defn}
  \label{def:Atomic}
  An {\em atomic} element in $\Pong{m}{k}$ is a non-zero element of the form
  $\Idemp{\x}\cdot \Xij$, $\Idemp{\x}\cdot \Rij$, and
  $\Idemp{\x}\cdot \Lji$.
\end{defn}

The terminology is justified by the following:

\begin{prop}
  \label{prop:PongGen}
  Every pure algebra element in $\Pong{m}{k}$ (i.e. those elements
  corresponding to lifted partial permutations) can be written as a
  product of idempotents and atomic generators.
\end{prop}

\begin{proof}
  We find it convenient to construct factorizations into ${\mathcal
    X}_{i,j}$, ${\mathcal R}_{i,j}$, and ${\mathcal L}_{j,i}$.  This
  suffices, since for any idempotent $\x$, the element
  $\Idemp{\x}\cdot {\mathcal X}_{i,j}$ can be written as a product of
  $\Idemp{\x}$ and elements of the form $X_{i',j'}$, $R_{i',j'}$ and
  $L_{i',j'}$. Similarly, $\Idemp{\x}\cdot {\mathcal R}_{i,j}$ can be
  written as a product of $\Idemp{\x}$ and elements of the
  form $R_{i',j'}$; and $\Idemp{\x}\cdot{\mathcal L}_{j,i}$ can be
  written as a product of $\Idemp{\x}$ and elements of the form
  $L_{j',i'}$.

  Suppose that $a=\Idemp{\x}\cdot a \cdot\Idemp{\y}$ for idempotent
  states $\x$ and $\y$. The factorization exists by induction on the
  total weight of $a$. Obviously, when the weight is zero, then $a$ is
  an idempotent. For the inductive step,  there are 3 cases.
  
  {\bf Case 1.} Suppose that for some $x_i\in\x$, the strand in $a$
  out of $x_i$ moves to the right. Choose $i$ maximal with this property.
  We divide this into two further
  subcases.

  {\bf Case 1a.} Suppose that there is some strand $j>i$
  with the property that the strand out of $x_j$ moves to the
  left. Choose $j$ minimal with this property.  
  If the strand out of $x_i$ reaches $x_j$ and the strand out of $x_j$ reaches
  $x_i$, we can  factor
  \[ a= \Idemp{\x}\cdot{\mathcal X}_{x_i,x_j} \cdot b;\]
  or if the strand out of $x_i$ terminates at some $y\in\x$ with $y<x_j$, 
  we can factor 
  \[ a = \Idemp{\x}\cdot{\mathcal R}_{x_i,y}\cdot b;\]
  or if the strand out of $x_j$ terminates at some $y\in\x$ with $x_i<y$,
  we can factor 
  \[ a = \Idemp{\x}\cdot{\mathcal L}_{x_j,y}\cdot b.\]

  {\bf Case 1b.} 
  If there is no strand $j>i$ so that the strand out of $x_j$ moves to the left,
  then there is a factorization
  $a=\Idemp{\x}\cdot{\mathcal X}_{x_i,m}\cdot b$ or a factorization of
  $a=\Idemp{\x}\cdot{\mathcal R}_{x_i,y}\cdot b$.
  The result now holds by induction on   the total weight.
  
  {\bf Case 2.} All strands in $a$ are stationary. This is clearly an idempotent.

  {\bf Case 3.} There is some strand in $a$ that starts out moving to the left,
  but none that starts moving to the right. Take the minimal $i$ so that there
  is a strand in $a$ starting at $x_i$ which moves to the left.
  If the strand bounces back past $x_i$, then we can factor
  \[ a= \Idemp{\x}\cdot{\mathcal X}_{0,x_i}\cdot b,\]
  or the strand terminates at some $y<x_i$, and we can factor
  \[ a =\Idemp{\x}\cdot{\mathcal L}_{x_i,y}\cdot b.\]
\end{proof}

We have the following characterization of the atomic generators:

\begin{prop}
  \label{prop:AtomicElements}
  For any non-zero homogeneous, non-idempotent element $a\in \Pong{m}{k}$, 
  $2\weight(a)-\Cr(a)\geq 1$, with equality precisely when $a$ is atomic.
\end{prop}

\begin{proof}
  Let $\gr(a)=2\weight(a)-\Cr(a)$.  Clearly, $\weight$ is additive
  under multiplication, in the sense that if $a$ and $b$ are pure
  algebra elements with $a\cdot b\neq 0$, then
  $\weight(a\cdot b)=\weight(a)+\weight(b)$. By construction, $\Cr$ is
  also additive under multiplication. It is easy to see that for all
  atomic elements $a$, $\gr(a)=1$. Thus, from
  Proposition~\ref{prop:PongGen}, it follows that that $\gr(a)\geq 1$
  for all pure algebra elements, with $\gr(a)=1$ if and only if $a$ is
  atomic. Having proved the proposition for pure algebra elements, the
  result for all algebra elements follows readily from the fact that
  for any homogenous element $a$ and $i\in\{1,\dots,m\}$.
  $\gr(v_i\cdot a)=2+\gr(a)$.
\end{proof}

We can describe the differential on the pong algebra in terms of its action on
atomic generators. Specifically:

\begin{lemma}
  \label{lem:dAtomic}
  The differentials of atomic generators are specified as follows.
  Given $i<j$, we have that
  \begin{align}
    d\Rgen{i}{j} &=\sum_{i<\ell<j}\Rgen{\ell}{j}\cdot\Rgen{i}{\ell} \label{eq:dRij}\\
    d\Lgen{j}{i} &=\sum_{i<\ell<j} \Lgen{\ell}{i}  \cdot \Lgen{j}{\ell}
    \label{eq:dLji}
  \end{align}
  Moreover, if $i\in \{0,\dots,m-1\}$, then 
  \[ d \Xgen{i}{i+1}=U_{i+1},\]
  while
  if $i+1<j$, then
  \[ d\Xgen{i}{j}=\sum_{i<\ell<j}\Xgen{\ell}{j}\cdot\Xgen{i}{\ell}
  + \Xgen{i}{\ell}\cdot\Xgen{\ell}{j}.\]
\end{lemma}

\begin{proof}
 This follows immediately from the definition of the boundary map.
\end{proof}

 \section{Homology of the pong algebra}
 \label{sec:HomologyPong}

 Consider the pong algebra $\Pong{m}{k}$, over $\Field[v_1,\dots,v_{m}]$.
 The aim of the present section is to compute its homology.

 Let
 \[ \Omega=\sum_{0<i<m} [{\mathcal X}_{0,i},{\mathcal X}_{i,m}].\]

 \begin{lemma}
   $\Omega$ is an element with degree $2k$, $\weight_i(\Omega)=1$ for
   all $i=1,\dots,m$, and  $d\Omega=0$.
 \end{lemma}

 \begin{proof}
   This is straightforward.
 \end{proof}

 \begin{thm}
   \label{thm:HomologyPong}
   Fix $k< m$. The homology of $H_*(\Pong{m}{k})$
   is isomorphic to a polynomial algebra in $\Omega$ over
   $\Clg(m,m-k-1)$; i.e.
   \[ H_*(\Pong{m}{k}) \cong \Clg(m,m-k-1)[\Omega].\]
 \end{thm}

 We will also sometimes consider $\PongP{m}{k}$, which is obtained by adjoining
 to $\Pong{m}{k}$ an element $X$ with $dX=\Omega$. 

 \subsection{A quotient}

 As a stepping-stone, we will find it convenient to consider the quotient
 $\Quot{m}{k}$ of $\Pong{m}{k}$ by the ideal
 generated by all the $v_i$, and the analogous 
 $\QuotP{m}{k}$ obtained by adjoining $X$ with $dX=\Omega$.
 We will typically abbreviate $Q=\Quot{m}{k}$ and $Q'=\QuotP{m}{k}$.

 For fixed weight vector $w=(w_1,\dots,w_m)$, let $Q(w)\subset Q$ and
 $Q'(w)\subset Q'$ be the subcomplexes with weight $w$; and let
 $Q(\x,\y;w)=\Idemp{\x}\cdot Q(w) \cdot \Idemp{\y}$ and
 $Q'(\x,\y;w)=\Idemp{\x}\cdot Q'(w)\cdot \Idemp{\y}$.
 Note that $Q(\x,\y;w)$ has a further grading 
 \[ Q_*(\x,\y;w)=\bigoplus_{d\geq 0} Q_d(\x,\y;w),\]
 where $d$ counts the number of crossings.

 To state the homology of $\Quot{m}{k}$, it helps to introduce some notions.

 \begin{defn}
   A {\em compatible triple} $(\x,\y,w)$ consists of 
   a vector $w=(w_1,\dots,w_m)\in (\OneHalf \Z)^{m}$ and
   two idempotent states $\x$ and $\y$, satisfying the condition that
   for each $i=1,\dots,m$,
   \[ 2 w_i+ \#(x_t<i)+\#(y_t<i)\] is an even
   integer. 
 \end{defn}

Clearly, $Q(\x,\y;w)=0$ unless $(\x,\y;w)$ is a compatible triple.
    
\begin{thm}
   \label{thm:HomologyQp}
   Given a compatible triple $(\x,\y;w)$, we have that
   $H_*(Q'(\x,\y;w))=0$ unless all of the following conditions holds:
   \begin{enumerate}[label=(Q-\arabic*),ref=(Q-\arabic*)]
     \item 
       \label{item:Interleaved}
       For each $s=1,\dots,k-1$, $\max(x_s,y_s)<\min(x_{s+1},y_{s+1})$
     \item 
       \label{item:BoundedWeight}
       Each weight $w_i$ satisfies the inequality
       $0\leq w_i \leq 1$.
   \end{enumerate}
   If all of the above conditions holds, then
   $H_*(Q'(\x,\y;w))\cong \Field$.
 \end{thm}

The proof will occupy the most of the rest of the present subsection.

 For cases where $Q'(\x,\y;w)$ has trivial homology, it will be useful
 to use the following notion:

 \begin{defn}
   \label{def:Excessive}
   Fix an integer $p$ with $1\leq p\leq m-1$.
   The algebra element $a$ is is called {\em (right-)excessive} at $p$
   if there are at least two non-left-moving lifted strands with
   inequivalent initial points that meet (i.e. cross, terminate at, or
   start at) $p$.  Equivalently,  let ${\mathcal E}_p$
   denote the set of (unlifted) strands 
   in the algebra element $a$ satisfying at least one of the following conditions:
   \begin{itemize}
   \item $s$ is stationary at $p$
   \item $s$ starts out at $p$, moving to the right
   \item $s$ terminates at $p$, moving from the left
   \item $s$ crosses $p$ from left to right.
   \end{itemize}
   Then $a$ is excessive at $p$ precisely when ${\mathcal E}_p$ has at least two elements in it.
 \end{defn}

\begin{figure}[ht]
\input{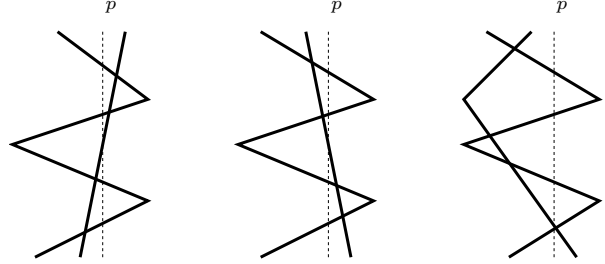}
\caption{\label{fig:Excessive} 
  The generator on the left is not (right)-excessive (at $p$);
  the other two are.
  }
\end{figure}

 \begin{lemma}
   \label{lem:Excessive}
   Fix $p\in\{1,\dots,m-1\}$. 
   If all the  pure algebra generators in $Q_d(\x,\y;w)$
   and $Q_{d-1}(\x,\y;w)$ are excessive, then
   $H_d(Q_*(\x,\y;w))=0$. 
 \end{lemma}

 \begin{proof}
   Fix an excessive generator $a$ in $Q'(\x,\y;w)$.  By hypothesis,
   there must be at least two right-moving lifted strands $s_1$ and
   $s_2$ with inequivalent initial points that are either stationary at
   $p$ or that meet $p$.  Choose the strand $s_1$ so that its initial
   point $i_1$ (which is $\leq p$) is maximized.  Choose $s_2$ so that
   its initial point $i_2$ is maximized among the remaining strands.

   If $s_1$ and $s_2$ cross, let $h(a)=0$.  Otherwise, let $h(a)=a'$ be
   the pure algebra element obtained replacing the strands $s_1$ and $s_2$,
   which start at $i_1$ and $i_2$ and end at $j_1$ and $j_2$,
   by a strand $s_1'$ from $i_1$ to $j_2$ and a strand $s_2'$ from $i_2$ to $j_1$.
   Note that $s_1'$ and $s_2'$ intersect once.

   Let $f(a)=i_1+i_2$.  Given two pure algebra elements, we write
   $a<b$ if $f(a)<f(b)$. Given two algebra elements $a$ and $b$ we
   write $a<b$ if we can write $a=\sum_{i} a_i$ and $b=\sum_{i} b_i$ as
   sums of pure algebra elements with $f(a_i)<f(b_j)$ for all $i,j$.

\begin{figure}[ht]
\input{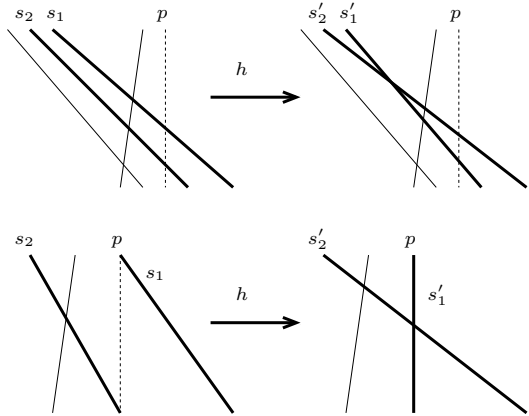}
\caption{\label{fig:HomotopyOperator} {\bf{Homotopy operator.}}
  The dashed line represents position $p$. The bottom line represents a degenerate
  case, where the strand $s_1'$ is stationary.
  }
\end{figure}

   We claim that if $a\in Q_d(\x,\y;w)$, 
   \begin{equation}
     \label{eq:HomotopyFormula}
     \partial h(a) + h\partial(a) + a > a
   \end{equation}
   Note that $f$ is an integer-valued function.
   Since $w$ is fixed, $f$ is bounded below; also, $f$ is bounded above by
   $2p-1$.
   It follows that 
   \[ \Id + \partial h + h \partial\colon Q_d(\x,\y;w)\to Q_d(\x,\y;w) \]
   is nilpotent.
   On the other hand, this nilpotent map is an
   isomorphism on homology, since it is chain homotopic to the
   identity; thus, the homology is trivial.

   We verify Equation~\eqref{eq:HomotopyFormula} as follows.
   Suppose that $a$ has no crossing between $s_1$ and $s_2$, so that
   $a'=h(a)$ is non-zero.  Terms in $\partial h(a)$ correspond to
   various resolutions of the crossings in $a'$; terms in
   $h \partial(a)$ correpond to resolutions of the crossings in $a$. 

\begin{figure}[ht]
\input{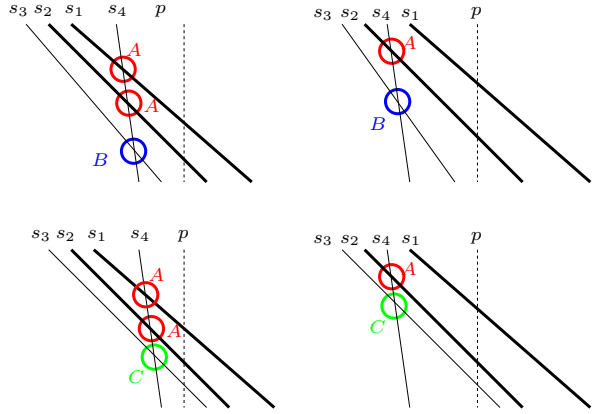}
\caption{\label{fig:CrossingTypes} {\bf{Types of crossings.}}
  The dashed lines represent the position $p$}
\end{figure}

   We partition the set of crossings in $a$ into the following types:
   \begin{itemize}
   \item the set $A$ of crossings in $a$ that involve one of the
     strands $s_1$ and $s_2$; 
   \item the set $B$  of crossings in $a$ do not involve
     either $s_1$ or $s_2$ and which leave $i_1$ and $i_2$
     unchanged; 
   \item the set $C$ of crossings in
     $a$ that do not involve either $s_1$ or $s_2$ but which change $i_1$
     and $i_2$. Such a crossing occurs between strands $s_3$ and $s_4$ with the following properties:
     $s_3$ crosses $p$ and its initial point is smaller than $i_2$; and $s_4$ 
     does not cross $p$ and its initial point is greater than $i_2$.
   \end{itemize}

   See Figure~\ref{fig:HomotopyOperator} for illustrations.  Let $A'$,
   $B'$, and $C'$ be analogous crossings in $a'$, with the
   understanding that $A'$ does not include the distinguished crossing
   $q_0'$ between $s_1'$ and $s_2'$. The resolution of $a'$ at $q_0$, of course, gives back $a$;
   i.e. $\partial_{q_0'} h(a)=a$.
   
   There is a bijection between $A$ and
   $A'$ which gives a cancellation of terms in $\partial h(a)$
   involving $s_1$ or $s_2$ with corresponding terms in $A'$. We
   construct this bijection $\phi\colon A \to A'$ explicitly, as
   follows. Suppose there is some strand $s_3$ which meets $s_1\cup
   s_2$, so that the intersections $s_3\cap (s_1\cup s_2)$ are of type
   $A$. There are two subcases. Either $s_3\cap (s_1\cup s_2)$
   consists of one point, call it $q$; in which case $s_3\cap
   (s_1'\cap s_2')$ consists of one point, and we define $\phi(q)$ to
   be that point. Or, $s_3\cap (s_1\cup s_2)$ consists of two points,
   denoted $q_1$ and $q_2$; and also $s_3\cap (s_1\cap s_2)$ conists
   of two points, as well, which we denote $q_1'$ and $q_2'$. We can
   label $\{q_1, q_2\}$ so that the resolution of $a$ at $q_1$ is in
   the kernel of $h$, while the resolution at $q_2$ is
   not. Similarly, we can label $\{q_1',q_2'\}$ so that the resolution
   of $h(a)$ at $q_1'$ contains a double-crossing, but the resolution
   at $q_2'$ does not. See Figure~\ref{fig:CrossingCancellation}.  Clearly,
   \[ \partial_{\phi(q)} \circ h(a)=h\circ \partial_q(a),\]
   where $\partial_q$ is as in Equation~\eqref{eq:PartialCross}.
   (Indeed, for $q=q_1$, both terms vanish.)

\begin{figure}[ht]
\input{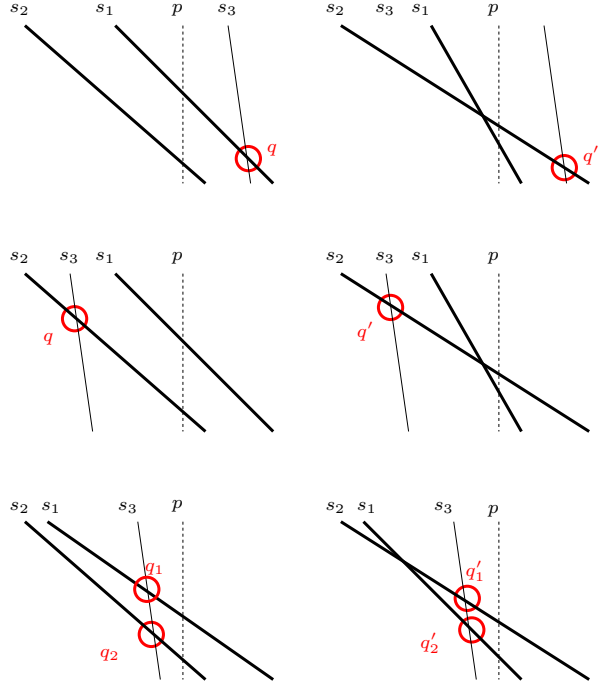}
\caption{\label{fig:CrossingCancellation}
  On the left, we have intersection points of type $A$, and on the right we
  have their corresponding intersection points in $A'$.
  The top two rows represent cases where $|s_3\cap (s_1\cup s_2)|=1$,
  while the third represents the case where $|s_3\cap (s_1\cup s_2)|=2$.}
\end{figure}

There is a similar one-to-one correspondence between elements in $B$
and those in $B'$, giving a further cancellation of terms
$\partial_{q'}h(a)$ with those in $h\circ \partial_{q}(a)$ (with $q\in B$).

The remaining terms in $\partial\circ h(a)$ and $h\circ\partial (a)$ all correspond
to crossings in $C$ or $C'$. 

   Note that the only way a resolution at a crossing between strands
   other than $s_1$ and $s_2$ can change $i_1$ or $i_2$ is if that resolution
   properly increases $f$; i.e. all the terms $c$ in $\partial h(a) +
   h\partial(a)$ corresponding to the resolutions at $C$ or $C'$ have
   $f(c)>f(a)=f(a')$. This completes the verification of
   Equation~\eqref{eq:HomotopyFormula} when $s_1$ and $s_2$ do not
   cross. 

   The formula where they do cross follows similarly.
 \end{proof}

 \begin{lemma}
   \label{lem:Interleaved}
   If Condition~\ref{item:Interleaved} from
   Theorem~\ref{thm:HomologyQp} is violated, then $H(Q(\x,\y;w))=0$.
 \end{lemma}

 \begin{proof}
   Condition~\ref{item:Interleaved} is equivalent to the following
   condition for all $p=1,\dots,m-1$:
   \begin{equation}
     \label{eq:Interleaved2}
     \Big|\#\{x\in\x \big|x\leq p\}-\#\{y\in\y
     \big|y\leq p\}\Big |
     \leq 
     \begin{cases}
       1 &{\text{if $p\not\in\x\cap\y$}} \\
       0 &{\text{if $p\in\x\cap \y$}}
     \end{cases}
   \end{equation}

   Let ${\mathcal E}_p$ denote the set of 
   (unlifted) strands as in Definition~\ref{def:Excessive}.
   It is elementary to verify that
   \[ 
     \#\{x\in\x \big|x\leq p\}-\#\{y\in\y
     \big|y< p\}\leq |{\mathcal E}|.\]
     Thus, Lemma~\ref{lem:Excessive}, shows that 
     if 
     $H_*(Q(\x,\y;w))\neq 0$, then
     \begin{align*}
     \{x\in \x \big|x \leq p\} &-\#\{y\in\y\big|y\leq p\}
     +\begin{cases}
       1 &{\text{if $p\in\x\cap\y$}} \\
       0 &{\text{otherwise}}
       \end{cases}\Bigg\}  \\
       &\leq      \#\{x\in\x \big|x\leq p\}-\#\{y\in\y \big|y< p\} \leq 1.
       \end{align*}
       There is an isomorphism of chain complexes $Q(\x,\y;w)\cong
       Q(r(\x),r(\y),r(w))$, where $r$ is reflection of the interval.
       This gives the other bound
       \[ 
     -\{x\in \x \big|x \leq p\} +\#\{y\in\y\big|y\leq p\}
     +\begin{cases}
       1 &{\text{if $p\in\x\cap\y$}} \\
       0 &{\text{otherwise}} \leq 1
       \end{cases}\Bigg\}\leq 1
       \]
       needed to complete  the verification of Equation~\eqref{eq:Interleaved2}.
 \end{proof}

 Let $r_p$ resp. $\ell_p$ denote the number of lifted right-moving
 resp. left-moving lifted strands that cross $p+\frac{1}{2}$.

 \begin{lemma}
   \label{lem:LeftRight}
   Suppose that $a\in Q(\x,\y;w)$ represents a non-trivial homology
   class in $H_*(Q(\x,\y;w))$. Then, $a$ is a sum of pure algebra
   elements with $|r_p-w_p|\leq \OneHalf$ for all $p=1,\dots,m$.
 \end{lemma}

 \begin{proof}
   Clearly, $r_p-\ell_p=\#(x\in \x <p)-\#(y\in \y <p)$.
   Lemma~\ref{lem:Interleaved} implies that if $H_*(Q(\x,\y;w))\neq 0$, then 
   $r_p-\ell_p\leq 1$ for all $p$. The stated inequality follows from the fact that $r_p+\ell_p=2 w_p$.
 \end{proof}

 \begin{lemma}
   \label{lem:HomologyVanish}
   Let $w$ be a non-constant vector with the following properties:
   \begin{itemize}
     \item there is some integer
       $c>0$ so that $w_i>c$ for some $i$.
     \item $H_d(Q(w))\neq 0$.
   \end{itemize}
   Then, for all $i$,
   $c\leq w_i$, and $d\geq 2 k c$.
 \end{lemma}

 \begin{proof}
   Let $a$ be any representative of $Q_d(\x,\y;w)$.  Choose $p$ so that
   exactly one of $w_{p-1}$ or $w_p$ is greater than $c$.  By symmetry,
   we can assume that there is a right-moving lifted strand $s$ that
   either starts (when $w_p>w_{p-1}$) or terminates at $p$.  Moreover,
   by Lemma~\ref{lem:LeftRight}, there are at least $c$ additional
   lifted strands that cross position $p$ from the left.  If $a$ fails
   to be excessive, which it must by Lemma~\ref{lem:Excessive}, all $c$
   of those lifted strands, and the additional strand $s$, must have
   equivalent starting points; i.e. there is an unlifted, right-moving
   strand that either starts or ends at $p$, and that covers the entire
   interval $[1,m]$ with multiplicity at least $c$. The bounds on $d$
   and $c$ follow readily.
 \end{proof}

\begin{defn}
  \label{def:CanonicalCycle}
  Let $(\x,\y,w)$ be a compatible triple, and suppose that there is some
  integer $c$ so that $c\leq w_i\leq c+1$ for all $i=1,\dots,m$. 
  The {\em associated canonical cycle} is the element $z=z(\x,\y;w)\in Q(\x,\y;w)$ characterized by the following properties:
  \begin{itemize}
  \item $\weight_i(z)=w_i-c$
  \item If $s$ is any right-moving strand in $z$, then the terminal position
    of $s$ is not the initial position of any other right-moving strand in $z$.
  \end{itemize}
  An element $z\in Q(\x,\y)$ is called a {\em canonical cycle} if there is a vector $w$ 
  for which $z$ is the canonical cycle associated to $(\x,\y;w)$.
\end{defn}

For example, if all the components of $w_i$ are integers, then $z$ is
the product of those ${\mathcal X}_{x_i,x_{i+1}}$ for which  $w_{x_i}=c+1$, with the factors
are arranged so that $i$ is increasing. See also Figure~\ref{fig:CanonicalCycles}.
\begin{figure}[ht]
\input{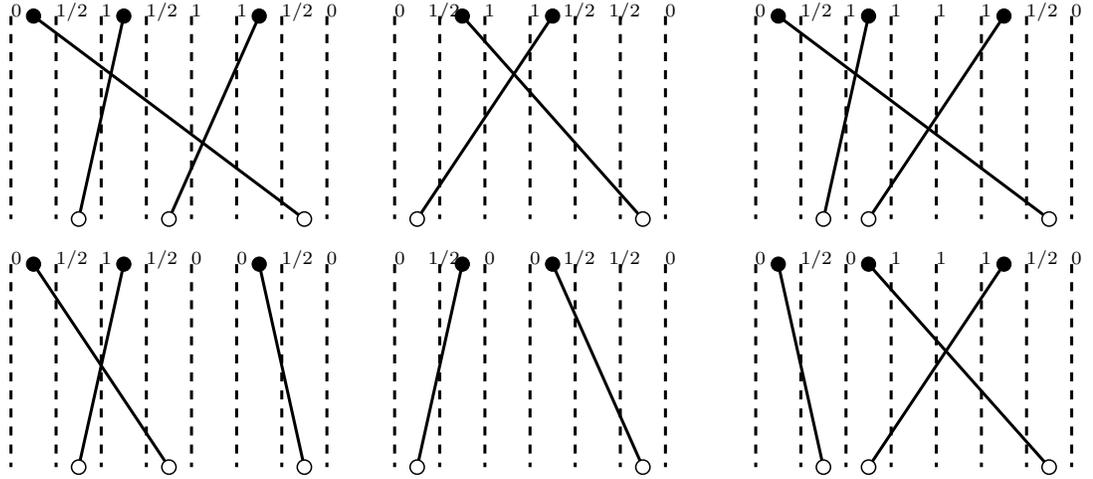}
\caption{\label{fig:CanonicalCycles} {\bf{Canonical cycles.}}
  Here are six canonical cycles; the multiplicities are indicated at the dotted lines.
  (Note that there is another cycle homologous to the third on the top line with the same multiplicities and idempotents.) Cycles in the second row are obtained as resolutions of crossings on the first row.}
\end{figure}

 \begin{lemma}
   \label{lem:TwoValues}
   Let $(\x,\y;w)$ be a compatible triple. Suppose 
   that there is an integer $c$ with the property that
   $c\leq w_i\leq c+1$ for all $i$ and $w_i$ is not constant.
   Let $z=z(\x,\y;w)$ be the associated canonical cycle.
   $Q'(\x,\y;w)$ is quasi-isomorphic  to the subcomplex
   $(z Q')(\x,\y;w)=z\cdot Q'(\y,\y;(c,\dots,c))\subset Q'(\x,\y;w)$.
   In particular, if $w$ is a non-constant vector
   with $w_i\leq 1$ for all $i$, 
   then $H(Q'(\x,\y;w))$ is one-dimensional.
 \end{lemma}

 \begin{proof}
   Fix some $p$ so that $w_{p-1}=w_p=c+1$ and $p\in \x\cap \y$, 
   consider the subspace $L_p\subset Q'(\x,\y;w)$ generated by 
   pure algebra elements $a$ with the property that the 
   strand out of $p$ is left-moving. It is
   easy to see that $L_p\subset Q'(\x,\y;w)$ is a subcomplex.

   We claim that the inclusion map $L_p \to Q'(\x,\y;w)$ is a
   chain homotopy equivalence. To verify this, consider the map
   \[ h_p\colon Q'(\x,\y;w)\to Q'(\x,\y,w) \]
   which is defined to vanish if $a$ is not excessive; otherwise,
   $h_p(a)$ is as defined in the proof of Lemma~\ref{lem:Excessive}.
   Observe that if the strand out of $p$ in $a$ is left-moving, then
   the strand out of $p$ in $h_p(a)$ is also left-moving; i.e.
   $h_p(L_p)\subset L_p$.  Observe also that the quotient
   $Q'(\x,\y;w)/L_p$ is generated by excessive elements: the strand
   out of $p$, the strand into $p$, and the $c$ additional strands
   that cross position $p$ cannot all coincide, for otherwise, we
   would get an inequality $w_i\geq c+1$, violating our hypotheses.
   It follows now from the proof of Lemma~\ref{lem:Excessive} that
   $H(Q'(\x,\y;w)/L_p)=0$, verifying that $L_p\to Q'(\x,\y;w)$ is a
   quasi-isomorphism.

   More generally, let
   \[ S=\{p\in \{1,\dots,m\}\big| w_{p-1}=w_p=c+1~\text{and}~p\in\x\cap \y\}.\]
   Given any $T\subset S$, let $L_T\subset Q'(w)$ denote the
   intersection $\cap_{t\in T} L_t$.  We claim that for any $p\in
   S\setminus T$, $h_p$ maps $L_T$ into $L_T$.  (The homotopy operator
   $h$, introduces a crossing between two right-moving strands in the
   lift; These strands could start out left-moving or right-moving at
   their initial points in the projection; but the homotopy operator
   never changes left-moving to right-moving or vice versa. In a
   special case, for the strand through $p$, it might turn a
   right-moving strands into a stationary one; but here we are
   specifically considering $p\in S\setminus T$.)  
   It follows that the inclusion $L_{\{p\}\cup T} \to L_T$ is a
   quasi-isomorphism.  Thus, by induction, $L_S\subset Q'(\x,\y;w)$ is
   a quasi-isomorphism.

   It remains to show that $L_S= (z Q')(w)$. This can be seen as
   follows.  Choose $p$ minimal so that $w_p=c+1$ and $p\in \x$. If $p=x_1$, the strand
   out of $p$ moves to the left, so we can factor $a=z_0\cdot a'$,
   where $z_0={\mathcal X}_{0,p}$.  Assume that $p=x_i>x_1$.
   Let $x_{i+1}\in\x$ be the successor of $p=x_i\in \x$,
   and choose the maximal $y_j\in\y$ so that $x_i<y_j$. If
   $y_j<x_{i+1}$, we can write $a=z_0\cdot a'$, where $z_0$ is the
   algebra element associated to a strand from $x_i$ to $y_j$.  If
   $y_j\geq x_{i+1}$, the strand out of $q=x_{i+1}$ necessarily moves
   to the left (because either $w_{q}>w_{q-1}$; or if $w_q=w_{q-1}$,
   but in that case
   $a\in L_q$). Thus, we can factor $a=z_0\cdot a'$, where $z_0$ is the
  algebra element consisting of two strands, one of which goes (to the
  right) from $x_i$ to $x_{i+1}$, and the other leaves $x_{i+1}$,
  going to the left, and terminating at $y_{j-1}$. In cases where the
  weight vector of $w'=\weight(a')$ is constant, we have the desired
  factorization. Otherwise, we can use the induction on the support of
  $z$, noting that $z(a)=z_0\cdot z(a')$.
\end{proof}

\begin{lemma}
  \label{lem:TwoValues2}
  Let $w$ be a non-constant vector with the property that $\max
  w_i> 1$, then $H(Q'(\x,\y;w))=0$.
\end{lemma}

\begin{proof}
  By Lemma~\ref{lem:HomologyVanish}, we can assume that there is some integer $c$
  so that $c\leq w_i\leq c+1$; by hypothesis, $c>0$.
  Let $z=z(\x,\y,w)$ be the canonical cycle.

  We wish to analyze the chain complex 
  \[ (zQ')(\x,\y;w)=z\cdot
  Q'(\y,\y;w-w(z)).\]  Note that
  there is an integer $c$ so that for all $i$, $(w-w(z))_i=c$.
  Let $Q^0$ denote the algebra $\Quot{k+1}{k}$.
  There is an identification of $Q(\y,\y;w-w(z))$ with the portion of
  $Q^0$ whose support is the constant vector $c$ with $k+1$ coordinates,
  which we write $Q^0(c)$.

  Indeed,
  let 
  \[ L=\{i\big| i\in\y, \weight_i(z)>0\} \qquad
  {\text{and}}\qquad
  R =\{i\big| i\in\y, \weight_{i-1}(z)>0, \weight_i(z)=0\}.\]
  At each position $i\in L$, the strand out of 
  $z$  is left-moving; at each position $i\in R$, the strand out
  of $z$ is right-moving.

  We claim that $(z Q')(\x,\y;w)$ is the the quotient complex of
  $Q^0$ by the differential right ideal generated by
  $X_{i,t}$ with $i\in L$ and $X_{s,j}$ with $j\in R$. i.e. 
  \begin{equation}
    \label{eq:IdentifyQ}
    z\cdot Q(\y,\y)\cong \left(\frac{Q^0}{\sum_{i\in L}X_{i,t}Q^0
      + \sum_{j\in R}X_{s,j}Q^0}\right)_{c}.
  \end{equation}

  Let $d_0$ be the dimension (number of crossings) in $z$.
  It follows from Equation~\eqref{eq:IdentifyQ} that
  $(zQ(w))_d=0$ if $d>d_0+2 c k$; and $(zQ(w))_{d_0+2ck}$ is
  spanned by the elements $\{ z \cdot (X_{0,i}\cdot X_{i,k+1})^c,
  z\cdot (X_{i,k+1}\cdot X_{0,i})^c\}_{0<i<k+1}$, with the understanding that
  \[ z\cdot (X_{0,i}\cdot X_{i,k+1})^c = 0 \]
  exactly when $i\in R$; and 
  \[ z\cdot (X_{i,k+1} \cdot X_{0,i})^c = 0 \]
  exactly when $i\in L$.
  It follows that $H((zQ'(w))_{d})=0$ for all $d\geq d_0+2ck$.

  It remains to show that $H((zQ(w))_d)=0$ for all $d<d_0+ck$.
  Since $w$ is non-constant, there is some $p\in\x\cup\y$ with the property that
  \begin{equation}
    \label{eq:pChange}
    \min\{w_{p-1},w_p\}\leq c <\max\{w_{p-1},w_p\}.
  \end{equation}
  Our verification is subdivided into four (very similar) cases, according to whether 
  or not $p\in \x$ and whether or not $w_{p-1}<w_p$.

  Suppose first that there is some
  $p\in\x$ so that $w_{p}>c$, $w_{p-1}< w_p$. There are $c$
  right-moving strands that cross position $p$; and an additional
  right-moving strand that leaves position $p$.  If all the elements
  of $Q_d(\x,\y;w)$ are excessive at $p$, then $H_d(Q_*(\x,\y;w))=0$ by
  Lemma~\ref{lem:Excessive}.  If these elements are not all excessive,
  the (unlifted) strand leaving $p$ covers all of the interval with
  multiplicity at least $c$. Thus, all $k$ of the strands leaving $z$
  must cross this moving strand with multiplicity at least $c$,
  giving the inequality $d\geq d_0+2ck$, provided that there is some $p\in\x$ with
  $w_p>c$ and $w_{p-1}< w_p$.

  The same argument applies when there is some $p\in \x$ so that $w_{p-1}> w_p$ and
  $w_{p-1}>c$, only now using the strand entering $p$ (rather than othe one leaving $p$).

  Finally, there is no $p\in \x$ so that 
  Equation~\eqref{eq:pChange} holds, there must be such a $p\in \y$; and the result follows from symmetry.
\end{proof}

\begin{lemma}
  \label{lem:ConstantVector}
  Suppose that there is a $c$ so that $w_i=c$ for all $i$.
  Fix any idempotent state $\x$.
  Then, \[ H(Q'(\x,\x;w))=\begin{cases}
    \Field &{\text{if $c=0$  or $1$}} \\
    0 &{\text{otherwise}}.
    \end{cases}     \]
\end{lemma}

\begin{proof}
  The case $c=0$ is obvious.  Since $w$ is constant, it follows that
  $\x=\y$.  Let $p=x_1\in \x$ be the smallest element, and let $L\subset
  Q'=Q'(\x,\y;w)$ be the subcomplex generated by all pure algebra
  elements for which the strand out of $p$ is left-moving.  Consider
  the short exact sequence
  \[
  \begin{CD}
    0@>>> L(w)@>>> Q(w)@>>> (Q/L)(w) @>>> 0.
    \end{CD}
  \]
  Observe that $(Q/L)(w)$ is generated by excessive elements and the single
  non-excessive element $({\mathcal X}_{p,m}\cdot {\mathcal X}_{0,p})^c$
  (which in fact is the image of $\Omega$ under the quotent map).  It
  follows at once that $H_*((Q/L)(w))$ is one-dimensional, generated by
  $\Omega^c$.  Thus, $H((Q'/L)(w))=0$. 
  
  As in the proof of Lemma~\ref{lem:TwoValues}, $L= {\mathcal
    X}_{0,p}\cdot Q$. We claim if $w'$ is obtained from $w$ by
  subtracting $1$ from its first $p$ coordinates, then the map
  $Q'(w')\to Q'(w)$ given by $a\mapsto {\mathcal X}_{0,p} \cdot a$ is
  an isomorphism of chain complexes.
  Lemma~\ref{lem:TwoValues2} gives the stated result when $c>1$;
  while Lemma~\ref{lem:TwoValues} gives the result when $c=1$.
  Note that in this case,
  $z(\x,\x;w')={\mathcal X}_{x_1,x_2}\cdots {\mathcal X}_{x_k,m}$,
  so the non-trivial homology class in $Q'(w)$ is represented by
  ${\mathcal X}_{0,x_1}\cdots {\mathcal X}_{x_k,m}$.
\end{proof}

\begin{proof}[of Theorem~\ref{thm:HomologyQp}]
  Condition~\ref{item:Interleaved} was verified in
  Lemma~\ref{lem:Interleaved}; Condition~\ref{item:BoundedWeight} is a
  combination of Lemmas~\ref{lem:HomologyVanish} (when the vector is
  non-constant) and Lemma~\ref{lem:ConstantVector} (when the vector is
  constant).
  
  The fact that the homology has rank one in the remaining cases
  follows from Lemma~\ref{lem:ConstantVector} and~\ref{lem:ConstantVector}.
\end{proof}

Theorem~\ref{thm:HomologyQp} has the following quick corollary:

\begin{cor}
  \label{cor:QmSpecial}
  There is an isomorphism of algebras
  \begin{equation}
    \label{eq:QmSpecial}
    H_*(\Quot{m}{m-1})\cong \Field[X_1,\dots,X_{m},\Omega]/(X_i^2=0)_{i=1}^m.
  \end{equation}
\end{cor}

\begin{proof}
  Abbreviate $Q=\Quot{m}{m-1}$ and $Q'=\QuotP{m}{m-1}$.
  Clearly, $X_i$ are homologically non-trivial elements in
  $H_*(Q')$.  They satisfy $X_i^2=0$ (since they do so on
  the chain level).  They also satisfy $[X_i][X_j] = [X_j]\cdot [X_i]$. When
  $j\neq i\pm 1$, the corresponding relation holds on the chain level; 
  $[X_i][X_{i+1}]=[X_{i+1}][X_i]$, since 
  \[ d X_{i-1,i+1}= X_i \cdot X_{i+1} + X_{i+1} X_i.\]
  Thus, there is a tautological map of rings
  \[\tau'\colon \Field[X_1,\dots,X_{m}]/(X_i^2=0)_{i=1}^m\to    H_*(Q').\]
  
  If $w$ is any weight vector with $\max(w_i)\leq 1$, then $Q'(w)$ is
  one-dimensional, by Theorem~\ref{thm:HomologyQp}.  (In the notation,
  we suppress idempotents now, since there is only one.)  In fact, by
  Lemma~\ref{lem:TwoValues} when $w$ is non-constant, and
  Lemma~\ref{lem:ConstantVector} when $w$ is constant, $H_*(Q'(w))$ is
  generated by the canonical cycle from
  Definition~\ref{def:CanonicalCycle}. Let $w$ be a weight vector with
  $w_i\leq 1$. Let $I=\{i\in \{1,\dots,m\}\big| w_i=1\}$. Let $\{n_1,\dots,n_\ell\}$ be the sequence of elements of $I$, arranged in decreasing order.
  It is straightforward to see that 
  $\prod_{i=1}^{\ell} X_{n_i}$ is the canonical cycle of $w$.

  Let $q\colon Q\to Q'$ denote the natural inclusion map. By
  construction, the cycles from $H_*(Q')$ come from cycles in
  $H_*(Q)$; thus $q$ induces a surjection on homology. Thus, the long
  exact sequence of the mapping cone becomes the short exact sequence
  \[\begin{CD}
  0@>>> H(Q)@>{\Omega}>> H(Q)@>>> H(Q')@>>> 0,
  \end{CD}\]
  The corollary follows readily.
\end{proof}

\subsection{The Pong algebra}

Given $\x$ and $\y$, let $Z(\x,\y)$ denote the set of canonical cycles $z$
for compatible triples $z(\x,\y;w)$, for various choices of $w$.
Each canonical cycle is determined by its weight vector, which we can take
to have $0\leq w_i\leq 1$ for all $i$.

\begin{defn}
  Let $(\x,\y;w)$ be a compatible triple, and suppose that $\x$ and
  $\y$ satisfy
  \begin{equation}
    \label{eq:Interleaved}
    \max(x_s,y_s)< \min(x_{s+1},y_{s+1}). 
  \end{equation}
  The {\em selected weight vector of $(\x,\y;w)$}, denoted $s(z)\in \Z^{k+1}$, is given
  by the components
  \[ (w_0,w_{\max(x_1,y_1)}, w_{\max(x_2,y_2)},\dots w_{\max(x_k,y_k)},w_m). \]
\end{defn}

\begin{lemma}
  Fix two idempotent states $\x$ and $\y$ satisfying
  Equation~\eqref{eq:Interleaved}, the map from $Z(\x,\y)\to \Z^{k+1}$
  that maps $z=z(\x,\y;w)\in Z(\x,\y)$ to the selected 
  weight vector of $w$ identifies $Z(\x,\y)$
  with the set of vectors $\xi\in \Z^{k+1}$ with
  $\xi_i\in\{0,1\}$.
  Under this correspondence, the dimension of $z$ is given by
  $\xi_i \sigma_i$.
\end{lemma}

\begin{proof}
  This is mostly straightforward. For the identification of the
  degree, note that the crossings in a canonical cycle $z=z(\x,\y;w)$ are in
  one-to-one correspondence with the non-zero components of the
  associated selected weight vector.
\end{proof}

For $s=1,\dots,k+1$, let
\begin{equation}
  \label{eq:DefVs}
  V_s = \prod_{i=\max(x_{s-1},y_{s-1})}^{\min(x_{s},y_{s})-1} v_i\in\Field[v_1,\dots,v_m],
\end{equation}
with the conventions that $x_0=y_0=1$ and $x_{k+1}=y_{k+1}=m$.

\begin{lemma}
  \label{lem:ModelComplex}
  The $\Field[v_1,\dots,v_{m}]$-submodule of $P'(\x,\y)$ generated by $Z(\x,\y)$
  is a subcomplex, which is quasi-isomorphic to $P'(\x,\y)$. Moreover,
  its differential is specified
  by 
  \[ \partial \xi = \sum_{\{i\big|s_i=i\}} V_i \cdot (\xi-e_i),\]
  where $e_i\in \Z^{k+1}$ is the $i^{th}$ basis vector.
\end{lemma}

\begin{proof}
  Again, we use the fact that crossings in a canonical cycle $z$ are
  in one-to-one correspondence with the non-zero components of the
  selected weight vector. It is also clear that resolving any of those
  crossings also gives a canonical cycle, times $U_i$.
\end{proof}

\begin{thm}
  \label{thm:HomologyPong2}
  The homology of $P'(\x,\y;w)$ is zero unless for 
  each $s=1,\dots,k-1$, 
  \begin{equation}
    \label{eq:InterleavedP}
    \max(x_s,y_s)<\min(x_{s+1},y_{s+1});
  \end{equation}
  and in that case,
  \begin{equation}
    \label{eq:IdentifyPongHomology}
    H_*(P'(\x,\y;w))\cong \frac{\Field[v_1,\dots,v_m]}{(V_1,\dots,V_{k+1})}.
  \end{equation}
\end{thm}

\begin{proof}
  If $\max(x_s,y_s)\geq \min(x_{s+1},y_{s+1})$ for some $s$, then
  $H(Q'(\x,\y))=0$ by Theorem~\ref{thm:HomologyQp}. Now,
  $H(P'(\x,\y))$ has a filtration by $U_i$-powers, whose associated
  graded object is copies of $Q'(\x,\y)$; so if the latter has trivial
  homology, then so does $H(P'(\x,\y))$.

  Otherwise, 
  the verification of Equation~\eqref{eq:IdentifyPongHomology}
  follows quickly from  the model chain complex provided by
  Lemma~\ref{lem:ModelComplex}.
\end{proof}

\begin{proof}[of Theorem~\ref{thm:HomologyPong}]
  If $\x$ is an idempotent state (with $k$ components), let
  \[ \x'=\{1,\dots,m-1\}\setminus \x.\]  The
  condition that Equation~\eqref{eq:InterleavedP} holds for $s=1,\dots
  m-1$ is equivalent to the condition that $\x'$ and $\y'$ are not
  too far, in the sense of Definition~\ref{def:TooFar}.
  
  Moreover, the quotient by the ideal generated by the
  $(V_1,\dots,V_s)$ appearing in Theorem~\ref{thm:HomologyPong2}
  corresponds exactly to the quotient described in
  Proposition~\ref{prop:Ideal}, which gives
  $\Idemp{\x'}\cdot \Clg(m,m-k-1)\cdot \Idemp{\y'}$. Thus, the theorem follows from
  Theorem~\ref{thm:HomologyPong2}.
\end{proof}

\section{A type $DD$ bimodule}
\label{sec:DDmod}

In Section~\ref{sec:Duality}, we establish a Koszul duality,
identifying the cobar algebra of $\Clg(m,k)$ with the quotient of the
pong algebra $\Quot{m}{k}$ (from Equation~\eqref{eq:DefQuot}).

This duality is achieved by tensoring a type $AA$ bimodule defined
using holomorphic methods in Section~\ref{sec:AAmod} and a $DD$
bimodule which is constructed using the following fairly
straightforward algebraic considerations. 

Let ${\mathfrak A}$ be the set of atomic generators for for $\Quot{m}{k}$
(as in Definition~\ref{def:Atomic}). 
Let $f\colon {\mathfrak A}\to \Clg(m,k)$ be the map given by 
\[
  f(\Lji)=L_{i+1}\cdots L_{j}
  \qquad
  f(\Rij)=R_{j}\cdots R_{i+1}
  \qquad
  f(\Xij\cdot\Idemp{\x})=U_{i+1}\cdot \dots U_j\cdot \cdot \Idemp{\x},
\]
provided $\Xij\cdot \Idemp{\x}\neq 0$; i.e. this is the map which maps
an atomic generator to the element of $\Clg(m,k)$ with complementary 
left idempotent and the same
weight.

In Section~\ref{sec:PongGen}, we introduced certain atomic generators
\[ \{\Xij\}_{\substack{0\leq i<j\leq m\\ (i,j)\neq (0,m)}}, \qquad
\{\Lji\}_{1\leq i<j\leq m-1} \qquad\{\Rij\}_{1\leq i<j\leq m-1}.\]
These induce corresponding elements of $\Quot{m}{k}$.

\begin{thm}
  \label{thm:DDmod}
  There is a type $DD$ bimodule
  $\lsup{\Clg(m,k)}\DDmod^{\Quot{m}{k}}$
  with the following properties:
  \begin{itemize}
    \item $\DDmod$ is isomorphic to $\IdempRing{m}{k}$ as a
      $\IdempRing{m}{k}$-bimodule; i.e. 
      there is a set of generators
      of $\DDmod$ which is in one-to-one correspondence with the idempotent states $\x$, $\gamma_\x$, 
      so that
      \[ (\Idemp{\x}\cdot \gamma_{\x} \cdot \Idemp{\x})=\gamma_\x.\]
    \item
      \begin{align}
         \delta^1(\gamma_\x) =&
      \left(\sum_{1\leq i<j\leq m-1} f(\Rij)\otimes \gamma_\x\otimes \Lji
      + f(\Lji) \otimes \gamma_{\x}\otimes \Rij\right)  \nonumber \\
      & +
      \left(\sum_{0\leq i<j\leq m} f(\Xij)\otimes \gamma_{\x}\otimes \Xij \right).
      \label{eq:DDmodDef}
      \end{align}
  \end{itemize}
\end{thm}

More conceptually, in view of Proposition~\ref{prop:AtomicElements}, we
can interpret $\delta^1$ as the sum over all atomic elements $a$ of
pong tensored with the element of $\Clg$ with the same weight as $a$.

\begin{lemma}
  \label{lem:DDmodProd}
  If $\alpha$ and $\beta$ are two atomic generators (for $\Quot{m}{k}$).
  Then, one of the following three properties holds:
  \begin{itemize}
  \item $\alpha\cdot\beta=0$.
  \item There is another pair $\alpha'$ and $\beta'$ of atomic generators with
    $\alpha\cdot\beta=\alpha'\cdot\beta'$.
  \item $\alpha\cdot\beta$ is of the form $\Rgen{a}{b}\cdot \Rgen{b}{c}$ or
    $\Lgen{c}{b}\cdot\Lgen{b}{a}$.
  \item $\alpha\cdot\beta$ is of the form $\Rgen{b}{c}\cdot\Rgen{a}{b}$,
    $\Lgen{b}{a}\cdot\Lgen{c}{b}$, $\Xgen{a}{b}\cdot\Xgen{b}{c}$, or
    $\Xgen{b}{c}\cdot\Xgen{a}{b}$.
  \end{itemize}
\end{lemma}

\begin{proof}
  The atomic generators satisfy the following relations.
  For $a<b<c<d$,
  \begin{align*}
    \Xgen{a}{d}\cdot \Xgen{b}{c}&=\Xgen{b}{c}\cdot\Xgen{a}{d} \\
    \Xgen{a}{b}\cdot \Xgen{c}{d}&=    \Xgen{c}{d}\cdot \Xgen{a}{b} 
  \end{align*}
  Also for any $a\leq b\leq c$
  \begin{align}
    \Xgen{a}{b}\cdot \Xgen{a}{c}&= \Xgen{a}{c}\cdot\Xgen{a}{b} \label{eq:BorderlineXrel1} \\
    \Xgen{a}{c}\cdot \Xgen{b}{c}&= \Xgen{b}{c}\cdot\Xgen{a}{c}; \label{eq:BorderlineXrel2}
  \end{align}
  and indeed both sides of Equation~\eqref{eq:BorderlineXrel1} vanish if $0<a$;
  similarly, both sides of Equation~\eqref{eq:BorderlineXrel2} vanish if $c=m$.
  For $a<b<c<d$,
  \begin{align*}
    \Rgen{a}{b}\cdot \Rgen{c}{d} &= \Rgen{c}{d}\cdot \Rgen{a}{b} \\
    \Rgen{a}{c}\cdot \Rgen{b}{d} &= 0 \\ 
    \Rgen{b}{d}\cdot \Rgen{a}{c} &=0 \\ 
    \Rgen{a}{d}\cdot \Rgen{b}{c}&= 0 =\Rgen{b}{c}\cdot\Rgen{a}{d} \label{eq:NestedR}
  \end{align*}
  ($\Rgen{a}{c}\cdot\Rgen{b}{c}=0$ because the juxtaposition contains a double-crossing; $\Rgen{b}{c}\cdot\Rgen{a}{c}=0$ because of the idempotents).
  Corresponding identities hold with $\Lgen{j}{i}$ in place of $\Rgen{i}{j}$.
  For any $a\leq b\leq c$, we have the following:
  \begin{align*}
    \Rgen{a}{c}\cdot \Lgen{b}{a} &= \Xgen{a}{b}\cdot \Rgen{b}{c} \\
    \Lgen{c}{b}\cdot\Rgen{a}{c} &=\Rgen{a}{b}\cdot \Xgen{b}{c}  \\
    \Rgen{a}{c}\cdot \Lgen{b}{a} &= \Xgen{b}{c}\cdot \Lgen{b}{a} \\
    \Rgen{a}{b}\cdot \Lgen{c}{a}&= \Lgen{c}{b}\cdot \Xgen{a}{b} 
  \end{align*}
  For any $a<b$, if $\xi_{a,b},\eta_{a,b}\in\{\Xgen{a}{b},\Rgen{a}{b},\Lgen{b}{a}\}$, we have that
  \[ \xi_{a,b}\cdot \eta_{a,b}=0. \]
  (When $\xi$ is of type $R$, this holds for idempotent reasons unless
  $\eta$ is of type $L$, in which case it holds because of a
  double-crossing if $b>a+1$. Moreover, $\Rgen{a}{a+1}\cdot\Lgen{a+1}{a}=U_{a+1}$
  in $\Pong{m}{k}$, therefore the product vanishes in $\Quot{m}{k}$.)
\end{proof}

  \begin{proof}[of Theorem~\ref{thm:DDmod}]
    According to Lemma~\ref{lem:DDmodProd}, the terms from
    $(\mu_2\otimes\Id\otimes \mu_2)\circ (\Id_{\Clg}\otimes \delta^1\otimes\Id_{\mathcal P})\circ\delta^1$
    has possibly non-zero terms of the following types:
    \begin{itemize}
      \item 
        $\alpha\otimes \gamma_\x\otimes \Rgen{a}{b}\cdot\Rgen{b}{c}$;
        but in this case $\alpha=0$.
      \item 
        $\alpha\otimes \gamma_\x\otimes \Lgen{c}{b}\cdot\Lgen{b}{a}$;
        but in this case $a=0$.
      \item
        terms of the form
        $\alpha\otimes\gamma_\x\otimes \Rgen{b}{c}\cdot \Rgen{a}{b}$
      \item
        terms of the form
        $\alpha\otimes\gamma_\x\otimes \Lgen{b}{a}\cdot \Lgen{c}{b}$
      \item
        terms of the form
        $\alpha\otimes\gamma_\x\otimes (\Xgen{a}{b}\cdot \Xgen{b}{c}+\Xgen{b}{c}\cdot \Xgen{a}{b})$
    \end{itemize}
    These terms, in turn, are precisely those that appear in
    $(\Id\otimes \Id\otimes\partial)\circ\delta^1$. (See Lemma~\ref{lem:dAtomic}.)
  \end{proof}

\subsection{Gradings}

Endow $Q$ with the $\Z$-grading $\Ngr$ inherited from the $\Z$-grading
on $\Pong{m}{k}$ specified in Equation~\eqref{eq:Ngr}.

\begin{prop}
  \label{prop:GradedDD}
  Consider the type $DD$ bimodule $X=~\lsup{C}X^Q$ constructed above.
  Suppose that there are generators $x$ and $y$ so that $\delta^1(x)$
  contains $a\otimes y \otimes b$ with non-zero multiplicity. Then,
  $\weight(a)=\weight(b)$ and $\Ngr(b)=-1$.
\end{prop}

\begin{proof}
  This is straightforward from the definitions.
\end{proof}

In the language of Section~\ref{sec:GradingConventions}, the above
proposition states that $X$ is a Maslov/Alexander graded type $DD$
bimodule supported in grading $0$, where $C=\Clg(m,k)$ is given the
trivial $\Z$-grading and $\OneHalf\Z^m$-grading specified by the weight vector;
and $\Quot$ is given the $\Z$-grading $\Ngr$ and $\OneHalf \Z^m$-grading
given by $-1$ times the weight vector.

\newcommand\bSource{\overline\Source}
\newcommand\westpunctures{\mathbf{P}}
\newcommand\eastpunctures{\mathbf{Q}}
\newcommand\westrhos{\vec{\rhos}}
\newcommand\sigmas{\boldsymbol\sigma}
\newcommand\eastrhos{\vec{\sigmas}}
\newcommand\ind{\mathrm{ind}}

\newcommand\bdyA{\partial^{\partial_\alpha}}
\newcommand\bdyB{\partial^{\partial_\beta}}

\section{Holomorphic aspects}
\label{sec:AAmod}

Fix integers $m$ and $k$ with $0<k<m$.  

We endow $\Clg(m,k)$ with a trivial grading -- every element is
supported in grading $0$. Endow $\Quot{m}{k}$ with the renormalized
grading $\Ngr(a)=\#\cross(a)-2\left(\sum_i \weight_i(a)\right)$; while
$\Clg(m,k)$ is thought of as having the trivial (i.e. vanishing) $\Z$-grading.

Our aim is to prove the following:

\begin{thm}
  \label{thm:ConstructY}
  There is an Maslov/Alexander bigraded, strictly unital type $AA$ bimodule
  $\lsub{\Quot{m}{k}}\Ymod_{\Clg(m,k)}$ with the following properties:
  \begin{enumerate}[label=(Y-\arabic*),ref=(Y-\arabic*)]
  \item
    \label{Y:Idempotents}
    $\Ymod$ is isomorphic to $\IdempRing{m}{k}$ as an
      $\IdempRing{m}{k}$-bimodule; i.e. 
      there is a set of generators
      of $\Ymod$ which is in one-to-one correspondence with the idempotent states $\x$, 
      so that
      \[ \Idemp{\x}\cdot Y_\x\cdot \Idemp{\x}=Y_\x.\]
    \item
      \label{Y:LR}
      If $i\in \x$ and $i+1\not\in\x$, 
      let $\x'=(\x\cup \{i+1\})\setminus\{i\}$, then 
      \[        m_{1|1|1}(R_i,Y_{\x'},L_i)=Y_{\x} \qquad{\text{and}}\qquad m_{1|1|1}(L_i,Y_{\x},R_i)=Y_{\x'}.
      \]
    \item
      \label{Y:X}
      If $\{i,i+1\}\subset \x$, then
      \[
        m_{1|1|1}(X_i,Y_{\x},U_i)=Y_{\x}.
        \]
      \item
        \label{Y:0}
        If $1\in \x$, then
        \[
        m_{1|1|1}(X_1,Y_{\x},U_1)=Y_{\x}.
        \]
      \item \label{Y:m}
        If $m-1\in \x$, then
      \[        m_{1|1|1}(X_m,Y_{\x},U_m)=Y_{\x}.
      \]
    \item \label{Y:graded} The bimodule $\Ymod$ is supported in grading $\gr=0$ and $\vec{v}=0$.
  \end{enumerate}
\end{thm}

See Equations~\eqref{eq:GradedBimodule}
and~\eqref{eq:AlexanderGrading} for the grading conventions on
bimodules used here.

\begin{remark}
  We have assumed throughout that $0<k<m$. In the case where $k=0$,
  $\Clg(m,0)\cong \Field$, and $U_i=0$, so the above statement does
  not make any sense. Although the following discussion does work
  $k=m-1$, in the applications, we will not need
  Theorem~\ref{thm:ConstructY} in that case;
  see~\ref{subsec:Extremes}.
\end{remark}

The above bimodule is constructed as follows. First, we construct
bimodules associated to certain types of Heegaard diagrams by suitably
adapting bordered methods (cf.~\cite{InvPair,HolKnot}).  The bimodule
is then associated to a simple ``half-identity'' Heegaard diagram (in
the spirit of~\cite{HomPairing}). The verification of the properties
from Theorem~\ref{thm:ConstructY} is fairly routine.

Most of this section is devoted to importing bordered methods; the
proof of Theorem~\ref{thm:ConstructY} is given at the end of this
section.

\subsection{Lifted partial permutations and mirror-matched circles}

\begin{defn}
  Fix an integer $m\geq 2$.  Let $P$ be a set of $2m-2$ points in the
  circle, equipped with a fixed-point free involution $M$ (called a matching)
  that extends to an 
  orientation-reversing involution of the circle.  The data $(S^1,P,M)$
  is called a {\em mirror-matched circle} with $m-1$ pairs.
\end{defn}

\begin{defn}
A path $\rho\colon [0,1]\to S^1$ is called {\em Reeb-like} if its
local degrees away from $\rho(0)$ and $\rho(1)$ are non-negative. We
view two Reeb-like paths as equivalent if they are homotopic relative
to their endpoints.  The initial point $\rho(0)$ is denoted $\rho^-$
and the terminal point $\rho(1)$ is denoted $\rho^+$.  A set $\rhos$
of Reeb-like paths is called {\em consistent} if for any two distinct
$\rho_1,\rho_2\in \rhos$, the set $\{\rho_1^-, \rho_2^-,
M(\rho_1^-),M(\rho_2^-)\}$ consists of four (distinct) points; and
also $\{\rho_1^+, \rho_2^+, M(\rho_2^+),M(\rho_2^+)\}$ consists of
four points. (Note that if $\rhos$ consists of a single Reeb-like
path, then it is automatically consistent.)  Finally, if $\rhos$ is a
set of Reeb-like paths, let $\rhos^-$ resp. $\rhos^+$ denote the set of initial
resp. terminal points of each path in $\rhos$,
\end{defn}

\begin{defn}
  Let $(S^1,P,M)$ be a mirror-matched circle with $m-1$ pairs.
  For each  integer $1\leq k\leq m-1$, an {\em $M$-consistent constraint packet} consists of a pair of 
  subsets $S,T\subset P/M$ and a
  consistent set $\rhos$
  of Reeb-like paths $\rhos$ with the following properties:
  \begin{itemize}
  \item $\rhos^-/M \subset S$
  \item $\rhos^+/M\subset T$
  \item $|S|=|T|=k$
  \item $S\setminus (\rho^-/M)=T\setminus (\rhos^+/M)$.
  \end{itemize}
\end{defn}

We write an explicit model for a mirror-matched circle, as
follows. The circle will be thought of as the quotient $\R/2m \Z$, the
involution $r$ is induced by the map $x\mapsto -x$ (thought of as in
involution of $\R$), and $P=\{i+\OneHalf\}_{i=-m}^{m-1}$.

For the above parametrization of the mirror-matched circle, a set of
homotopy classes of Reeb-like paths corresponds to a subset $S\subset
P$ with $S\cap -S=\emptyset$, and a function $f\colon S\to \OneHalf +
\Z$, with the property that $s\leq f(s)$ for all $s\in S$.

Let $({\widetilde S},{\widetilde f})$ be a lifted partial permutation.
There is a corresponding $M$-consistent constraint packet defined as follows:
\begin{itemize}
\item $S\subset{\widetilde S}/2m \Z$ is the image of those
  $\sigma\in {\widetilde S}$ with $\sigma\leq {\widetilde f}(\sigma)$
\item $T={\widetilde f}({\widetilde S}/2m\Z)$
  consists of the equivalence class of those ${\widetilde f}(\sigma)\in\R$
  with the property that $\sigma\leq {\widetilde f}(\sigma)$.
\end{itemize}

\begin{prop}
  \label{prop:PacketToLiftPartialPermutation}
  The above map sets up a one-to-one correspondence between lifted
  partial permutations on $k$ letters and $M$-consistent constraint
  packets.
\end{prop}

\begin{proof}
  The argument is straightforward, amounting to the observation
  that ${\widetilde f}(\sigma)<\sigma$ if and only if
  ${\widetilde f}(-\sigma)>-\sigma$.
\end{proof}

On the level of pong diagrams, the above one-to-one correspondence can
be described as follows. Start from a pong diagram for a lifted partial
permutation, and reflect it once (to obtain two fundamental domains
for the $G_m$ action). Next, remove all the strands that are moving
down and to the left. Finally, view the resulting picture as taking place
on $S^1\times [0,1]$ (where $S^1=\R/2m\Z$).

See Figure~\ref{fig:PongToReeb}.

\begin{figure}[ht]
\input{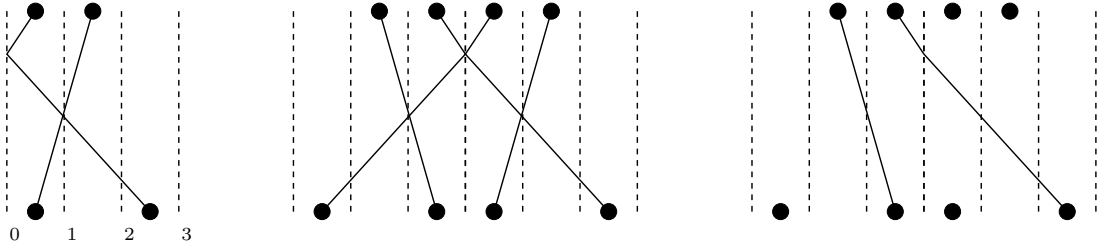}
\caption{\label{fig:PongToReeb}
{\bf{Pong diagrams to Reeb-like chord packets.}}}
\end{figure}

\begin{defn}
  \label{def:PongAlgElt}
  If $\sigmas$ is an $M$-consistent constraint packet, and
  $|\sigmas|\leq k$, we define the corresponding pong algebra element
  $a(\sigmas)$ as follows.  When $|\sigmas|=k$, let $a(\sigmas)$ be
  the element of $\Pong{m}{k}$ associated to the corresponding lifted
  partial permutation as in
  Proposition~\ref{prop:PacketToLiftPartialPermutation}.  When
  $|\sigmas|<k$, we sum over all lifted partial permutations obtained
  by extending by the identity map the given lifted partial
  permutations.
\end{defn}

\subsection{$\alpha$-$\beta$ bordered diagrams}
  
\begin{defn}
  \label{def:Hab}
An {\em $\alpha$-$\beta$-bordered diagram} $\HD$ consists of the following data:
\begin{itemize}
  \item An oriented surface $\Sigma$ with genus $g$,
    whose boundary is decomposed as a disjoint union
    \[ \partial\Sigma=\partial_\alpha\coprod \partial_\beta.\]
  \item A collection $\{\alpha_1,\dots,\alpha_{m-1}\}$ 
    of $\alpha$-arcs whose boundaries lie on $\partial_\alpha$
  \item A collection 
    $\{\beta_1,\dots,\beta_{m-1}\}$ of $\beta$-arcs whose boundaries lie
    on $\partial_\beta$.
  \item A collection $\{\alpha^c_i\}_{i=1}^g$ and $\{\beta^c_i\}_{i=1}^g$ of
    pairwise disjoint, embedded circles.
  \item The boundary $\partial_\alpha$ consists of a disjoint union
    of circles $Z_1,\dots,Z_m$, labelled so that $\alpha^c_i$ is a path
    from $Z_i$ to $Z_{i+1}$.
  \item The boundary $\partial_\beta$ consists of a single
    mirror-matched circle $Z$, where the matching is given by the
    endpoints of the $\beta_i$ arcs. 
\end{itemize}
This data is also required to satisfy the following compatibility
condition.  If we cut $\Sigma$ along the along all the $\beta$-circles
and arcs, we obtain $m$ components $B_1\cup\dots\cup B_m$ with the following 
properties:
\begin{itemize}
\item The $B_i$ has a boundary component $\beta_{i-1}$
  (unless $i=1$, in which case this boundary component is empty) and
  another boundary component $\beta_i$ (unless $i=m$, in which case
  this boundary component is empty).
\item The boundary $\partial B_i$ meets $Z$ in one arc 
if $i\in \{1,m\}$, and it meets $Z$ in two arcs otherwise. 
\item  $B_i$ meets exactly one component of $\partial_\alpha$.
\end{itemize}
\end{defn}

The boundary component $\partial_\beta$ is the ``pong-bordered''
portion of the boundary of $\HD$ mentioned in the introduction.

The key example of an $\alpha$-$\beta$-bordered diagram $\HD$ --
indeed, the only one we consider in this paper -- is the {\em
  half-identity diagram} $\HD_m$, indexed by the integer $m$, defined
as follows. $\HD_m$ has genus $0$ and $m+1$ boundary components, as
pictured in Figure~\ref{fig:HalfIdentity}.

\begin{defn}
  \label{def:HalfIdentity}
  The {\em half-identity diagram} $\HD$ is obtained from a genus $0$ surface
  with $m+1$ boundary components, one of which is
  $\partial_\beta$; and the remaining $m$ boundary components are
  labelled $\partial_\alpha^1,\dots,\partial_\alpha^m$.  Draw an arc
  $\alpha_i$ from $\partial_\alpha^i$ to $\partial_\alpha^{i+1}$, and
  let $\beta_i$ be an arc from $\partial_\beta$ to itself that is disjoint from all
  $\alpha_j$ with $i\neq j$, intersecting $\alpha_i$ in a single transverse point.
\end{defn}

\begin{figure}[ht]
\input{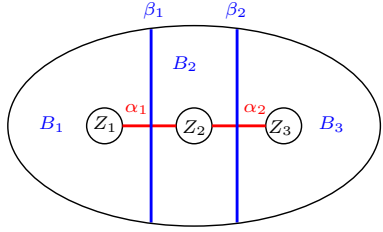}
\caption{\label{fig:HalfIdentity}
{\bf{Half identity diagram with $m=3$.}}}
\end{figure}

It follows from the definitions that to each $\alpha$-$\beta$-bordered
diagram $\HD$, there is a permutation $\sigma$ on $m$ letters with the
property that $B_i$ meets $Z_{\sigma(i)}$.  For the half-identity,
this permutation is the identity.

\subsection{Heegaard states for bordered diagrams}
\label{subsec:HeegaardStates}

\begin{defn}
  \label{def:HeegaardState}
  Let $\HD$ be an $\alpha$-$\beta$-bordered Heegaard diagram as in
  Definition~\ref{def:Hab}.  Fix an integer $0\leq k < m$.  A {\em
    $k$-fold Heegaard state} consists of the following data:
\begin{itemize}
\item A $k$-element subset $S\subset \{1,\dots,m-1\}$. 
\item A $k$-element subset $T$ of $\{1,\dots,m-1\}$.
\item a $g+k$-tuple of points $\x$ 
\end{itemize}
satisfying the following conditions:
\begin{itemize}
\item
  $\x\subset \left(\bigcup_{s\in S}\alpha_s\right)\cup
  \left(\bigcup_{i=1}^g \alpha_i^c\right)$.
\item $\x\subset \left(\bigcup_{t\in T}\beta_t\right)\cup
  \left(\bigcup_{i=1}^g \beta_i^c\right)$.
  \item Each $\alpha$-circle contains exactly one component of $\x$.
  \item Each $\beta$-circle contains exactly one component of $\x$.
  \item For each $s\in S$, the corresponding $\alpha$-arc
    $\alpha_s$ contains exactly one component of $\x$.
  \item Each $t\in T$, the corresponding $\beta$-arc 
    $\beta_t$ contains exactly one component of $\x$.
\end{itemize}
\end{defn}

Note that $S$ and $T$ are determined by $\x$.
Let $\States(\HD)$ denote the set of Heegaard states of $\HD$.

We associate to $\HD$ a type $AA$-bimodule
$\lsub{\Quot{m}{k}}\CFBA(\HD)_{\Clg(m,k)}$.

As a $\Zmod{2}$-vector space, $\CFBA(\HD)$ is generated
by Heegaard states.
The $\IdempRing{m}{k}$-bimodule structure is specified by
\[ \Idemp{S}\cdot \x\cdot \Idemp{T}=\x,\]
where the sets $S,T\subset \{1,\dots,m-1\}$ are as in
Definition~\ref{def:HeegaardState}.

Actions on the module $\CFBA(\HD)$ are defined by counting holomorphic
disks, modifying the definition from~\cite{HolKnot} slightly to take
into account the $\partial_\beta$ boundary. (See
also~\cite{Bimodules,HomPairing}.)  We pause our definition now to
describe elements of the holomorphic theory, returning to the
definition afterwards.

\subsection{Pseudo-holomorphic curves}

Following~\cite{LipshitzCyl}, we consider pseudo-holomorphic curves in
$\Sigma\times[0,1]\times \R\to \Sigma$  with respect to an admissible
almost-complex structure~\cite[Section~5]{HolKnot}.

We consider $J$-holomorphic curves
\[u \colon (\Source,\partial \Source) \to (\Sigma\times [0,1]\times \R
,(\alphas\times\{1\}\times \R)\cup(\betas\cup\{0\}\times \R)),\]
satisfying the asymptotic conditions of a pre-flowline
(cf.~\cite[Definition~5.4]{HolKnot}), with a few minor modifications:
there will now not be any punctures marked by Reeb orbits, and there will be constraints on both $\partial_\alpha\HD$ and $\partial_\beta\HD$.

We outline the construction in a little more detail presently.

\begin{defn}
  \label{def:AlgConstraintPacket}
  A {\em $\beta$-constraint packet} $\sigmas$ is a union of Reeb
  chords on $\partial_\beta\HD$. The constraint packet is called {\em
    algebraic} if
  \begin{itemize}
    \item for any pair $\sigma_a$ and $\sigma_b$ of distinct
      chords in $\sigmas$, the two initial points $\sigma_a^-$ and $\sigma_b^-$ lie on
      distinct $\beta$-arcs.
    \item for any pair $\sigma_a$ and $\sigma_b$ of distinct
      chords, the two terminal points $\sigma_a^+$ and $\sigma_b^+$ lie on
      distinct $\beta$-arcs.
  \end{itemize}
  An {\em $\alpha$-constraint packet} $\rhos$ is defined similarly.
  In that case, the constraint packet is called {\em algebraic} if
  the following third condition is also met:
  \begin{itemize}
  \item Any two distinct chords in $\rhos$ lie on distinct boundary components.
  \end{itemize}
\end{defn}

\begin{defn}
  A {\em profile} is a triple
  $(\x,\vec{\rhos},\vec{\sigmas})$ consisting of
  \begin{itemize}
  \item a Heegaard state $\x$,
  \item a sequence of $\alpha$-constraint packets
    $\vec{\rhos}=\rhos_1,\dots,\rhos_r$,
  \item a sequence of $\beta$-constraint packets
    $\vec{\sigmas}=\sigmas_1,\dots,\sigmas_s$.
  \end{itemize}
  A profile is called {\em boundary monotone},
  in which case we also define its $\alpha$-set, denoted
  $\alpha(\x,\vec{\rhos})$
  (which is independent of the $\vec{\sigmas}$);
  and its $\beta$-set, denoted $\beta(\x,\vec{\sigmas})$,
  if the  following conditions hold:
  \begin{itemize}
  \item $\alpha(\x)\supset \alpha(\sigmas_1^-)$
  \item $\beta(\x)\supset \beta(\rhos_1^-)$
  \item for $i=1,\dots,s-1$, the initial points $\sigmas_i^-$ all lie on distinct $\alpha$-arcs, the terminal points $\alpha_i^+$ all lie on distinct $\beta$-arcs, and 
    \[
      \alpha(\sigmas_i^-)\subset \alpha(\x,\{\sigmas_1,\dots,\sigmas_{i-1}\}). \]
    In this case we define the $\alpha$-set to be
    \[ \alpha(\x,\{\sigmas_1,\dots,\sigmas_i\})=
    \left(\alpha(\x,\{\sigmas_1,\dots,\sigmas_{i-1}\})\setminus\sigmas_i^-\right)\cup
    \sigmas_i^+.\]
  \item for $i=1,\dots,r-1$, the initial points $\rhos_i^-$ all lie on distinct
    $\beta$-arcs,
    the terminal points $\rhos_i^+$ all lie on distinct $\beta$-arcs, and 
    \[
      \beta(\rho_i^-)\subset \beta(\x,\{\rhos_1,\dots,\rhos_{i-1}\}). \]
    In this case we define the $\beta$-set to be
    \[ \beta(\x,\{\rhos_1,\dots,\rhos_i\})=
    \left(\beta(\x,\{\rhos_1,\dots,\rhos_{i-1}\})\setminus\rhos_i^-\right)\cup
    \rhos_i^+.\]
  \end{itemize}
\end{defn}

Note that for a boundary monotone profile
$(\x,(\sigmas_1,\dots,\sigmas_s),(\rhos_1,\dots,\rhos_r))$, all the
chord packets $\sigmas_i$ are algebraic, in the sense of
Definition~\ref{def:AlgConstraintPacket}, and we can define
corresponding algebra elements $a(\sigmas_i)$ as in Definition~\ref{def:PongAlgElt}. The packets $\sigmas_i$, however, need
not be algebraic in the sense of
Definition~\ref{def:AlgConstraintPacket} (though we will often require
them to be).  When they are, we can define algebra elements
$a(\rhos_i)$ whose weights coincide with the Reeb chords.

The map from algebraic $\beta$-constraint packets to $\Clg(m,k)$ and
the map from algebraic $\alpha$-constraint packets to $\Quot{m}{k}$
are both denoted $a$; however, the distinction should be clear from
the context.  Also note that the map from algebraic chord packets to
elements of $\Quot{m}{k}$ is injective; whereas for $\Clg(m,k)$, the
corresponding map is not injective.  For example, the algebra element
$U_i\Idemp{\x}$ when $i-1,i\in\x$, can be represented by either length
$1$ chord which covers $Z_i$.

\begin{defn}
  \label{def:DecoratedSource}
  Fix an $\alpha$-$\beta$ bordered diagram $\HD$ as in Definition~\ref{def:Hab}, and an integer $k$ with $0\leq k\leq m-1$.
  A {\em decorated source} $\Source$ is the following collection of data:
\begin{itemize}
\item a smooth oriented surface $S$ with boundary and punctures on the boundary
\item a labeling of each puncture of $\Source$ by one of $+$, $-$,
  or a Reeb chord on $\partial\HD$.
\end{itemize}
Let $P(\Source)$ resp. $Q(\Source)$ denote the set of punctures marked by chords on $\partial_\alpha\HD$ resp. $\partial_\beta\HD$.
\end{defn}
\begin{defn}
  \label{def:GenFlow}
  A {\em pre-flowline} is a map
\[ u\colon (\Source,\partial\Source)\to
(\Sigma\times[0,1]\times\R,(\alphas\times\{1\}\times \R)\cup(\betas\times\{0\}\times\R))\]
subject to the constraints:
\begin{enumerate}[label=(${\mathcal M}$-\arabic*),ref=(${\mathcal M}$-\arabic*)]
  \item 
    \label{property:First}
    $u\colon \Source\to \Sigma\times[0,1]\times \R$ is proper.
  \item 
    \label{property:ProperTwoa}
    $u$ extends to a proper map ${\overline u}\colon
    {\bSource}' \to {\overline{\Sigma}}\times[0,1]\times \R$,
    where $\bSource'$ is obtained from $\Source$ by filling in the punctures
    labelled by Reeb chords
  \item
    \label{prop:BrCover}
    $\pi_{\CDisk}\circ u$ is a $d$-fold branched cover,
    where $d=g+k$.
  \item At each $-$-puncture $q$ of $\Source$,
    $\lim_{z\goesto q}(t\circ u)(z)=-\infty$.
  \item At each $+$-puncture $q$ of $\Source$,
    $\lim_{z\goesto q}(t\circ u)(z)=+\infty$.
  \item
    For each $p\in P(\Source)\cup Q(\Source)$, 
    $\lim_{z\to q}(\pi_{\Sigma}\circ u)(z)$ is the Reeb
    chord $\rho$ labeling $p$. 
  \item \label{prop:FiniteEnergy}
    There are generalized Heegaard states $\x$ and $\y$
    with the property that as $t\goesto - \infty$,
    $\pi_{\Sigma}\circ u$ is asymptotic to $\x$  and
    as $t\goesto +\infty$, $\pi_{\Sigma}\circ u$ is asymptotic
    to $\y$.
    \setcounter{bean}{\value{enumi}}
  \item 
    \label{prop:WeakBoundaryMonotone}
    For each $t\in \R$ and $i=1,\dots,g$, 
    $u^{-1}(\beta^c_i\times\{0\}\times\{t\})$ consists of exactly one point;
    and also 
    $u^{-1}(\alpha_i^c\times\{1\}\times\{t\})$ consists of exactly one point.
  \item
    For each $t\in \R$ and $i=1,\dots,m$,
    $u^{-1}(\alpha_i\times \{1\}\times\{t\})$ consists of at most one point;
    and also 
    $u^{-1}(\beta_i\times \{0\}\times\{t\})$ consists of at most one point;
\end{enumerate}
If $u$ is a pre-flowline, the {\em associated profile} is the data
consisting of the initial state $\x$, and the sequences $\vec\sigmas$
resp. $\vec\rhos$,
consisting of the images of the $P(\Source)$ resp. $Q(\Source)$,
ordered according to
the value of $t\circ u$.
\end{defn}

Fix the following data:
\begin{itemize}
\item $\x,\y\in \States(\HD)$
\item compatible sequences of constraint packets
  $\westrhos$ and $\eastrhos$
\item a homology class class $B\in\pi_2(\x,\y)$,
\end{itemize}
and let $\ModFlow^B(\x,\y;\Source,\westrhos,\eastrhos)$ be
the moduli space of holomorphic representatives of $B$ whose
asymptotics are as specified by the profile $(\x,\westrhos,\eastrhos)$.

\subsection{A digression on the index}

We collect here properties of the index of holomorphic curves. This
discussion is mostly an adaptation
of~\cite[Section~5.7.1]{InvPair}.

Following~\cite{InvPair} (see also~\cite[Section~7]{HolKnot}), for
a packet $\rhos$ of Reeb chords, let
$\iota(\rhos)=\inv(\rhos)-m([\rhos],\rhos^-)$, $\inv(\rhos)$ denotes
the minimal number of crossings between the various chords (i.e. this
coincides with $\inv(\rhos)$ when $\rhos$ represents a lifted partial
permutation on a mirror matched circle), and $m([\rhos],\rhos^-)$
denotes the sum over all the initial points $p$ of the chords
$\rhos^-$ of the average local multiplicity of $[\rhos]$ near $p$. For
example, if the packet consists of the single chord $\rhos=\{\rho\}$,
then $-\iota(\rhos)$ is the total weight of the chord.

Suppose that $\x$ and $\y$ are two states, and $(B,\vec{\rho}=(\rhos_1,\dots,\rhos_r),
\vec{\sigmas}=(\sigmas_1,\dots,\sigmas_s))$ is strongly boundary monotone.
Let $|\vec{\rho}|=r$ and $|\vec{\sigmas}|=s$.
Define
  \begin{align}
    \chiEmb(B,\vec{\rhos},\vec{\sigmas})&= d+ e(B) - n_\x(B)-n_\y(B) -\iota(\vec{\rhos})
    -\iota(\vec{\sigmas})
    \label{eq:ChiEmbA} \\
    \ind(B,\x,\y;\vec{\rhos}) &= e(B)+n_\x(B)+n_\y(B)+|\vec{\rhos}|+|\vec{\sigmas}| 
    +\iota(\vec{\rhos})
    +\iota(\vec{\sigmas}), \label{eq:IndexFormula}
  \end{align}
  where
  \[ \iota(\vec{\rhos})=\sum_{i=1}^{\ell} \iota(\rhos_i).
  \] 

\begin{remark}
Observe that constraint packets $\rhos_i$ in
$\vec{\rhos}=(\rhos_1,\dots,\rhos_r)$ that we consider here consist
entirely of chords; whereas
in~\cite[Definition~\ref{HK:def:IndexTypeA}]{HolKnot}, the formulas
allowed for additional orbit constraints.  Also, when comparing with
the formulas from~\cite{HolKnot}, bear in mind that the Euler measure
there is taken on the compactification of the Heegaard surface; this
is why the total weight on the boundary appears there.
\end{remark}

The following proposition is proved
in~\cite[Proposition~5.69]{InvPair}~\cite[Proposition~4.2]{LipshitzCyl};
compare also~\cite[Proposition~7.11]{HolKnot}.

\begin{prop}
  \label{prop:Dimensions}
  If $\ModFlow^B(\x,\y,\Source,\vec{\rhos},\vec{\sigmas})$ is
  represented by some pseudo-holomoprhic representative $u$, then
  $\chi(\Source)=\chiEmb(B)$ if and only if $u$ is embedded. In this
  case, the expected dimension of the moduli space is given by
  $\ind(B,\x,\y,\vec{\rhos},\vec{\sigmas})$. Moreover, if a strongly
  monotone moduli space has a non-embedded holomorphic representative,
  then its expected dimension is $\leq
  \ind(B,\x,\y\vec{\rhos},\vec{\sigmas})-2$.
\end{prop}

\begin{prop}
  \label{prop:BoundaryMonotoneAndDimension}
  If $(\x,\vec{\rhos},\vec{\sigmas})$ is a profile with the property that
  each of the $\sigmas_i$ and $\rhos_i$ are algebraic, then the profile is boundary monotone if and only if
  \[ a(\rhos_s)\otimes\dots\otimes a(\rhos_1)\otimes \x \otimes a(\sigmas_1)\otimes \dots\otimes a(\sigmas_r) \neq 0 \]
  in $\Quot{m}{k}^{\otimes s}\otimes \Ymod(\HD)\otimes \Clg(m,k)^{\otimes r}$,
  where all tensor products are taken over $\IdempRing{m}{k}$.
  Moreover, in this case, 
  \[
    e(B)+n_\x(B)+n_\y(B)=\ind(B,\vec{\rhos},\vec{\sigmas})-r-s+\iota(\vec{\rhos})+\iota(\vec{\sigmas}). 
    \]
\end{prop}

\begin{proof}
  This follows as in~\cite[Lemma~\ref{HK:lem:NonZeroAlgElts}]{HolKnot}.
  The equation is a straightforward consequence of Equation~\eqref{eq:IndexFormula}.
\end{proof}

\subsection{Codimension one ends}

Following~\cite[Theorem~5.61]{InvPair},
we name various kinds of codimension one ends of this moduli space.
\begin{itemize}
\item    A {\em two-story end} corresponds to a decomposition
  \[\ModFlow^{B_1}(\x,{\mathbf p};\Source_1;\westrhos_1,\eastrhos_1)\times
  \ModFlow^{B_2}(\x,{\mathbf p};\Source_2;\westrhos_2,\eastrhos_2),\]
  where ${\mathbf p}\in\States(\HD)$, $B_1\in\pi_2(\x,{\mathbf p})$, $B_2\in\pi_2({\mathbf p},\y)$,
  $B=B_1+B_2$.
\item An {\em odd shuffle curve end} is an element of
  $\ModFlow^B(\x,\y;\Source',\westrhos',\eastrhos')$ obtained by
  picking two punctures $q_1$ and $q_2$ whose corresponding Reeb
  chords $\sigma_1$ and $\sigma_2$ (on $\partial_\beta\HD$) are
  constrained in the same packet and nested.
  Here, $\Source'$ is obtained by pre-gluing an odd shuffle
  component with punctures $q_1'$ and $q_2'$, labelled by interleaved
  Reeb chords $\sigma_1'$ and $\sigma_2'$.  Here, $\westpunctures'$ is
  obtained from $\westpunctures$ by replacing the chords $\sigma_1$ and $\sigma_2$
  by $\sigma_1'$ and $\sigma_2'$; while $\eastpunctures'=\eastpunctures$.
  (Note that there is no space on the cylinders attached to
  $\partial_\alpha\HD$ for shuffle curves, as in~\cite{HolKnot}.)
  See the left of Figure~\ref{fig:Ends} for a schematic.
\item A {\em join curve end}
  $\ModFlow(\x,\y;\Source',\westrhos',\eastrhos')$ is determined by
  choosing a Reeb chord label $\rho$ (which can be on $\partial_\alpha\HD$ or
  $\partial_\beta\HD$) for
  one of the punctures $q$ of $\Source$ and a decomposition
  $\sigma=\sigma_1\uplus \sigma_2$.  The decorated source $\Source'$
  has two new punctures $q_1$ and $q_2$, so that $\Source$ is obtained
  from $\Source'$ by pregluing a join component to $\Source'$ at the
  punctures $q_1$ and $q_2$. The sequences of constraint packets for
  $\westrhos'$ and $\eastrhos'$ are obtained from $\westrhos$ and $\eastrhos$
  by replacing a single copy of $\rho$ by the two chords $\rho_1$ and $\rho_2$.
  See the right of Figure~\ref{fig:Ends} for a schematic.
\item A {\em collision level} is an element
  of $\ModFlow^B(\x,\y;\Source',\westrhos',\eastrhos')$
  obtained by contracting the 
  arcs on $\partial \Source$ that connect punctures whose labels lie in two
  consecutive packets ($\rhos_i$ and $\rhos_{i+1}$ of $\westrhos'$ or 
  $\sigmas_i$ and $\sigmas_{i+1}$ of $\eastrhos$), and replacing the constraint packet by
  the one obtained by forming the join of those two constraint packets.
\end{itemize}

\begin{figure}[ht]
\input{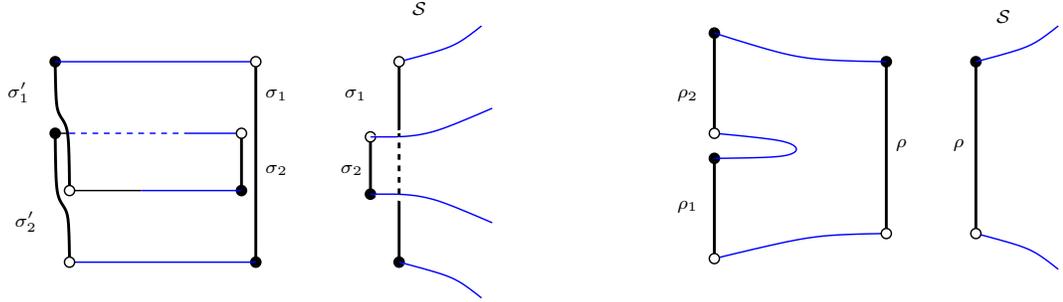}
\caption{\label{fig:Ends}
{\bf{Some codimension one ends.}} 
At the left, we have illustrated an odd shuffle curve end; at the right, a join curve end.}
\end{figure}

\begin{thm}
  \label{thm:CodimOneBoundary}
  If $(\x,\westrhos,\eastrhos)$ is strongly boundary monotone and
  $B\in\pi_2(\x,\y)$.  Fix $\Source$ so that
  $\ind(B,\Source,\westrhos,\eastrhos)=2$.  Then, the total number of
  ends of $\ModFlow=\ModFlow^B(\x,\y;\Source;\westrhos,\eastrhos)$ of
  the following kinds is even:
  \begin{itemize}
    \item two-story ends of $\ModFlow$.
    \item join curve ends of $\ModFlow$. 
    \item odd shuffle curve ends on $\ModFlow$
    \item collision of levels of $i$ and $i+1$, 
      in which case the packets in the colliding levels
      are weakly composable.
  \end{itemize}
\end{thm}

\begin{proof}
  This follows as in~\cite[Theorem~5.61]{InvPair}, though a few
  remarks are in order.  The codimension-one phenomena on the
  $\beta$-side are those listed above, as
  in~\cite[Theorem~5.61]{InvPair}. As in that case, there are also
  possibly even shuffle-chord ends, but those cancel in pairs. (See~\cite[Proposition~5.42]{InvPair}.)

  There are also possibly codimension-one phenomena on the $\alpha$-side, which are enumerated in
  Theorem~\cite[Theorem~7.17]{HolKnot}. In that statement, there were additional curves we did not consider here --
  the  ``orbit curve ends''  from~\cite[Theorem~7.17]{HolKnot}. These do not appear, however,
  since the constraint packets we consider here do not contain Reeb orbits.

  Boundary degenerations appearing as limiting objects cannot be ruled
  out using the same logic as in~\cite{InvPair}: the $\beta$-arcs are
  not homologically linearly indendent. However, our hypotheses on
  $\HD$ (Definition~\ref{def:Hab}) ensure that all the regions which
  are obtained by cutting along $\beta$-arcs (i.e. which support
  $\beta$-boundary degenerations) also meet the
  $\alpha$-boundary. Thus, $\beta$-boundary degenerations occur in
  moduli spaces which contain closed Reeb orbits on that
  $\alpha$-side; but these are not moduli spaces we consider
  presently. A simpler argument also rules out $\alpha$-boundary degenerations.
\end{proof}

\subsection{Actions on the bimodule}

Let $\HD$ be an $\alpha$-$\beta$-bordered diagram,
and let $\lsub{\Quot{m}{k}}\CFBA(\HD)_{\Clg(m,k)}$ be its associated
$\IdempRing{m}{k}$-bimodule.
Fix a sequence ${\vec{a}}=a_1,\dots,a_r$ of lifted partial permutations
representing algebra elements in $\Quot{m}{k}$ and a sequence
$\vec{b}=b_1,\dots,b_s\in\Clg(m,k)$ be pure algebra elements.
Write
\[ \vec{a}\otimes \x\otimes \vec{b}=a_1\otimes\dots\otimes a_r\otimes \x
\otimes b_1\otimes\dots\otimes b_s.\]

If $\vec{a}\otimes\x\otimes\vec{b}=0$, define
\[ m_{r|1|s}(a_1\otimes\dots\otimes a_r\otimes \x\otimes b_1\otimes\dots\otimes \otimes b_s)=0.\]
Otherwise, if $\vec{\rhos}=\rhos_1,\dots,\rhos_r$ and
$\vec{\sigmas}=\sigmas_1,\dots,\sigmas_s$ are two sequences of constraint packets, define
\[ a(\vec{\rhos})=a(\rhos_1)\otimes\dots\otimes a(\rhos_r)
\qquad{\text{and}}\qquad
a(\vec{\sigmas})=a(\sigmas_1)\otimes\dots\otimes a(\sigmas).\]


$(\x,(\rhos_1,\dots\rhos_r),(\sigmas_1,\dots,\sigmas_s))$
be the corresponding profile with
\[ a(\rhos_1)\otimes\dots\otimes a(\rhos_r)\otimes \x\otimes
a(\sigmas_1)\otimes\dots\otimes a(\sigmas_s)=a_1\otimes\dots\otimes a_r\otimes \x\otimes b_1\otimes\dots\otimes b_s,\]
and let
\begin{align*}
  m_{r|1|s}&(a_1,\dots, a_{r},\x,b_1,\dots,b_s) \\
  &=
  \sum_{\left\{\y\in\States(\HD),\vec{\rhos},\vec{\sigmas}\big|
    \begin{array}{l}
      \ind(B,\x,\y,\vec{\rhos},\vec{\sigmas})=1\\
      a(\vec{\rhos})\otimes \x\otimes a(\vec{\sigmas})=\vec{a}\otimes\x\otimes{\vec{b}}
    \end{array}\right\}}
  \#\left(\frac{\ModFlow^B(\x,\y;\vec{\rhos},\vec{\sigmas})}{\R}\right)\cdot \y.
\end{align*}

\begin{thm}
  \label{thm:HolomorphicAAbimodule}
  If $\HD$ is an $\alpha$-$\beta$ bordered diagram, then
  $\lsub{\Quot{m}{k}}\CFBA(\HD)_{\Clg(m,k)}$ is a type $AA$ bimodule.
\end{thm}

\begin{proof}
  The ${\mathcal A}_{\infty}$ relations are a reinterpretation of
  Theorem~\ref{thm:CodimOneBoundary}. See also~\cite[Theorem~\ref{HK:thm:DefTypeA}]{HolKnot}, which in turn is modelled on~\cite[Proposition~7.12]{InvPair}.
\end{proof}

\subsection{The dualizing bimodule}

We now turn to the proof of Theorem~\ref{thm:ConstructY}.

First, we establish its grading properties:

\begin{prop}
  \label{prop:GradedBimodule}
  The dualizing bimodule is graded, in the sense that
  Equation~\eqref{eq:GradedBimodule} holds.
\end{prop}

\begin{proof}
  It suffices to prove that for any boundary monotone profile $(\x,\vec{\rhos},\vec{\sigmas})$
  and compatible domain $B$, we have that
  \begin{equation}
    \label{eq:1}\weight_i(\partial_\alpha B)=\weight_i(\partial_\beta B)
    \end{equation}
  and also
  \begin{equation}
    \label{eq:2}
    \ind(B,\x,\y,\vec{\rhos},\vec{\sigmas})  = |\vec{\rhos}|+|\vec{\sigmas}|
    +\sum_{i=1}^{|\vec{\rhos}|}\Ngr(a(\vec{\rhos}))|.
  \end{equation}
  
  Equation~\eqref{eq:1} in fact holds for any domain $B$. It is a consequence of the fact
  Equation~\eqref{eq:1} holds for each elementary domain.
  In particular, it follows that
  \[ \weight(\vec{\rhos})=\weight(\vec{\sigmas})=\sum_{i=1}^{|\vec{\rhos}|} \weight(a(\rhos_i)).\]

  Equation~\eqref{eq:2} is equivalent to the condition that for any
  domain $(B,\vec{\rhos},\vec{\sigmas})$,
  \[
  e(B)+n_\x(B)+n_\y(B) +\iota(\vec{\rhos}) +\iota(\vec{\sigmas})=
  \sum_{i=1}^{|\vec{\rhos}|} \Ngr(a(\rhos_i)).\] Since each the
  various chords in the chord packets $\sigmas_i$ lie on different
  boundary components (this is the ``algebraic'' hypothesis from
  Definition~\ref{def:AlgConstraintPacket}), it is easy to see that
  \[ \iota(\vec{\sigmas})=-\weight(\vec{\rhos}).\]

  Let $C$ denote the total count of left- and right-walls hit among
  all of the pong elements $a(\rhos_1),\dots,a(\rhos_r)$.  We claim
  that $e(B)=-\weight(\vec{\rhos})+\frac{C}{2}$, This is a
  straightforward computation: cut the diagram $\HD$ along all the
  $\alpha$- and $\beta$-arcs.  The result will consist of $2m-2$
  elementary domains with $e=-1/2$, meeting
  the chords $U_1$, $L_1$, $R_1$, $L_2$, $R_2$, \dots, $L_{m-2}$,
  $R_{m-2}$, and $U_{m-1}$. 
    
  Thus, it remains to verify that 
  \begin{equation}
    \label{eq:RemainsToSee}
    n_\x + n_\y + \iota(\vec{\rhos})=-\frac{C}{2}+\sum_{i=1}^{|\vec{\rhos}|}\cross(a(\rhos_i)).
  \end{equation}
  We prove this by decomposing the domain $(B,\vec{\rhos},\vec{\sigmas})$ into elementary pieces,
  corresponding to the decomposition of pong elements from Proposition~\ref{prop:PongGen}.
  
  Consider the Reeb chords in the first packet $\rhos_1$.
  
  {\bf Case 1.} If there is a right-moving Reeb chord, consider
  the rightmost, right-moving Reeb chord $\xi_0$.  
  
  {\bf Case 1a.}
  Suppose there there is also a
  left-moving Reeb chord $\eta_0$ to the right of $\xi_0^-$, and choose the leftmost such strand.
  If the terminal points of the two strands cross (on the interval), we can factor off the two 
  oppositely-moving chords in their own packet $\rhos_0$. There is a corresponding decomposition
  \[(B,\x,\y,\{\rhos_0,\rhos_1',\rhos_1,\dots,\rhos_r\})
  =  (B_0,\x,x',\{\rhos_0\})\#(B',\x',\y,\{\rhos_1',\rhos_2,\dots,\rhos_r\})\]
  (We have suppressed here the chord packets appearing on $\partial_\alpha$, as they do not affect either
  side of the equation. Note that $\rhos_0$ contains a portion of $\xi_0$ and $\eta_0$.)
  If $\xi_0$ terminates before the initial point of $\eta_0$,
  we can factor off $\xi_0$, 
  \[(B,\x,\y,\{\rhos_0,\rhos_1',\rhos_1,\dots,\rhos_r\})
  =  (B_0,\x,\x',\{\{\rhos_0\}\})\#(B,\x',\y,\{\rhos_1',\rhos_2,\dots,\rhos_r\}),\]
  
  {\bf Case 1b.} There is no left-moving strand starting to the right of the $\xi_0-$. In that case,
  we can factor 
  \[(B,\x,\y,\{\rhos_0,\rhos_1',\rhos_1,\dots,\rhos_r\})
  =  (B,\x,\x',\{\rhos_0\})\#(B,\x',\y,\{\rhos_1',\rhos_2,\dots,\rhos_r\}),\]
  where $\rhos_0$ consists a single chord, which is a portion of $\xi_0$. 

  {\bf Case 2.} All strands are stationary.
  
  {\bf Case 3.} There are no right-moving strands.  There are two
  cases: either the leftmost left-moving strand terminates before
  hitting the wall, or it hits the wall.  See Figure~\ref{fig:Cases}
  for illustrations. (We have drawn the five non-trivial examples of
  initial domain $B_0$ in the order they appear in the text above.)

\begin{figure}[ht]
\input{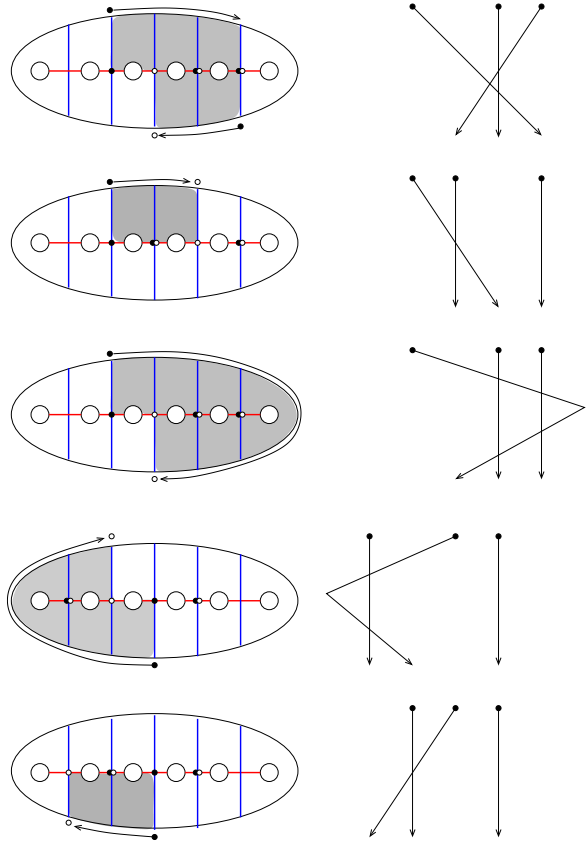}
\caption{\label{fig:Cases}
{\bf{Decomposing the domain.}}}
\end{figure}

In all cases, we have decomposed our domain as a juxtaposition of
simpler pieces $B=B_0\# B'$, the first of which has a single chord
packet on $\partial_\beta$). Verifying Equation~\eqref{eq:RemainsToSee}
for the domains $B_0$ is straightforward, and hence the equation holds in
general, by induction.
\end{proof}

\begin{proof}[Proof of Theorem~\ref{thm:ConstructY}.]
  Let $\HD_m$ denote the genus $0$ $\alpha$-$\beta$-bordered diagram
  as in Definition~\ref{def:HalfIdentity}. The module $\Ymod$ is the
  associated bimodule $\CFBA(\HD)$. It is an ${\mathcal A}_{\infty}$
  bimodule by Theorem~\ref{thm:HolomorphicAAbimodule}.
  Property~\ref{Y:Idempotents} is straightforward.  Bigons representing
  the actions from Properties~\ref{Y:LR},~\ref{Y:0},~\ref{Y:m} are
  straightforward to see; compare Figure~\ref{fig:HalfIdentityFlows}.

  \begin{figure}[ht]
\input{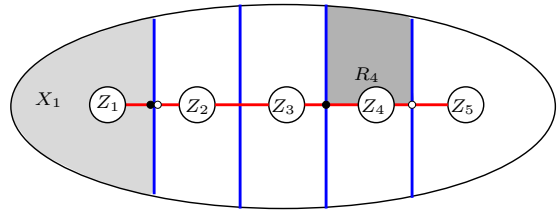}
\caption{\label{fig:HalfIdentityFlows}
  {\bf{Bigons representing actions.}}
  Consider $\HD_5$ with $k=2$. The black dots represent a generator,
  corresponding
  to the idempotent state $\x=(1,3)$, and the white dots represent $\y=(1,4)$.
  The darkly shaded bigon represents the action
  $m_{1|1|1}(L_4,\x,R_4)=\y$; the lightly shaded one represents
  $m_{1|1|1}(X_1,\x,U_1)=\x$.}
\end{figure}

The verification of Property~\ref{Y:X} is verified as follows.
Consider the moduli space ${\mathfrak M}^1$ of curves with the following properties:
\begin{itemize}
\item $\Source$ has genus zero, two boundary punctures labelled by $+$, two boundary punctures labelled by $-$, which divide the boundary into two $\alpha$-boundaries and two $\beta$-boundaries.
\item $\Source$ has a single puncture on its $\alpha$-boundary, which is labelled
with the Reeb chord $R_i L_i$.
\item $\Source$ has two punctures $q_1$ and $q_2$ on the $\beta$-boundary,
which are labelled $R_i$ and $L_i$ (on $\bdyB$).
\end{itemize}
There is another moduli space ${\mathfrak M}^2$ defined as above, only now using
the chord $L_i R_i$. 

Boundary monotone elements of ${\mathfrak M}^1\cup{\mathfrak M}^2$ are
precisely those elements for which $t(q_1)=t(q_2)$. 

The moduli space ${\mathfrak M}^1$ can be concretely understood as an
interval, parameterized by a cut parameter: cut in along $\alpha_i$
and use the Riemann mapping theorem to obtain a quadrilateral.  Each
quadrilateral has a single non-trivial involution, which in turn is
equivalent a branched cover of the disk, to obtain the desired element
of ${\mathfrak M}^1$.

The moduli space ${\mathfrak M}^1$ has a join curve end, where the
constraint packet on $Z_i$ is the set $\{L_i,R_i\}$. Explicitly, this
is realized as the cut along $\alpha_{i+1}$ extends from $\beta_i$ to
$Z_i$. Let $u_0$ denote the main component of this join curve.  Note
that this consists of a union of two bigons which meet along their
$\alpha$-boundary at a puncture. For generic choices of complex structure, 
$t(u_0(q_1))\neq t(u_0(q_2))$.

This moduli space has another end as the cut parameter goes to
$0$.  The limiting object in this case is a union of two bigons which
meet at a single point on the $\beta$-boundary, so that both
punctures $q_1$ and $q_2$ occur on the same bigon (separated by the
node). Let $u_1$ denote the bigon component containing $q_1$ and $q_2$.
Clearly, $t(u_1(q_1))>t(u_1(q_2))$. 

The moduli space ${\mathfrak
  M}^2$ also has a join curve end, whose main component is $u_0$
(i.e. it is  the main component of the join curve end in ${\mathfrak M}^1$).
It also has another end where the cut parameter goes to $0$, which is identified with a union of two bigons, one of which contains both punctures $q_1$ and $q_2$. Let $u_2$ denote this bigon. In this case, $t(u_2(q_1))<t(u_2(q_2))$.
See Figures~\ref{fig:YX} and~\ref{fig:YXmod}.

\begin{figure}[ht]
\input{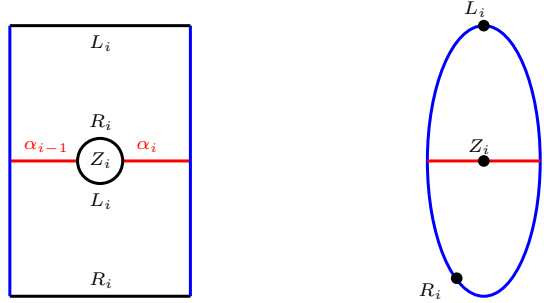}
\caption{\label{fig:YX}
{\bf{The region $B_i$}.} The region responsible for the action~\ref{Y:X}.
On the left is a combinatorial representation; on the right is a  conformally accurate model (with cylindrical ends/punctures replacing boundary).}
\end{figure}

\begin{figure}[ht]
\input{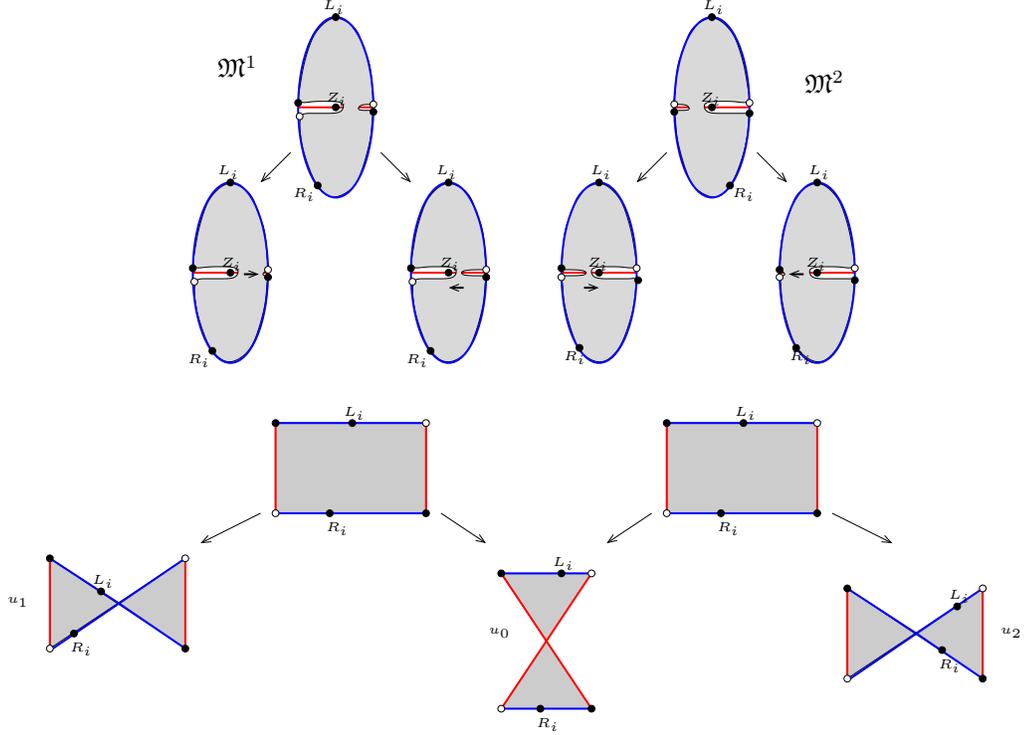}
\caption{\label{fig:YXmod}
{\bf{Verifying the Action~\ref{Y:X}}.} 
The moduli spaces ${\mathfrak M}^1$ and ${\mathfrak M}^2$ are represented
by the cuts indicated in the first line of pictures; the conformal models
of the domain curves are shown on the second line of pictures.}
\end{figure}

It follows that if we let ${\mathfrak M}$ denote the union of
${\mathfrak M}_1$ and ${\mathfrak M}_2$ glued along their join curve
end, then the map $u\mapsto t(u(q_1))-t(u(q_2))$ defines a proper map
of degree $1$ near $0$. In particular, the preimage of $0$ has an odd number of points; but this count specifies the  $Y_{\x}$-component in the
action $m_{1|1|1}(X_i,Y_{\x},U_i)$) 
\end{proof}

\section{The $DD$ bimodule is quasi-invertible}
\label{sec:Inverses}

We abbreviate $\Clg=\Clg(m,k)$ and $Q=\Quot{m}{k}$. Our aim here is to prove the following:

\begin{thm}
  \label{thm:QuasiInverses}
      The bimodules $\lsup{\Clg}\DDmod^{Q}$ from Section~\ref{sec:DDmod} 
      and $\lsub{Q}\AAmod_{\Clg}$ from Section~\ref{sec:AAmod} are
      quasi-inverses of one another.
\end{thm}

This will be a key step in computing the cobar algebra of $\Clg$ in the next section.
The proof will occupy the rest of this section.

As a preliminary step, we must establish the boundedness assumptions required
to ensure that $\lsub{Q}M^\Clg=\lsub{\Clg}\AAmod_Q\DT \lsup{Q}\DDmod^\Clg$ exist; i.e. we must prove that
the structure maps in the tensor products are finite.

To this end,
fix an input string
$a_1,\dots,a_p$ in $Q$, and let $c=\sum_{j=1}^p \weight(a_j)$.
The following two statements hold:
\begin{itemize}
\item If
  $m_{p|1|q}(a_1\otimes\dots\otimes a_p \otimes x \otimes b_1\otimes\dots\otimes b_q)\neq 0$, then
  $c=\sum_{j=1}^q\weight_i(b_j)$.
  (This follows from  Property~\ref{Y:graded} from Theorem~\ref{thm:ConstructY}.)
\item 
  If $\delta^{1}({\mathbf 1})=b\otimes {\mathbf 1}\otimes e$,
  then $\weight(b)\geq \OneHalf$.
\end{itemize}
Thus, if the sum of the weights of the inputs is $c$, then the
number of type $DD$ operations which can pair non-trivially is bounded above by $2c$.
    
  Consider now the $DA$ bimodule $M=\lsub{Q}\AAmod_{\Clg}\DT
  ~\lsup{\Clg}\DDmod^{Q}$.  Its generators correspond to basic
  idempotents, since the same is true for $\AAmod$ and
  $\DDmod$. Moreover, the Maslov/Alexander multi-graded bimodule on
  the tensor factors induces one on the tensor product; i.e.  if
  $s,t\in M$, and $n\geq 0$, and there is a non-trivial action
  \[ \delta^1_{n+1}(a_1,\dots,a_{n},s)=t\otimes b,\]
  then for for all $i=1,\dots,m$,
  \begin{equation}
    \label{eq:WeightPreserved}
    \weight_i(a_1)+\dots+\weight_i(a_{n})=\weight_i(b);
  \end{equation}
  and 
  \begin{equation}
    \label{eq:RespectGrading}
    \cross(a_1)+\dots+\cross(a_{n})-n+1=\cross(b).
  \end{equation}

\begin{lemma}
  \label{lem:Morphisms}
  There are two Maslov/Alexander bigraded 
  ${\mathcal A}_{\infty}$ homomorphisms 
  \[ \phi\colon Q\to Q \qquad{\text{and}}\qquad
  \psi\colon \Clg\to \Clg,\]
  so that 
  \[
  \lsub{Q}[\phi]^{Q}=\lsub{Q}\AAmod_{\Clg}\DT 
  ~\lsup{\Clg}\DDmod^{Q}\qquad\text{and}\qquad
  \lsup{\Clg}[\psi]_{\Clg}=~\lsup{\Clg}\DDmod^{Q}\DT 
  \lsub{Q}\AAmod_{\Clg}.
  \]
\end{lemma}

\begin{proof}
  Consider $\lsub{Q}M^{Q}=\lsub{Q}\AAmod_{\Clg}\DT 
  ~\lsup{\Clg}\DDmod^{Q}$.
  From Equation~\eqref{eq:WeightPreserved}, it follows easily that
  $\delta^1_1$ vanishes. By Lemma~\ref{lem:BimoduleOfPhi} (which is~\cite[Lemma~2.2.50]{Bimodules}), it follows that there is a  Maslov/Alexander multi-graded
  $\Ainfty_\infty$ homomorphism $\phi\colon Q\to Q$
  so that $\lsub{Q}[\phi]^Q\cong \lsub{Q}M^Q$.

  The analogous statements for $[\psi]$ follow similarly
  (except that in this case, we replace the grading $\cross$ on $Q$
  with the trivial grading on $\Clg$).
\end{proof}

\begin{lemma}
  The map $\psi$ from Lemma~\ref{lem:Morphisms} is the identity map;
  i.e. 
  \[ \lsup{\Clg}\DDmod^{Q}\DT 
  \lsub{Q}\AAmod_{\Clg}=\lsup{\Clg}\Id_{\Clg}.\]
\end{lemma}

\begin{proof}
  In principle, $\psi=\{\psi_n\colon C^{\otimes n}\to
  C\}_{n=1}^{\infty}$.  Gradings ensure that $\psi_n=0$ for $n>1$.
  (Instead of Equation~\eqref{eq:RespectGrading}, we are now using the
  trivial grading on the algebra; and the analogue of that equation
  simply ensures that $n=1$.)  By Theorem~\ref{thm:ConstructY}, we
  have
  \begin{align*}
    \psi(L_i)=L_i &\qquad\forall i=1,\dots,m \\
    \psi(R_i)=R_i &\qquad\forall  i=0,\dots,m-1 \\
    \psi(U_i)=U_i &\qquad\forall  i=0,\dots,m.
  \end{align*}
  Since $\Clg$  is generated (over the idempotent algebra)
  by the $L_i$, $R_i$, and $U_i$  as  above, it follows that
  $\psi$ is the identity  map.
\end{proof}

It is easy enough to see that $\psi\circ \phi\colon \Clg\to \Clg$ is
the identity map. To see that $\phi\circ \psi\colon Q\to Q$ induces
the identity map on homology, we use the following proposition. Recall
the generating set of atomic generators $\Idemp{\x}\cdot \Xgen{i}{j}$,
$\Idemp{\x}\cdot \Rij$ and $\Idemp{\x}\cdot \Lji$
(Definition~\ref{def:Atomic}).  For $i=1,\dots,m$, let $X_i$ denote
$\Xgen{i-1}{i}$.

\begin{prop}
  \label{prop:RigidQ}
  If $f\colon \Quot{m}{k}\to \Quot{m}{k}$ is an ${\mathcal A}_{\infty}$ homomorphism
  with the following properties:
  \begin{itemize}
    \item $f(X_{i-1,i})=X_{i-1,i}$ for all $i=1,\dots,m$.
    \item $f(L_{i,i-1})=L_{i,i-1}$ for all $i=2,\dots m$.
    \item $f(R_{i,i+1})=R_{i,i+1}$ for all $i=1,\dots,m-1$.
  \end{itemize}
  Then, $f$ induces the identity
  map on homology.
\end{prop}

Proposition~\ref{prop:RigidQ} is proved in Section~\ref{subsec:RigidQ}.
We prove Theorem~\ref{thm:KoszulDuality}, assuming this result, presently:

\begin{proof}[Proposition~\ref{prop:RigidQ}$\Rightarrow$
Theorem~\ref{thm:QuasiInverses}]
  From
  Properties~\ref{Y:X},~\ref{Y:0}, and~\ref{Y:m},
  it follows that $\phi(X_i)=U_i$ and
  $\psi(U_i)=X_i$.  Thus, $\psi\circ \phi(X_i)=X_i$.  Thus, by
  Proposition~\ref{prop:RigidQ} implies that $\psi\circ\phi\simeq \Id$.
  It is straightforward to see that $\phi\circ \psi=\Id$.
\end{proof}

\subsection{On the homology of $\Quot{m}{k}$}
\label{subsec:RigidQ}

To complete the proof of Theorem~\ref{thm:QuasiInverses}, it remains to verify
Proposition~\ref{prop:RigidQ}. We start with the case where $k=m-1$.

We wish to glean information from the ${\mathcal A}_{\infty}$
structure on $H(\Quot{m}{m-1})$. To this end, let $\phi\colon
H(\Quot{m}{m-1})\to \Quot{m}{m-1}$ be an $\Ainfty_{\infty}$
homomorphism.
Given $T\subset \{1,\dots,m\}$, let ${\mathfrak S}(T)$ denote the set
of one-to-one correspondences $\tau\colon \{1,\dots,\ell\}\to T$.  Let
\[\Phi(T)=\sum_{\tau\in{\mathfrak S}(T)}
\phi_\ell(X_{\tau(1)},\dots,X_{\tau(\ell)}). \]

\begin{lemma}
  \label{lem:AinfQ}
  Let $\phi\colon H(\Quot{m}{m-1})\to \Quot{m}{m-1}$ be a graded ${\mathcal
    A}_{\infty}$ homomorphism with $\phi(X_i)=X_i$, and fix a subset
  $T\subset \{1,\dots,m\}$.  If $T$ is not an interval, then
  $\Phi(T)=0$.  If $T$ is an interval from $i$ to $j$ for some $i<j$,
  with $(i,j)\neq (1,m)$, then $\Phi(T)=X_{i,j}$.
\end{lemma}

\begin{proof}
  We consider the sum of the ${\mathcal A}_{\infty}$ relations on the
  homomorphism $\phi$, where we sum over all permutations of the
  elements $X_i$ with $i\in T$
:
  \begin{itemize}
  \item for any sequence of distinct integers
    $i_1<\dots<i_\ell$ with $\ell<m$, gradings ensure that
    $\mu_\ell(X_{i_1},\dots,X_{i_\ell})=0$.
  \item for $1\leq i<j\leq m$, we have that $[X_i][X_j]=[X_j][X_i]$ in $H(\Quot{m}{m-1})$.
    There are two cases: if $j=i+1$, the homology is supplied by $X_{i-1,i+1}$.
    If $j>i+1$, then $X_i \cdot X_j = X_j \cdot X_i$.
    \end{itemize}
    It follows that
  the ${\mathcal A}_{\infty}$ relation shows that for any $T$, 
  \[ d\Phi(T)=\sum_{T=T_1\coprod T_2} \Phi(T_1)\cdot \Phi(T_2).\]
  By  the inductive hypothesis, if $T=\{i,\dots,j\}$ is an interval (different from $(1,m)$), then
  \[ d\Phi(T)=\sum_{\{k\mid i<k<j\}}X_{k,j}\cdot X_{i,k}.\]
  It follows that $\Phi(T)$ is non-zero.
  Now, $2\weight(\Phi(T))=\Cr(\Phi(T))+1$, so
  by  Proposition~\ref{prop:AtomicElements}, we can conclude
  that
  $\Phi(T)=X_{i,j}$.
\end{proof}

The above lemma has the following easy consequence:

\begin{prop}
  \label{prop:AinfQ}
  For any induced ${\mathcal A}_\infty$ structure on
  $\Field[X_1,\dots,X_{m},\Omega]/(X_i^2=0)_{i=1}^m$ induced from the 
  isomorphism of Equation~\eqref{eq:QmSpecial}, 
  we have that
  \[ \sum_{\sigma\in{\mathfrak S}_m}\mu_m(X_{\sigma(1)},\dots,X_{\sigma(m)})=\Omega.\]
\end{prop}

\begin{proof}
  The induced  $\Ainfty_{\infty}$ structure on $H(\Quot{m}{m-1})$ can be computed
  from the homological perturbation lemma~\cite{Keller}.
  Consider any ${\mathcal A}_{\infty}$ homomorphism
  $\phi\colon H(\Quot{m}{m-1})\to \Quot{m}{m-1}$ which induces the identity
  map on homology. 
  Combining the ${\mathcal A}_{\infty}$ relation on $\phi$ with
  Lemma~\ref{lem:AinfQ}, we find that
  \begin{align*}
    \phi_1&\left(\sum_{\sigma\in{\mathfrak S}_m}\mu_m (X_{\sigma(1)},\dots,X_{\sigma(m)}) \right)
    \\&\sim 
    \sum_{0<\ell<m}\Big( \Phi(\{0,\dots,\ell\})\cdot \Phi(\{\ell+1,\dots,m\})+
    \Phi(\{\ell+1,\dots,m\})\cdot \Phi(\{0,\dots,\ell\})\Big) \\
    &=\sum_{0<\ell<m} X_{0,\ell}\cdot X_{\ell,m}+X_{\ell,m}+\cdot X_{0,\ell} \\
    &=\Omega.
  \end{align*}
  Since $\phi_1$ induces the identity map on homology, the result follows.
\end{proof}

We can now prove Proposition~\ref{prop:RigidQ} when $m=k$.

\begin{lemma}
  \label{lem:RigidQm}
  Let $\phi'\colon H(\Quot{m}{m-1})\to H(\Quot{m}{m-1})$ be a graded
  ${\mathcal A}_{\infty}$ homomorphism with $\phi'([X_i])=[X_i]$. Then
  $\phi'_1$ induces the identity map on homology.
\end{lemma}

\begin{proof}
  Define
  \[\Phi'(T)=\sum_{\tau\in{\mathfrak S}(T)}
  \phi'_\ell([X_{\tau(1)}],\dots,[X_{\tau(\ell)}]). \]
  As in the proof of Proposition~\ref{prop:AinfQ},
  the ${\mathcal A}_{\infty}$ relation on $\phi'$ and Lemma~\ref{lem:AinfQ}
  ensures that
  \[ \phi_1'\left(\sum_{\sigma\in{\mathfrak S}_m} \mu_m(X_{\sigma(1)},\dots,X_{\sigma(m)}\right)=\Omega.\]
  But by Proposition~\ref{prop:AinfQ}, the left-hand-side is
  $\phi_1'(\Omega)$, so we conclude that the map on homology induced
  by $\phi_1'$ fixes $\Omega$, in addition to the $[X_i]_{i=1}^m$. But
  $\phi_1$ is an algebra map, and $H_*(\Quot{m}{m-1})$ is generated by
  $[X_i]$ and $\Omega$; so we can conclude that $\phi_1'$ fixes all of
  $H_*(\Quot{m}{m-1})$.
\end{proof}

We consider next the case of $\Quot{m}{k}$ with $k<m-1$.

\begin{lemma}
  \label{lem:GenerateHomologyOfQ}
    The ring $H(Q)$ is generated by the elements $L_{i}$,
    $R_i$, $\Omega$, and the cycles of the form ${\mathcal X}_{i,j}$.
\end{lemma}

\begin{proof}
  Consider $Q(\x,\y;w)$.
  Suppose $\x=\y$. Then, there is an isomorphism of rings
  $\Idemp{\x}\cdot \Quot{m}{k}\cdot\Idemp{\x}\cong\Quot{k+1}{k}$,
  which identifies the generators $X_{i,j}$ of $\Quot{k+1}{k}$ with
  ${\mathcal X}{x_i,x_j}$ in $\Quot{m}{k}$. Thus, the statement follows from
  Corollary~\ref{cor:QmSpecial}.

  Suppose that $\x\neq \y$.  Let $a\in H(Q(\x,\y;w))$ be non-zero. We
  prove that $a$ can be factored as above, by induction on the
  maximum of $w_i$.  Suppose that the maximum is $\leq 1$, but $w$ is
  non-constant.  Then such an element $a$ can be easily seen to be
  factored as a product of $L_i$, $R_i$ and 
  ${\mathcal X}_{i,j}$. (This factorization is actually constructed in
  the proof of Lemma~\ref{lem:TwoValues}.) For the inductive step,
  suppose that $w$ is a vector whose maximum is greater than
  $1$. Then, by Theorem~\ref{thm:HomologyQp}, we have that that
  $H(Q'(\x,\y;w))=0$. It follows from the long exact sequence of the
  mapping cone that $a\in \Omega\cdot H(Q(\x,\y;w-\vec{1}))$, and the
  result follows from the inductive hypothesis.
\end{proof}

\begin{proof}[Of Proposition~\ref{prop:RigidQ}]
  For any idempotent $\x$, there is a ring isomorphism
  \[ \Idemp{\x}\cdot \Quot{m}{k}\cdot \Idemp{\x}\cong \Quot{k+1}{k}.\]
  identifying ${\mathcal X}_{x_i,x_j}$ with some $\Xgen{i}{j}$, It
  follows now from Lemma~\ref{lem:RigidQm} that if $\phi$ fixes all
  the ${\mathcal X}_{i,j}$, then in fact $\phi$ fixes the entire
  subring $\sum_{\x}\Idemp{\x}\cdot\Quot{m}{k}\Idemp{\x}$. Note that
  $\Omega$ lies in this subring.
  Now, by Lemma~\ref{lem:GenerateHomologyOfQ}, $\phi$ fixes the generators of 
  $H(Q)$; and hence it must fix all of $H(Q)$.
\end{proof}

With the verification of Proposition~\ref{prop:RigidQ} in place, the
proof of Theorem~\ref{thm:QuasiInverses} is complete.

\newcommand\BigCobar{\mathbf{Cobar}}
\newcommand\kDD{\mathbb K}
\newcommand\kAA{\mathbb L}

\section{Identifying the cobar algebra}
\label{sec:Duality}

Choose integers $m$ and $k$ with $0<k<m$.
The aim of this section is to describe the cobar algebra of
$\Clg(m,k)$ in terms of the pong algebra. Before doing this, we begin
with a few general remarks about the cobar algebra of an associative
algebra.

Let $A$ be an augmented algebra, as in definition~\ref{def:AugmentedAlgebra}.
The {\em big cobar algebra} of $A$ is a graded algebra, defined as
follows. As a $\Ground$-$\Ground$ bimodule, it is given by
\[ \BigCobar(A)=\bigoplus_{n=0}^{\infty} \Hom(A_+^{\otimes n},\Ground).\]
Multipliciation is obtained by juxtaposing tensor summands.
The differential is induced by $\mu_2^*$. The summand 
$\Hom(A_+^{\otimes n},\Ground)$ is thought of as living in grading $-n$; and we
call this induced grading the {\em length grading}.

We specialize to the case of $A=\Clg(m,k)$.  
As before,  cobar
algebra has a $\Z$-valued grading, defined so that
$\Hom(A_+^{\otimes n},\Ground)$ has grading $-n$. We also have an
Alexander grading with values in $(\OneHalf \Z)^m$ which is
induced from the sums of the weights of the algebra elements in
$\Clg(m,k)$.  For example in $\Clg(4,2)$, the element $U_2^*\otimes
U_4^*$ has homological grading $-1$ and Alexander multi-grading
$(0,1,0,1)$. Moreover, its differential is the element
\[ R_2 ^* \otimes L_2 ^*\otimes U_4^* + L_2 ^*\otimes R_2^* \otimes U_4^*.\]

Let
\[ \Cobar(C)_{n,\vec{w}}=\Hom([C^{\otimes n}_+]_{\vec{w}},\Ground).\]
The {\em Maslov/Alexander graded cobar algebra}, or simply, the cobar
algebra, is defined  by
\[ \Cobar(C)=\bigoplus_{n,\vec{w}}\Cobar(C)_{n,\vec{w}}.\] (The big
cobar algebra, by contrast, would be an infinite direct product, indexed by
the various choices of weight vector.)

The algebra $\Quot{m}{k}$ has a $\Z$-grading
$\Ngr$ as specified in Equation~\eqref{eq:Ngr}.
The differential preserves the weight vector, with values in $\OneHalf \Z^m$. 

Our aim here is to prove the following:
\begin{thm}
  \label{thm:KoszulDual}
  Choose $0<k<m$.  There is a graded quasi-isomorphism
  $\Cobar(\Clg(m,k))$ to $\Quot{m}{k}$, which respects the Alexander
  gradings, and which identifies the length grading on $\Cobar$
  with $\Z$-grading $\Ngr$ on $\Quot{m}{k}$.
\end{thm}

\begin{example}
  Consider the case $k=1$ and $m=2$. Then, $C=\Clg(2,1)\cong
  \Field[U_1,U_2]/U_1 U_2=0$. It is easy to see that the algebra
  $H(\Cobar(C))$ is generated by $U_1^*$ and $U_2^*$. Moreover, there
  are relations $[U_1^*]\otimes [U_1^*]=0$ and $[U_2^*]\otimes
  [U_2^*]=0$. Thus, $H(\Cobar(C))$ is isomorphic to $\Quot{m}{1}$, under
  an isomorphism that identifies $[U_1^*]$ and $\Xgen{0}{1}$ and $[U_2^*]$
  with $\Xgen{1}{2}$.
\end{example}

As usual, we abbreviate
$C=\Clg(m,k)$, $Q=\Quot{m}{k}$.

Theorem~\ref{thm:KoszulDual} will follow from Koszul duality, in
the form described in~\cite{HomPairing}, with a suitable adaptation of
some of the finiteness conditions from~\cite{TorusAlg}.
The proof follows the strategy of~\cite[Proposition~8.11]{HomPairing}.
Formally, the construction there used two
quasi-inverse bimodules between $C$ and $\Cobar(C)$,
$\lsup{\Cobar(C)}\kDD^{C}$ and $\lsub{\Cobar(C)}\kAA_{C}$.
The type $AA$ bimdodule $\lsub{\Cobar(C)}\kAA_C$, called the {\em Kronecker module},
is defined as
follows.  Its underlying $\Ground$-bimodule has rank one, and it is
equipped with the following operations.
All $m_{p|1|n}=0$  except for the two unital actions
\[ m_{1|1|0}({\mathbf 1},x)=x \qquad{\text{and}}\qquad
m_{0|1|1}(x,{\mathbf 1})=x,\]
and actions $m_{p|1|n}$ with $p=1$, specified for  $\gamma\in \Cobar(A)_n$
and $a_1\otimes \dots\otimes a_n\in (A_+)^n$ by
\[ m_{1,1,n}(\gamma\otimes{\mathbf 1}\otimes (a_1\otimes\dots\otimes a_n))=\gamma(a_1\otimes\dots\otimes a_n).\]

We would also like a bimodule
$\lsup{C}\kDD^{\Cobar(C)}$
has generators corresponding to idempotents
(which we suppress), and
\[\delta^1(\One)=\sum_{c} c\otimes c^*,\]
where the sum is taken over the basic generators $c$ of $C$.
Unlike~\cite{HomPairing}, this sum is not finite, since the algebra is
infinitely generated; rather the differential could be interpreted as
an element of some completion of $C\otimes\Cobar(C)$;
compare~\cite{TorusAlg}.
Rather than pursuing this perspective, we reformulate the proof of
Theorem~\ref{thm:KoszulDual} without explicitly referring to $\lsup{C}\kDD^{\Cobar(C)}$.

To this end,
abbreviate the differential of the type $DD$ bimodule from
Section~\ref{sec:DDmod} as $\delta^1(\One)=c\otimes q$,
so the $DD$ bimodule relation takes the form
\[ c\otimes dq + (c\cdot c) \otimes (q\cdot q)=0\]

We construct a DG algebra homomorphism $\Phi\colon \Cobar(C)\to Q$ as follows.
Given $\gamma\in \Cobar_n(C)$, let
\begin{equation}
  \label{eq:DefPhi}
  \Phi(\gamma)=\gamma(\overbrace{c\otimes\dots\otimes c}^n)\cdot q^n.
\end{equation}

\begin{lemma}
  \label{lem:DefPhi}
  The formula above defines a $DG$ algebra homomorphism
  $\Phi\colon \Cobar(C)\to Q$, where the $\Z$-grading on $\Cobar(C)$
  is given by the length grading, and the $\Z$-grading on $Q$ is 
  $\Ngr=\cross-2\Totweight$. Viewing $\Phi$ as an $A_{\infty}$ homomorphism
  with $\Phi_i=0$, the associated type $DA$ bimodule 
  $\lsub{\Cobar(C)}[\Phi]^Q$ satisfies
  \[ \lsub{\Cobar(C)}[\Phi]^Q \simeq \lsub{\Cobar(C)}\kAA_{C}\DT~ \lsup{C}\kDD^{Q}.\]
\end{lemma}

\begin{proof}
  To see that $\Phi$ respects the $\Z$-gradings, 
  note that if $\gamma$ is homogeneous and
  $\gamma(\overbrace{c\otimes\dots\otimes c}^n)\neq 0$, then the
  length grading of $\gamma$ is $-n$. By
  Proposition~\ref{prop:GradedDD}, $\Ngr(q)=-1$, so
  $\Ngr(q^n)=-n$. 
  
  The other properties are similarly straightforward.
\end{proof}

We describe next a way of turning (certain) type $AA$ bimodules into
type $DA$ bimodules, where the type $D$ side is a cobar algebra.  
In
the interest of clarity, we state the next lemma in slightly more
generality than we need. Specifically, we explained earlier how to
give $\Cobar(C)$ a Maslov/Alexander grading, whose $\Z$-component is
specified by the length grading. Note that the algebra $C$ has a trivial $\Z$-grading.
By contrast,  if $B$ has a non-trivial $\Z$-grading, 
\[ B= \bigoplus_{d\in \Z, \vec{w}\in \OneHalf \Z^m}B^{d,\vec{w}},\]
we endow $\Cobar(C)$ with a $\Z$-grading
so that the summand 
\[ 
(B^{d_1,\vec{w}_1}_+\otimes\dots\otimes B^{d_n,\vec{w}_n}_+)^*
\subset \Cobar(C)\]
has $\Z$-grading $\gr^{\Cobar(B)}$ given by $d_1+\dots+d_n-n$.

\begin{lemma} 
  Let $\lsub{A}M_B$ be a Maslov/Alexander graded bimodule
  with of rank one, whose generator, which we denote
  ${\bf 1}$, is supported in vanishing Maslov/Alexander bigrading.
  There is an associated
  (Maslov/Alexander-graded) type $DA$ bimodule
  $\lsub{A}{K(M)}^{\Cobar(B)}$, characterized by
  \[ \langle \delta^1_{p+1}(a_1,\dots,a_{p},{\bf 1}), b_1,\dots,b_q\rangle
  = m_{p|1|q}(a_1,\dots,a_p,{\bf 1},b_1,\dots,b_q).\]
\end{lemma}
\begin{proof}
  Note that for each fixed sequence $(a_1,\dots,a_p)$,
  the grading properties of $M$ ensure that there is an upper bound on 
  $q$ so that there is a sequence $(b_1,\dots,b_q)$ with
  $m_{p|1|q}(a_1,\dots,a_p,\One,b_1,\dots,b_q)\neq 0$. Thus,
  $\delta^1_{p+1}$ lands in $\Cobar(B)$.

  The $DA$ bimodule relation on $\lsub{A}{K(M)}^{\Cobar(B)}$ is a
  straightforward consequence of the $AA$ bimodule relation on
  $M$. The grading properties also follow easily. For example, if
  $m_{p|1|q}(a_1\otimes\dots \otimes a_p\otimes {\bf 1}\otimes b_1\otimes\dots\otimes b_q)\neq 0$, then 
  by Equation~\eqref{eq:GradedBimodule},
\begin{align*}
  \sum_{i=1}^p (\gr^A(a_i)) &=
  \left(-q -\sum_{i=1}^q \gr^B(b_i)\right)+
  1-p \\
  &= \gr^{\Cobar(B)}((b_1\otimes \dots\otimes b_q)^*) +
  1-p ;
\end{align*}
i.e. $\lsub{A}{K(M)}^{\Cobar(B)}$ is $\Z$-graded.
\end{proof}

\begin{lemma}
  \label{lem:Kfunctor}
  The above assignment $\lsub{A}M_B\Rightarrow \lsub{A}K(M)^{\Cobar(B)}$
  satisfies the following properties:
  \begin{enumerate}[label=(K-\arabic*),ref=(K-\arabic*)]
  \item
    \label{KDA}
    If $A$, $B$, and $C$ are all over the same ring of idempotents,
    and $\lsub{A}M^B$ is a type $DA$ bimodule of rank one,
    $\lsub{B}N_C$ is a type $AA$ bimodule of rank one, then
    \[ K(\lsub{A} M^{B}\DT \lsub{B}N_C)
    \simeq
        \lsub{A}M^B\DT\lsub{B} K(N)^{\Cobar(C)}.\]
      \item
        \label{Kinv1}
        If $\lsub{\Cobar(C)}\kAA_{C}$ is the Kronecker module,
        then $\lsub{\Cobar(C)}K(\kAA)^{\Cobar(C)}$ is the identity bimodule.
      \item
        \label{Kinv2}
        If $\lsub{A}Y_B$ is a type $AA$ bimodule, then
        \[ \lsub{A}K(Y)^{\Cobar(B)}\DT \lsub{\Cobar(B)}\kAA_{B}=
          \lsub{A}Y_{B}.\]
      \item
        \label{K:Map}
        if $\lsub{A}M_B$ has the property that
        $m_{0|1|r}=0$ for all $r\geq 0$, then
        there is an $A_\infty$ homomorphism
        $\Psi\colon A\to \Cobar(B)$ so that
        \[ \lsub{A}[\Psi]^{\Cobar(B)}=K(M).\]
      \item 
        \label{K:funct}
        If $M\sim M'$, then $K(M)\sim K(M')$.
  \end{enumerate}
\end{lemma}

\begin{proof}
  The statements are all easy consequences of the definitions.
\end{proof}

\begin{remark}
  When $B_+$ is nilpotent, then $\lsup{B}\kDD^{\Cobar}$ as
  in~\cite{HomPairing} is a type $DD$ bimodule, and
  $\lsub{A}K(M)^{\Cobar(B)}$ agrees with
  the tensor product $\lsub{A}M_{B}\DT \lsup{B}\kDD^{\Cobar(B)}$.
  In that case, the first property can then be thought of as associativity of $\DT$,
  and both properties~\ref{Kinv1} and~\ref{Kinv2} are manifestations of the
  fact that
  $\kAA$ and $\kDD$ inverses.
\end{remark}

Define $\Psi\colon Q \to \Cobar(C)$ by
\[ \lsub{Q}[\Psi]^{\Cobar(C)}=K(\lsub{Q}\AAmod_{C}).\]
This definition is valid by Property~\ref{K:Map}, and the grading properties
of $\AAmod$, which ensure that $m_{0|1|r}\equiv 0$ for all $r\geq 0$.

\begin{proof}[of Theorem~\ref{thm:KoszulDual}]
  Combining Lemma~\ref{lem:DefPhi}, Lemma~\ref{lem:Kfunctor}~Property~\ref{KDA}, associativity of $\DT$ (as in~\cite[Proposition~2.3.15]{Bimodules},
  the hypothesis that $X$ and $Y$ are quasi-inverses, and Lemma~\ref{lem:Kfunctor}~Property~\ref{Kinv1},
  we see that:
  \begin{align*}
    [\Phi\circ \Psi]&= (\lsub{\Cobar(C)}\kAA_{C}\DT~\lsup{C}\DDmod^Q)\DT
    K(\lsub{Q}\AAmod_{C}) \\
    &= K((\lsub{\Cobar(C)}\kAA_C\DT~\lsup{C}\DDmod^Q)\DT \lsub{A}\AAmod_C) \\
    &\simeq K(\lsub{\Cobar(C)}\kAA_C\DT (\lsup{C}\DDmod^Q\DT \lsub{A}\AAmod_C)) \\
    &\simeq K(\lsub{\Cobar(C)}\kAA_C) \\
    &=\lsub{\Cobar(C)}[\Id]^{\Cobar(C)}.
  \end{align*}
  Similarly, combining Lemma~\ref{lem:DefPhi}, associativity of $\DT$
  (now in the easier case of type $AA$ bimodules;
  see~\cite[Lemma~2.3.14]{Bimodules}), Property~\ref{Kinv2}, and the
  hypotheses on $X$ and $Y$, we find that:
  \begin{align*}
    [\Psi\circ \Phi]&=     K(\lsub{Q}\AAmod_{C})
    \DT (\lsub{\Cobar(C)}\kAA_{C}\DT\lsup{C}\DDmod^Q)
    \\
    &=  (K(\lsub{Q}\AAmod_C)\DT_{\Cobar(C)}\kAA_{C})\DT~\lsup{C}\DDmod^Q \\
    &= \lsub{Q}\AAmod_C\DT~\lsup{C}\DDmod^Q \\
    &\simeq \lsub{Q}\Id^Q.
  \end{align*}
  It follows that $\Phi$ and $\Psi$ are quasi-isomorphisms.
  (See for example~\cite[Lemma~8.6]{HomPairing}.)
\end{proof}

\newcommand\sHH{\mathrm{sHH}}
\newcommand\Dsmall{D^s}

\newcommand\BigHC{\mathbf{HC}}
\newcommand\BigHH{\mathbf{HH}}
\section{Computing the ${\mathcal A}_{\infty}$ action on the homology of $\Pong{m}{k}$}
\label{sec:ComputeAinf}

Consider $\Clg(m,k)$ with $1\leq k\leq m-1$. In the applications to knot Floer
homology, we will consider $m=2n$ and $k=n$ (although we could have
alternatively considered $m=2n$ and $k=n-1$). 

\smargin{Isn't it actually in $\OneHalf\Z^m$?}
Recall that $\Clg(m,k)$ is equipped with a weight grading
\[ \weight=(\weight_1,\dots,\weight_m)\in\Z^{m}.\] There is also 
very uninteresting homological grading on $\Clg(m,k)$, defined so that all 
$c\in \Clg(m,k)$ have $\gr(c)=0$.

We extend this data to the polynomial algebra $\Clg(m,k)[t]$, with
the convention that $\weight_i(t)=1$ for all $i$; and $\gr(t)=2m-2k-2$.
This is consistent with the existence of a non-trivial $\mu_{2m-2k}$.

\begin{defn}
  For a {\em graded ${\Ainfty}_{\infty}$ structure} on $\Clg(m,k)[t]$, we
  require that for each operation $\mu_\ell$, we have that
  \begin{align*}
    \weight(\mu_{\ell}(a_1,\dots,a_\ell))&=\sum_{i=1}^{\ell}\weight(a_i) \\
    \gr(\mu_{\ell}(a_1,\dots,a_\ell))&=\ell-2+ \sum_{i=1}^{\ell}\gr(a_i).
  \end{align*}
\end{defn}

\begin{remark}
  The above grading conventions are chosen so that the differential $\mu_1$
  on a differential graded algebra drops grading by one.
\end{remark}

\begin{thm}
  \label{thm:CharacterizeActions}
  There is a graded $\Ainfty_{\infty}$ structure on ${\mathcal C}(m,k)[t]$,
  extending the natural algebra structure on ${\mathcal C}(m,k)$, and
  with a non-trivial $\mu_{2m-2k}$ operation.  Moreover, any two
  such extensions are quasi-isomorphic as graded $\Ainfty_\infty$ algebras.
\end{thm}

\begin{remark}
  The above result is not true when $k=0$: in that case, ${\mathcal
    C}(m,0)\cong \Field$, and there is no room for an
  $\mu_{2m}$-operation. 
\end{remark}

An explicit realization of the ${\mathcal A}_{\infty}$-algebra
characterized above is supplied by the homology of a pong algebra;
see Theorem~\ref{thm:HomologyPongAinfP} below.

Theorem~\ref{thm:CharacterizeActions} is proved using Hochschild
cohomology.  A small model for the Hochschild cohomology can be
obtained using a duality theorem, proved in Section~\ref{sec:Duality}.
The relevant computation is then given in
Section~\ref{sec:ComputeAinf}.

\subsection{Deformation theory of graded algebras}

We review here the deformation theory of graded algebras used in our
proof of Theorem~\ref{thm:HomologyPongAinf}. We follow here the
perspective from~\cite{TorusAlg}, with a somewhat simplified exposition,
owing to the fact that we are here in a $\Z$-graded setting.

Let $A=A_*$ be a graded algebra -- i.e. an algebra equipped with
a graded multiplication. Denote that $\Z$-grading by $\gr$.
(We assume for simplicity that $\mu_1=0$;
the case we will be interested in is a polynomial algebra over $\Clg(m,k)$.)
The {\em Hochschild complex of the $\Z$-graded graded algebra} is a bigraded chain complex associated to $A$.
$\BigHC^{n,d}(A)$ consists of $\Ground-\Ground$ bimodule maps
\[ \phi\colon A_+^{\otimes n} \to A, \]
(where $A_+$ is as in Definition~\ref{def:AugmentedAlgebra})
with
\begin{equation}
  \label{eq:Grphi}
  \gr(\phi(a_1,\dots,a_n))=d + n-1 + \left(\sum_i \gr(a_i)\right).
\end{equation}
The differential of the complex
\[ \partial \colon \BigHC^{n,d}\to \BigHC^{n+1,d-1} \]
is defined by 
\begin{align*}
  \partial \phi(a_1,\dots,a_{n+1})&= a_1\cdot \phi(a_2,\dots,a_{n+1})+\phi(a_1,\dots,a_n)\cdot a_{n+1}\\
  &+ \sum_{i=1}^{n} \phi(a_1,\dots,a_i \cdot a_{i+1}, \dots a_{n+1}).
\end{align*}

The (bigraded) homology of this complex is the {\em Hochschild cohomology}
\[ \BigHH^{*,*}=\bigoplus_{n,d}\BigHH^{n,d}.\]

An {\em ${\mathcal A}_n$ algebra}
is $A$ equipped with operations
\[ \mu_\ell \colon A_+^{\otimes \ell} \to A \in \BigHC^{\ell,-1}(A)\]
for $\ell\leq n$, 
satisfying the ${\mathcal A}_{\infty}$ relations with 
up to $n+1$ inputs.

The following is well-known:

\begin{prop}
  \label{prop:BigDefTheory}
  Each ${\mathcal A}_{n}$-algebra structure on $A$ can be extended to an
  $\Ainfty_\infty$-algebra structure if $\BigHH^{\ell+2,-2}(A)=0$ for
  all $\ell\geq n$.  The set of isomorphism classes of ${\mathcal A}_{n+1}$
  extensions of a given ${\mathcal A}_{n}$ structure is parameterized by
  $\BigHH^{n+1,-1}(A)$. Indeed, if $\BigHH^{\ell+1,-1}(A)=0$ for all
  $\ell\geq n$, there is a unique isomorphism class of
  $\Ainfty_\infty$-algebra structure extending any given ${\mathcal A}_n$ structure on
  $A$.
\end{prop}

\begin{proof}
  This is well-known. A proof is given in~\cite[Corollary~5.25]{TorusAlg}.
\end{proof}

We will be interested in $\Ainfty_\infty$ operations that preserve a further
Alexander grading $\weight$ on $A$ (with values in $(\OneHalf \Z)^m$); i.e.
\begin{equation}
  \label{eq:PreserveWeight}
  \weight(\mu_\ell(a_1\otimes\dots\otimes a_\ell))=\sum_{i=1}^\ell \weight(a_i).
\end{equation}
To this end, we consider the subcomplex of $\HC\subset \BigHC$
consisting of $\phi$ satisfying Equation~\eqref{eq:PreserveWeight}
(with $\phi_n$ in place of $\mu_n$) in addition to Equation~\eqref{eq:Grphi}.
Its homology $\HH$ is the Hochschild homology which we will use henceforth.

The graded analogue of Proposition~\ref{prop:BigDefTheory} is the following

\begin{prop}
  \label{prop:DeformationTheory}
  Each graded ${\mathcal A}_{n}$-algebra structure can be extended to a graded $\Ainfty_{\infty}$ structure
  if $\HH^{\ell+2,-2}(A)=0$ for all $\ell\geq n$. Moreover, 
  the 
  set of isomorphism classes of ${\mathcal A}_{n+1}$ extensions of a given
  ${\mathcal A}_{n}$ structure is parameterized by $\HH^{n+1,-1}(A)$. Indeed, 
  if  $\HH^{\ell+1,-1}(A)=0$ for all $\ell\geq n$, there is a unique
  isomorphism class of $\Ainfty_\infty$-algebra extending any given 
  ${\mathcal A}_n$-algebra structure on $A$.
\end{prop}

We give the graded Hochschild complex the following explicit description.
Consider the following graded algebra associated to $A$,
\[ {\mathcal E}^*=\bigoplus_{n=1}^{\infty} {\mathcal E}^n,\] where
\[{\mathcal E}^{n}= \prod_{\{s\in \z,\vec{w}\in|(\OneHalf \Z)^m\}} [t^s A]_{\vec{w}}\otimes_{\Ground\otimes\Ground} [A_+^{\otimes n}]^*_{\vec{w}}.\]
The tensor product here is shorthand for the the condition that
we are considering the span of
\[
b\otimes (a_1\otimes\dots\otimes a_n)^*\in [A]_{\vec{w}}\otimes
[A_+^{\otimes n}]^*_{\vec{w}} \]
so that the left idempotent of $b$ agrees with
the right idempotent of $(a_1\otimes\dots\otimes a_n)^*$;
and the right idempotent of $b$ agrees with the left idempotent of
$(a_1\otimes\dots\otimes a_n)^*=a_n^*\otimes\dots\otimes a_n^*$.

The cobar differential induces a differential (which could be denoted
$\Id_A\otimes \partial$; but instead we simplify notation a little)
\[ \partial\colon {\mathcal E}^n\to {\mathcal E}^{n+1}.\]
There is an element $E\in {\mathcal E}$ given by
\[ E=\sum_{e} e\otimes e^*,\]
where the sum is taken over a homogeneous generating set of $A$.
(In the case we consider, $A=\Clg(m,k)$, and the sum is taken 
over all pure algebra generators $e$.)

Let
$D\colon {\mathcal E}^n \to {\mathcal E}^{n+1}$ be the map defined by
\[ D(x)=\partial x + E\cdot x + x \cdot E.\]

\begin{lemma}
  \label{lem:HochschildDifferential}
  The map $D$ defined above is a differential on ${\mathcal E}$.
  Moreover, 
  there is an isomorphism between the graded Hochschild complex
  and ${\mathcal E}$, equipped with the differential $D$.
\end{lemma}

\begin{proof}
  This is straightforward. 
\end{proof}

\subsection{Computing in small models}

Fix $m$ and $k$, and abbreviate $\Clg(m,k)$ by $C$
and $\QuotPong(m,k)$ by $Q$.

Consider the algebra
\[ {\mathcal S}=\bigoplus_{\x,\y}\prod_{\{\vec{w}\in(\OneHalf\Z)^m\}}
(\Idemp{\x}\cdot C[t]\cdot \Idemp{\y})_{\vec{w}}\otimes
(\Idemp{\x}\cdot Q\cdot\Idemp{\y})_{\vec{w}}.\]
In words, ${\mathcal S}$ is generated sequences (indexed by weight vectors)
of elements 
$t^s\alpha\otimes b$ where $\alpha$ is a pure algebra element in $C$,
$b$ is a pure algebra element in $Q$,
\begin{itemize}
\item $\weight_i(\alpha)+s=\weight_i(b)$ for $i=1,\dots,m$.
\item The left idempotents of $\alpha$ and $b$ agree.
\item The right idempotents of $\alpha$ and $b$ agree.
\end{itemize}

The algebra ${\mathcal S}={\mathcal S}^{n,d}$ inherits a bigrading 
where
each pure element $t^s \cdot a\otimes b\in {\mathcal S}$ has
\[ n = 2\Totweight(b)-\Cr(b)\qquad{\text{and}}\qquad
  d=1-2 \Totweight+\Cr(b)+2s(m-k-1).\]
(As we shall see below, in Proposition~\ref{prop:SmallModels},
${\mathcal S}$ is a small model for the bigraded
Hochschild complex; compare also~\cite[Definition~5.41]{TorusAlg}.)

Turn the $DD$ bimodule $\lsup{C}\DDmod^{Q}$ into a left bimodule
over $C\otimes Q^{\op}\cong C\otimes Q$.
The differential $\delta^1$ corresponds to the projection onto
${\mathcal S}$ of the element of $C\otimes Q$
given by
\begin{align*}
  S=
&\left(\sum_{1\leq i<j\leq m-1}
f(\Rij)\otimes\Rij
+ f(\Lji) \otimes\Lji\right) \\
      & +
\left(\sum_{0\leq i<j\leq m}
f(\Xij)\otimes \Xij\right).
\end{align*}
(Compare~\eqref{eq:DDmodDef}. Note that under the identification
$Q^{\op}\cong Q$, the elements $\Rij^\op$ is identified with $\Lji$.)

Let
$\Dsmall\colon {\mathcal S}\to {\mathcal S}$
be the endomorphism defined by
\[ \Dsmall(x)=(\Id\otimes \partial_Q)(x) + S\cdot x + x \cdot S.\]
Note that, like the Hochschild differential, $\Dsmall$ respects the bigradings in the sense that
$\Dsmall\colon {\mathcal S}^{n,d}\to {\mathcal S}^{n+1,d-1}$.

\begin{prop}
  \label{prop:SmallModels}
  There is a bigraded quasi-isomorphism
  ${\mathcal S}^{*,*}\simeq \HC^{*,*}$.
\end{prop}

\begin{proof}
Let $\Phi\colon \Cobar(C)\to Q$
be the algebra quasi-isomorphism from Theorem~\ref{thm:KoszulDual}.
(See Lemma~\ref{lem:DefPhi}.)
There is an induced algebra quasi-isomorphism
\[ {\widetilde \Phi}\colon {\mathcal E}\to {\mathcal S}.\]

Since $\Phi_i=0$ for all $i>1$ (Lemma~\ref{lem:DefPhi}), it follows that
$\Phi$, and hence ${\widetilde \Phi}$ is a homomorphism of DG algebras.
Write $S=c\otimes q$ for 
$c\in \Clg$ and $q\in Q$.  We have that
\[
  {\widetilde \Phi}(E)
   = \sum_e e \otimes \Phi(e^*) = \sum_e  e\otimes \langle e^*,c\rangle \otimes q 
   = c\otimes q =  S.\]
   It follows at once that ${\widetilde \Phi}\circ D = \Dsmall\circ {\widetilde \Phi}$; i.e. ${\widetilde \Phi}$ induces a bigraded quasi-isomorphism
   from $\HC^{*,*}$ (in view of Lemma~\ref{lem:HochschildDifferential})
   to ${\mathcal S}^{*,*}$
\end{proof}

\begin{lemma}
  \label{lem:HomologyE}
  For the bigraded chain complex ${\mathcal E}={\mathcal E}^{*,*}$, we
  have that $H({\mathcal E}^{n,-1})=0$, except when $n=2$ or
  $n=2m-2k$. Moreover, if $H(\Idemp{\x}\cdot {\mathcal E}^{2m-2k,-1}\cdot
  \Idemp{\y})\neq 0$, then $\x=\y$; and indeed
  \begin{equation}
  H(\Idemp{\x}\cdot {\mathcal E}^{2m-2k,-1}\cdot\Idemp{\x})=
  \begin{cases}
    \label{eq:ComputeEComplex}
    \Field & {\text{if $k>1$}} \\
    \Field\oplus \Field & {\text{if $k=1$.}}
  \end{cases}
  \end{equation}
  Also, if $H({\mathcal E}^{n,-2})\neq 0$, then
  $n\in\{3,2m-2k+1\}$.
\end{lemma}

\begin{proof}
  Let $\vec{w}$ be a weight vector, and choose $s=\min w_i$. Let
  $w'_i=w_i-s$. The only non-zero element of $C[t]$ with weight
  $\vec{w}$ is $t^s a$, where $a\in C$ has weight $\vec{w'}$.

  Assume first that $k>1$.
  By Theorem~\ref{thm:HomologyQp}, if $b\in H(\Idemp{\x}\cdot Q\cdot
  \Idemp{\y})$ is a homologically non-trivial element with weight
  $\vec{w}$, then $b=\Omega^s\cdot b'$, where $b'\in H(\Idemp{\x}\cdot
  Q\cdot\Idemp{\y})$ or, in the special case where $\vec{w}$ is
  constant, there is another possibility of
  $b=\Omega^{s-1}\cdot b'$, where $b'$ is the homologically
  non-trivial element with $\cross(b')=k+1$ and $\weight_i(b')=1$.  We
  call the latter type of element {\em special}, and the other types
  {\em generic}.

  Thus, a special element has the form $(t^s,\Omega^{s-1}\cdot b')$ with $s\geq 1$. The bigrading
  of a special element is given by
  \[ (n,d)=(2ms-k(2s-1)-1,
  2-k-2s)\]
  When $d=-2$, we have $k=2$ and $s=1$, so $n=2m-3=2m-2k+1$.

  Next,  
  we compute the bigrading $(n,d)$ of the element $(a t^s,b)$.
  In the generic case,
  \[ d=1-(2\Totweight(b)-\cross(b))+2(m-k-1)s 
    = -(2\Totweight(b')-\cross(b'))-2s+1 \]
    
    So if $d=-1$ (in the generic case), then $2\Totweight(b')-\cross(b')=2-2s$.
    Since for any algebra element $b'$, 
    $2\Totweight(b')-\cross(b')\geq 0$
    (Proposition~\ref{prop:AtomicElements}), we have that $s=0$ or $1$. If $s=0$, then $b=b'$, $n=2$, $d=-1$.
    If $s=1$, then $2\Totweight(b)-\cross(b)=0$,
    so $b'$ is an idempotent
    (again, by Proposition~\ref{prop:AtomicElements}), $b=\Idemp{\x}\cdot \Omega$ (and, in particular, $\x=\y$),
    and $n=2m-2k$.
    
    If $d=-2$, then $2\Totweight(b')-\cross(b')=3-2s$ and $s=0$ or $1$. If
    $s=1$, then $1=2\Totweight(b')-\cross(b')$, and 
    $n=2m-2k+1$. If $s=0$, then $n=3$.

    When $k=1$, Theorem~\ref{thm:HomologyQp} once again proves that
    $b=\Omega^s\cdot b'$, where $b'\in H(\Idemp{\x}\cdot Q\cdot
    \Idemp{\y})$, provided that $\vec{w}$ is a non-constant
    vector. Thus, when ${\vec{w}}$ is non-constant the above
    computations establish all of the claims of the lemma, except
    Equation~\eqref{eq:ComputeEComplex}. 

    When $k=1$,
    Theorem~\ref{thm:HomologyQp} implies that for each $s\geq 1$,
    $H(\Idemp{\x}\cdot Q(\x,\x;(s,\dots,s))\cdot \Idemp{\x})$ is
    two-dimensional, spanned by $\Omega_+^s$ and $\Omega_-^s$, where
    $\Omega_+=\Xgen{x,m}\cdot \Xgen{0}{x}$ and 
    $\Omega_-=\Xgen{0}{x}\cdot \Xgen{x}{m}$ are the two components of $\Omega$.
    The bigrading of $(t^s,\Omega_+^s)$ is given by
    \[ (n,d)=(2s(m-1),1-2s(m-1)+2s(m-2))=(2s(m-1), 1-2s).\]
    Again, $d=-1$ implies that $s=1$, verifying
    Equation~\eqref{eq:ComputeEComplex}; the two homology generators
    being $(t,\Omega_+)$ and $(t,\Omega_-)$.
\end{proof}

\begin{defn}
  Let $\x=\{x_i\}_{i=1}^k$ and $\y=\{y_i\}_{i=1}^k$ be two idempotent states.
  We say that $\x\leq \y$ if for all $t\in\{1,\dots,k\}$, $x_t \leq y_t$.
  We say that $\y$ is an {\em immediate successor} of $\x$ if 
  there is some $s\in\{1,\dots,k\}$ so that $x_t=y_t$ for all $t\neq s$ and 
  $y_s=x_s+1$; i.e. $\Idemp{\y}\cdot L_s\cdot \Idemp{\x}\neq 0$.
  We also say that $\x$ is an {\em immediate predecessor} of $\y$.
\end{defn}

\begin{lemma}
  \label{lem:ClosedUnderSuccession}
  Let $S$ be a non-empty set of idempotent states.  Suppose that if
  $\x\in S$, then every immediate successor and every immediate
  predecessor of $\x$ is in $S$, then $S$ is the set of all idempotent
  states.
\end{lemma}

\begin{proof}
  The fact that $S$ is closed under immediate successors implies that
  for all $\x\in S$, if $\x\leq \y$, then $\y\in S$. The fact that it
  is closed under immediate predecessors implies that if $\x\in S$ and
  $\y\leq \x$, then $\y\in S$. Since there is a unique maximal and
  unique minimal idempotent state, the result follows.
\end{proof}

\begin{lemma}
  \label{lem:ComputeComplex}
  The homology $\sHH^{*,*}$ of ${\mathcal S}^{*,*}$ has
  $\sHH^{n,-1}=0$ unless $n=2$ or $n=2m-2k$;
  $\sHH^{n,-2}=0$ unless $n=3$ or $n=2m-2k+1$. 
  Moreover, $\sHH^{2m-2k,-1}\cong \Field$.
\end{lemma}

\begin{proof}
  Consider the chain complex ${\mathcal S}$, filtered by the weight of $b$ in 
  $(a,b)$. The homology of the associated graded object is
  $H({\mathcal E})$. The vanishing statements in the lemma now follow.
  
  It remains to compute $\sHH^{2m-2k,-1}$. We start with the case $k>1$.
  To this end, we 
  compute the next differential in the spectral sequence
  associated to the filtration; i.e. the map
  \[ d_1\colon H({\mathcal E}^{2m-2k,-1})\to H({\mathcal E})^{2m-2k-1,-2}\]
  induced by $\Dsmall$.
  By Lemma~\ref{lem:HomologyE}, $H({\mathcal E}^{2m-2k,-1})$ has rank
  $\left(\begin{array}{c} m\\k \end{array}\right)$,
  with generators of the form $(t\Idemp{\x},\Omega\cdot \Idemp{\x})$, uniquely determined by the idempotent states $\x$. 

The map $d_1$ is specified by 
  \begin{align*}
    d_1(\Idemp{\x}\cdot t,\Omega\cdot \Idemp{\x})&= 
    (L\cdot \Idemp{\x}\cdot t ,L\cdot \Omega \cdot \Idemp{\x})+ 
    (\Idemp{\x}\cdot L \cdot \Idemp{\x}, \Omega\cdot \Idemp{\x}\cdot L) \\
    &+ 
    (R\cdot \Idemp{\x}\cdot t ,R\cdot \Omega \cdot \Idemp{\x})+ 
  (\Idemp{\x}\cdot R \cdot \Idemp{\x}, \Omega\cdot \Idemp{\x}\cdot R)
\end{align*}
  Clearly, 
  \[ (t,\Omega)=\sum_{\x}(\x\cdot t, \Omega\cdot \Idemp{\x})\]
  is in the kernel of $d_1$. Moreover, if there is a non-empty $S$
  so that $\sum_{\x\in S}(t,\Omega\cdot \Idemp{\x})$ in the kernel of $d_1$,
  then the form of $d_1$ implies that $S$ is closed under immediate
  successors and predecessors. From
  Lemma~\ref{lem:ClosedUnderSuccession}, it follows that $S$ consists
  of all idempotent states; i.e. $(t,\Omega)$ is the only element of
  the kernel of $d_1$.  The fact that $(t,\Omega)$ is not in the image
  of $D$ is an easy consequence of the form od $\Dsmall$, and the fact
  that $\Omega$ maximizes $\cross$ among all elements with weight
  vector identically $1$. 
  It follows readily that $\sHH^{2m-2k,-1}\cong \Field$, when $k>1$.

  For the case $k=1$, we have the following explicit description of
  ${\mathcal S}^{2m-2,-1}$ and ${\mathcal S}^{2m-2,-2}$, and the differential.  As in
  the proof of Lemma~\ref{lem:HomologyE}, the space ${\mathcal S}^{2m-2,-1}$
  has dimension $2m-2$, with basis given by $\{(t,{\mathcal X}_{0,i}
  \cdot {\mathcal X}_{i,m})\}_{i=1,m-1}$ and $\{(t,{\mathcal
    X}_{i,m}\cdot {\mathcal X}_{0,i}) \}_{i=1,m-1}$.  We abbreviate
  these generators $\{\Omega^-_i\}_{i=1}^{m-1}$ and
  $\{\Omega^+_i\}_{i=1}^{m-1}$.  (Note that Lemma~\ref{lem:HomologyE}
  states that $H(\Idemp{\x}\cdot {\mathcal
    E}^{2m-2,-1}\cdot\Idemp{\x})$ is two-dimensional; but the proof
  in fact shows that $\Idemp{\x}\cdot {\mathcal E}^{2m-2,-1}\cdot
  \Idemp{\x}$ is two-dimensional.)

  This follows as in the proof of Lemma~\ref{lem:HomologyE}: the
  generators in bigrading $(2m-1,-2)$ have the form $(t b', \Omega
  \cdot b')$, where $2\Totweight(b')-\cross(b)=1$.  Explicitly, the
  space ${\mathcal S}^{2m-1,-2}$ has dimension $2m-2$, with basis given by
  \begin{align*}
    \{(t L_{i}, L_{i}\cdot {\mathcal X}_{0,i-1} {\mathcal X}_{i-1,m}) \}_{i=2}^{m-1}
    &\qquad \bigcup\qquad
    \{(t R_{i}, R_{i}\cdot {\mathcal X}_{i+1,m} {\mathcal X}_{0,i+1}) \}_{i=1}^{m-2} \\
    \bigcup\{(t U_1, {\mathcal X}_{0,1}\cdot {\mathcal X}_{1,m}\cdot {\mathcal X}_{0,1}),&~(t U_m, {\mathcal X}_{m-1,m}\cdot {\mathcal X}_{0,m-1}\cdot {\mathcal X}_{m-1,m})\}.
  \end{align*}

We abbreviate basis vectors from these three sets
$\{L_i\}_{i=2}^{m-1}$, $\{R_i\}_{i=1}^{m-2}$, and $\{U_1,U_m\}$ respectively.
  The differential has the form
  \begin{align*}
    \Dsmall \Omega^-_i &= \begin{cases}
      U_1 + L_2 & {\text{for $i=1$}} \\
    L_{i}+ L_{i+1} & {\text{for $1<i<m-1$}} \\
    L_{m-1}+ U_m &{\text{for $i=m-1$}}
    \end{cases} \\
    \Dsmall \Omega^+_i &= \begin{cases} U_1+R_2 &{\text{for $i=1$}} \\
      R_i+R_{i+1}& {\text{for $1<i<m-1$}} \\
        R_{m-1}+ U_m &\text{for $i=m-1$.}
      \end{cases}
  \end{align*}
  It follows at once, that the kernel of $\Dsmall$ (when $k=1$)
  is one-dimensional, spanned by
  \[ \Omega=\sum_{i=1}^{m-1} \Omega^-_i + \Omega^+_i.\]
    \end{proof}

\begin{remark}
  It follows from the above proof that 
  the cycle $(t,\Omega)\in {\mathcal S}^{2m-2k,-1}$, in fact generates $\sHH^{2m-2k,-1}$.
\end{remark}

\begin{proof}[Of Theorem~\ref{thm:CharacterizeActions}]
  Any Hochschild cochain in
  $\HC^{2m-2k,-1}(\Clg)$ determines an $A_{2m-2k}$-algebra structure on
  $\Clg[t]$.  By Proposition~\ref{prop:SmallModels} and
  Lemma~\ref{lem:ComputeComplex}.  since $\HH^{n,-2}=0$ for $n\geq
  2m-2k+2$; so by Proposition~\ref{prop:DeformationTheory}, this
  $A_{2m-2k}$-algebra structure can be extended to an $A_{\infty}$ structure.
  Moreover, since $\HH^{n+1,-1}=0$ for all $n\geq 2m-2k$, this
  $A_\infty$ structure is unique up to isomorphism.

  We claim that there is a unique non-trivial possibility for the
  operation $\mu_{2m-2k}$ giving $\Clg[t]$ the structure of an
  $A_{2m-2k}$-algebra.  Such an operation is necessarily a cocycle in
  $\HC^{2m-2k,-1}$.
  Since $\HH^{2m-2k,-1}\cong\Field$ by Proposition~\ref{prop:SmallModels}
  and Lemma~\ref{lem:ComputeComplex}, we conclude that any two such
  operations $\mu_{2m-2k}$ and $\mu'_{2m-2k}$ are cohomologous. But
  $\HC^{2m-2k-1,0}=0$; which, in turn, is an easy consequence of
  Equation~\eqref{eq:Grphi} and the fact the grading of any
  homogeneous element of $C[t]$ is a multiple of $2m-2k-2$.
  Thus, $\mu_{2m-2k}=\mu_{2m-2k}'$.
\end{proof}

\section{The ${\mathcal A}_{\infty}$ structure on the homology of the
  Pong algbra}
\label{sec:ComputeAinfty}

The aim of this section is to determine the homology of
$\Pong{m}{m-k-1}$ with its $A_{\infty}$ structure, as follows:

\begin{thm}
  \label{thm:HomologyPongAinfP}
  For integers $0<k<m-1$, 
  the homology of $\Pong{m}{m-k-1}$ with its induced $A_\infty$
  structure is quasi-isomorphic to the $A_{\infty}$ algebra from
  Theorem~\ref{thm:CharacterizeActions}.
\end{thm}

\begin{proof}
  By Theorem~\ref{thm:HomologyPong},
  \[ H_*(\Pong{m}{m-k-1})\cong H_*(\Clg(m,k)[\Omega]),\]
  where $\gr(\Omega)=2m-2k-2$. 
  Thus, 
  by Theorem~\ref{thm:CharacterizeActions}, it suffices to compute
  a single non-zero $\mu_{2m-2k}$ action.
  To this end, consider  the algebra sequence
  \[ (\Idemp{\x}\cdot a_1,\dots,a_{2m-2k})=
  (\Idemp{\x}\cdot v_1\cdots v_k, L_{k+1},\dots,L_{m-1},v_m,R_{m-1},\dots,R_{k+1})\]
  with starting idempotent $\x=\{k+1,\dots,m-1\}$. Note that in $\Clg(m,k)$, this corresponds to the idempotent $\{1,\dots,k\}$ and the
  corresponding algebra sequence, 
  \[ 
  (v_1\cdots v_k, R_{k+1},\dots,R_{m-1},v_m,L_{m-1},\dots,L_{k+1}).\]
  Each $a_i$ is a cycle; and indeed, each $a_i$ represents a
  homologically non-trivial element.  We claim that
  \[ \mu_{2m-2k}(a_1,\dots,a_{2m-2k})=\Omega\cdot \Idemp{\x}.\]

  We consider first the case where $k=1$, after introducing some
  notation. Let $a_i*\dots*a_j$ be an element that satisfies
  \[d(a_i*\dots*a_j)=\sum_{i<k<j} (a_i*\dots* a_k)\cdot (a_{k+1}*\dots * a_j).\]
  
  It is straightforward to check the following identities (in the relevant idempotents):
  \begin{align*}
    v_1*L_2*\dots*L_i &= \Lgenx{i}{1} X_{0,1} \\
    L_{i+1} *\dots * L_{m-1} v_m R_{m-1}*\dots *R_2 &= \Rgenx{1}{i} X_{i,m} \\
    v_m*R_{m-1}*\dots * R_2 &= \Rgenx{1}{m-1} X_{m-1,m} \\
    v_1*L_2*\dots* L_{m-1}* v_m&= X_{m-1,m} \Lgenx{m-1}{1} X_{0,1} \\
    v_1*L_2\dots *L_{m-1}* v_m* R_{m-1}*\dots * R_i &=
    X_{i-1,m} \Lgenx{i-1}{1} X_{0,1} \\
    R_{i-1}*\dots * R_2 &= \Rgenx{1}{i-1}.
  \end{align*}
  (See Example~\ref{ex:PongMu6}.)
  
  It follows readily that
  \[ \mu_{2m-2}(\Idemp{\x}\cdot v_1,L_2,\dots,L_{m-1},v_m,R_{m-1},\dots,R_2)=\Omega\cdot \Idemp{\x}.\]
  The general case follows by substituting $v_1\cdots v_k$ for $v_1$,
  and shifting indices.
\end{proof}

\begin{example}
  \label{ex:PongMu6}
  Consider $\Pong{4}{2}$. We have shown the chains used in the
  computation of $\mu_6(v_1,L_2,L_3,v_4,R_3,R_2)$ in
  Figure~\ref{fig:PongMu6}. Observe  that
  \begin{align*}
    \mu_6&(v_1,L_2,L_3,v_4,R_3,R_2) \\&=
    (v_1* L_2)\cdot (L_3* v_4*R_3*R_2) +
    (v_1*L_2* L_3) \cdot (v_4*R_3*R_2) \\
    &\qquad + (v_1* L_2 * L_3* v_4)\cdot (R_3*R_2) 
    +(v_1* L_2 * L_3* v_4*R_3)\cdot R_2 \\
    &=\Omega.
  \end{align*}
\end{example}

\begin{figure}[ht]
\input{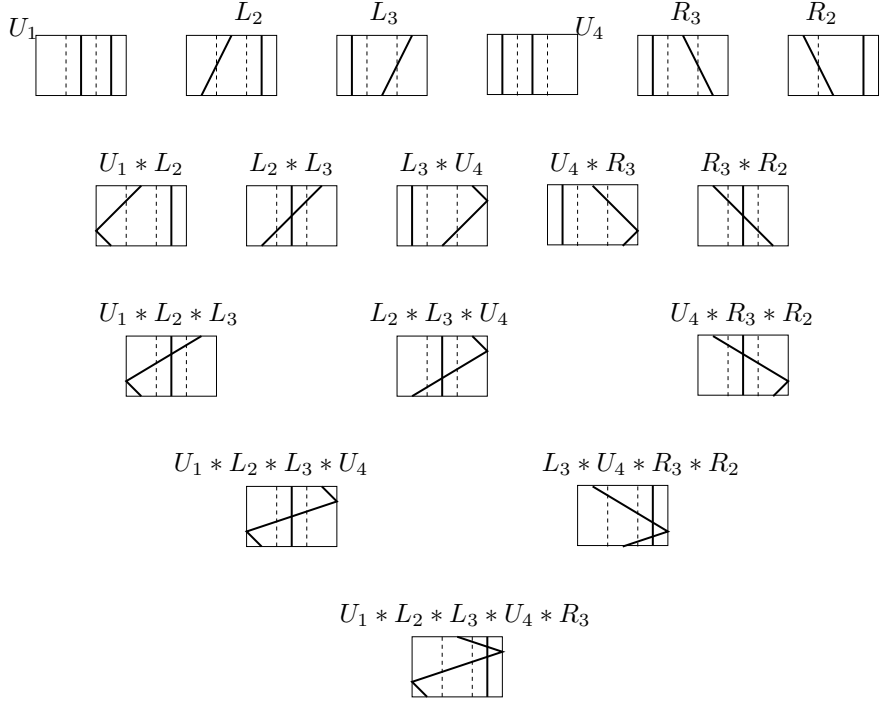}
\caption{\label{fig:PongMu6} {\bf{Computing $\mu_6$.}}
    In $H_*(\Pong{4}{2})$, we have shown steps in the verification that
    $\mu_6(v_1,L_2,L_3,v_4,R_3,R_2)=\Omega$.
\label{fig:ComputeM6}}
\end{figure}

Together, Theorems~\ref{thm:HomologyPong} and~\ref{thm:HomologyPongAinfP}
give Theorem~\ref{thm:HomologyPongAinf}.

\subsection{Degenerate cases}
\label{subsec:Extremes}

The reader may wonder to what extent Theorem~\ref{thm:HomologyPongAinf}
holds when $k=0$ or $k = m-1$.

When $k=m-1$, we have that
          \begin{align*}
          \Clg(m,m-1)&\cong \Field[v_1,\dots,v_m]/(v_1\cdots v_m)
          \\ \Pong{m}{0}=H(\Pong{m}{0})&\cong \Field[v_1,\dots,v_m]
          \end{align*}
Thus,
$H_*(\Pong{m}{0})\cong \Clg(m,m-1)[t]$,
provided we introduce to $\Clg(m,m-1)[t]$  new $\mu_2$ operations, such as
$\mu_2(v_1\cdots v_i,v_{i+1}\cdots v_m)=t$ for all $1<i<m$.

When $k=0$ we have $\Clg(m,0)\cong\Field$, while $H_*(\Pong{m}{m-1})\cong
\Field[\Omega]$.
Thus, Theorem~\ref{thm:HomologyPongAinf} holds on
additively (i.e. Theorem~\ref{thm:HomologyPong} holds), but there
cannot be a non-trivial $\mu_{2m}$-action.

Note that Theorem~\ref{thm:KoszulDuality} applies when $k=m-1$.
When $k=0$, the theorem is also true, but it is a tautology: in that case,
$\Quot{m}{0}\cong\Field$, $\Clg(m,k)=0$, so $\Cobar(\Clg(m,k))\cong\Field$.

\bibliographystyle{plain}
\bibliography{biblio}

\begin{thebibliography}{10}

\bibitem{AbouzaidSeidel}
M.~Abouzaid and P.~Seidel.
\newblock An open string analogue of {V}iterbo functoriality.
\newblock {\em Geom. Topol.}, 14(2):627--718, 2010.

\bibitem{Auroux}
D.~Auroux.
\newblock Fukaya categories and bordered {H}eegaard-{F}loer homology.
\newblock In {\em Proceedings of the {I}nternational {C}ongress of
  {M}athematicians. {V}olume {II}}, pages 917--941. Hindustan Book Agency, New
  Delhi, 2010.

\bibitem{BernsteinGelfandGelfand}
I.~N. Bernstein, I.~M. Gelfand, and S.~I. Gelfand.
\newblock Schubert cells, and the cohomology of the spaces {$G/P$}.
\newblock {\em Uspehi Mat. Nauk}, 28(3(171)):3--26, 1973.

\bibitem{DouglasManolescu}
C.~L. Douglas and C.~Manolescu.
\newblock On the algebra of cornered {F}loer homology.
\newblock {\em J. Topol.}, 7(1):1--68, 2014.

\bibitem{Keller}
B.~Keller.
\newblock {$A$}-infinity algebras, modules and functor categories.
\newblock In {\em Trends in representation theory of algebras and related
  topics}, volume 406 of {\em Contemp. Math.}, pages 67--93. Amer. Math. Soc.,
  Providence, RI, 2006.

\bibitem{LipshitzCyl}
R.~Lipshitz.
\newblock A cylindrical reformulation of {H}eegaard {F}loer homology.
\newblock {\em Geom. Topol.}, 10:955--1097 (electronic), 2006.

\bibitem{HomPairing}
R.~Lipshitz, P.~S. Ozsv\'{a}th, and D.~P. Thurston.
\newblock Heegaard {F}loer homology as morphism spaces.
\newblock {\em Quantum Topol.}, 2(4):381--449, 2011.

\bibitem{Bimodules}
R.~Lipshitz, P.~S. Ozsv\'ath, and D.~P. Thurston.
\newblock Bimodules in bordered {H}eegaard {F}loer homology.
\newblock {\em Geom. Topol.}, 19(2):525--724, 2015.

\bibitem{InvPair}
R.~Lipshitz, P.~S. Ozsvath, and D.~P. Thurston.
\newblock Bordered {H}eegaard {F}loer homology.
\newblock {\em Mem. Amer. Math. Soc.}, 254(1216):viii+279, 2018.

\bibitem{TorusAlg}
R.~Lipshitz, P.~S. Ozsv\'ath, and D.~P. Thurston.
\newblock A bordered {$HF^-$} algebra for the torus.
\newblock arxiv.org/abs/2108.12488, 2021.

\bibitem{TorusMod}
R.~Lipshitz, P.~S. Ozsv\'ath, and D.~P. Thurston.
\newblock Bordered {$HF^-$} with torus boundary.
\newblock In preparation, 2022.

\bibitem{ManionRouquier}
A.~Manion and R.~Rouquier.
\newblock Higher representations and cornered heegaard floer homology.
\newblock arxiv.org/2009.09627, 2020.

\bibitem{Knots}
P.~Ozsv{\'a}th and Z.~Szab{\'o}.
\newblock Holomorphic disks and knot invariants.
\newblock {\em Adv. Math.}, 186(1):58--116, 2004.

\bibitem{BorderedKnots}
P.~Ozsv{\'a}th and Z.~Szab{\'o}.
\newblock Kauffman states, bordered algebras, and a bigraded knot invariant.
\newblock {\em Adv. Math.}, 328:1088--1198, 2018.

\bibitem{GridBook}
P.~S. Ozsv{\'a}th, A.~I. Stipsicz, and Z.~Szab{\'o}.
\newblock {\em Grid homology for knots and links}, volume 208 of {\em
  Mathematical Surveys and Monographs}.
\newblock American Mathematical Society, Providence, RI, 2015.

\bibitem{Tilings}
P.~S. Ozsv{\'a}th and Z.~Szabo.
\newblock Planar graphs and deformations of bordered knot algebras.
\newblock In preparation.

\bibitem{HolKnot}
P.~S. Ozsv{\'a}th and Z.~Szabo.
\newblock Algebras with matchings and knot {F}loer homology.
\newblock arxiv.org/abs/1912.01657, 2019.

\bibitem{Bordered2}
P.~S. Ozsv\'{a}th and Z.~Szab\'{o}.
\newblock Bordered knot algebras with matchings.
\newblock {\em Quantum Topol.}, 10(3):481--592, 2019.

\bibitem{WrapPong}
P.~S. Ozsv{\'a}th and Z.~Szabo.
\newblock The pong algebra and the wrapped {F}ukaya category.
\newblock Preprint, 2022.

\bibitem{PetkovaVertesi}
I.~Petkova and V.~V\'{e}rtesi.
\newblock Combinatorial tangle {F}loer homology.
\newblock {\em Geom. Topol.}, 20(6):3219--3332, 2016.

\bibitem{RasmussenThesis}
J.~A. Rasmussen.
\newblock {\em Floer homology and knot complements}.
\newblock PhD thesis, Harvard University, 2003.

\end{thebibliography}

\end{document}